\documentclass[letterpaper,11pt,reqno]{amsart} 
\usepackage[portrait,margin=3.3cm]{geometry}
\usepackage{booktabs,multirow}

\usepackage{systeme}
\usepackage{mathrsfs,xfrac} 
\usepackage[unicode=true,pdfusetitle,
   bookmarks=true, 
   bookmarksnumbered=false, 
   bookmarksopen=false,
   breaklinks=false, 
   pdfborder={0 0 1}, 
   backref=false, colorlinks=true,linkcolor=blue,citecolor=blue,urlcolor=blue]{hyperref}

\hypersetup{pdfauthor={Dey and Wu}}

\usepackage{amsmath,amssymb,amsthm,amsfonts,amsbsy,latexsym,dsfont,color} 
\usepackage[numeric,initials,nobysame]{amsrefs} 
\usepackage[foot]{amsaddr}

\usepackage[utf8]{inputenc}
\usepackage[english]{babel}
\usepackage{comment}

\usepackage[textsize=tiny]{todonotes}
\usepackage{regexpatch}
\makeatletter
\xpatchcmd{\@todo}{\setkeys{todonotes}{#1}}{\setkeys{todonotes}{inline,#1}}{}{}
\makeatother

\usepackage{enumerate}
\newenvironment{enumeratei}{\begin{enumerate}[\upshape i.]}{\end{enumerate}}
\newenvironment{enumeratea}{\begin{enumerate}[\upshape a)]}{\end{enumerate}}

\newtheorem{thm}{Theorem}[section]
\newtheorem{lem}[thm]{Lemma}
\newtheorem{cor}[thm]{Corollary}
\newtheorem{prop}[thm]{Proposition}

\newtheorem{rem}[thm]{Remark}

\newtheorem{ass}[thm]{Assumption}

\renewcommand{\le}{\leqslant} 
\renewcommand{\ge}{\geqslant}

\newcommand{\ra}{\rangle}
\newcommand{\la}{\langle}

\newcommand{\ind}{\mathds{1}}

\newcommand{\norm}[1]{\left\Vert#1\right\Vert}
\newcommand{\abs}[1]{\left\vert#1\right\vert}

\newcommand{\ie}{\emph{i.e.,}}

\def\qed{ \hfill $\blacksquare$}  
\let\ga=\alpha \let\gb=\beta \let\gc=\gamma \let\gd=\delta 
     \let\gl=\lambda       \let\gn=\nu \let\go=\omega   \let\gs=\sigma \let\gt=\tau 
  
\let\gC=\Gamma \let\gD=\Delta  \let\gL=\Lambda 
         \let\gS=\Sigma  
                                
\newcommand{\cC}{\mathcal{C}}

\newcommand{\cN}{\mathcal{N}}\newcommand{\cO}{\mathcal{O}}
\newcommand{\cR}{\mathcal{R}}
\newcommand{\cS}{\mathcal{S}}
\newcommand{\cX}{\mathcal{X}}
  
\newcommand{\mvI}{\boldsymbol{I}}

\newcommand{\mvW}{\boldsymbol{W}}
\newcommand{\mvZ}{\boldsymbol{Z}}

\newcommand{\mvi}{\boldsymbol{i}}\newcommand{\mvj}{\boldsymbol{j}}

\newcommand{\mvx}{\boldsymbol{x}}

\newcommand{\mvgs}{\boldsymbol{\sigma}}\newcommand{\mvgt}{\boldsymbol{\tau}}

\newcommand{\mvtheta}{\boldsymbol{\theta}}

\newcommand{\bN}{\mathbb{N}}

\newcommand{\dN}{\mathds{N}}

\newcommand{\dR}{\mathds{R}}

\newcommand{\sC}{\mathscr{C}}

\newcommand{\sE}{\mathscr{E}}

\newcommand{\sI}{\mathscr{I}}

\newcommand{\sS}{\mathscr{S}}

\DeclareMathOperator{\E}{\mathds{E}}
\DeclareMathOperator{\pr}{\mathds{P}}

\DeclareMathOperator{\var}{Var}

\DeclareMathOperator{\argmax}{argmax}

\DeclareMathOperator{\tr}{Tr} 
  
\DeclareMathOperator{\diag}{diag}

\DeclareMathOperator{\N}{N}

\usepackage{subcaption}

\newcommand{\hh}{\hat{h}}
\newcommand{\quar}{\ensuremath{\sfrac14}}
\newcommand{\half}{\ensuremath{\sfrac12}}

\begin{document}

\title[Replica symmetric $p$-spin glass model]{Hypergraph Counting and Mixed $p$-Spin Glass Models under Replica Symmetry}
\author[Dey]{Partha S.~Dey$^\star$}
\author[Wu]{Qiang Wu$^\dagger$}
\address{$^\star$Department of Mathematics, University of Illinois at Urbana-Champaign, 1409 W Green Street, Urbana, Illinois 61801.
\newline 
\indent $^\dagger$School of Mathematics, University of Minnesota, 127 Vincent Hall 206 Church St. SE Minneapolis, MN 55455.}
\email{$^\star$psdey@illinois.edu,$^\dagger$wuq@umn.edu}
\date{\today}

\begin{abstract}
	Fluctuation problems in mean-field spin glass theory enjoy a rich history.
In this paper, we address this question in the high-temperature regime by studying the fluctuation of the partition function in the general mixed $p$-spin glass models under the weak external field assumption: $h= \rho N^{-\ga}, \rho>0, \ga \in [\quar,\infty]$. By extending the cluster expansion approach to this generic setting, we convert the fluctuation problem as a hypergraph counting problem and thus obtain a new {\it multiple-transition} phenomenon in terms of the fluctuation of the partition function. A byproduct of our results is a new critical inverse temperature obtained from optimal second-moment control of the partition function. All our fluctuation results hold up to the critical inverse temperature. 
Combining with multivariate Stein's method for exchangeable pairs, we also obtain an explicit convergence rate under appropriate moment assumptions on the general symmetric disorder coupling. Our results have several further implications.

-- First, our approach works for both even and odd pure $p$-spin models. The leading cluster structures in the odd $p$ case are different and more involved than in the even $p$ case. This gives a combinatorial explanation for the folklore that odd $p$-spin is usually more complicated than even $p$. 

-- Second, in the mixed $p$-spin setting, the cluster structures differ depending on the relation between the minimum effective even and odd $p$-spins: $p_e$ and $p_o$, respectively. As an example, at $h=0$, there are three sub-regimes: $p_e<p_o, p_o<p_e<2p_o, p_e\ge 2p_o$, wherein the first and third ones, the mixed $p$-spin model behaves essentially like a pure $p$-spin model, and only in the second regime, it is more like a mixture. This presents another criterion for classifying mean-field spin glass models compared to the work of 
Auffinger and Ben Arous (Ann.~Probab.~41 (2013), no.~6, 4214--4247),
where the idea is based on complexity computations for spherical models. 

-- Third, our framework naturally implies a {\it multi-scale fluctuation} phenomenon conjectured in the work of Bovier and Schertzer (Probab.~Theory~Relat.~Fields (2024)),
at $h=0$ case for pure $p$-spin model. Our results suggest that this phenomenon holds for general mixed $p$-spin models under weak external fields, including $h=0$. We also extend the approach in this paper to general multi-species mixed $p$-spin models.
\end{abstract}

\maketitle
\setcounter{tocdepth}1
\tableofcontents


\section{Introduction and Main Results}\label{sec:intro}
Cluster expansion, a powerful tool rooted in mathematical physics, has been used to rigorously study many problems in a wide range of statistical mechanics models. For example, cluster expansion was used in developing the Pigorov-Sinai theory to prove phase transitions for the Ising model. The rough idea is to decompose the configuration space into clusters of simple structures and then deal with the simple clusters instead to understand the log-partition function. We invite interested readers to~\cite{FV18} for more details on this classical treatment.

However, there are scarce results on cluster expansion to derive rigorous solutions in the disordered spin system, such as spin glass models. In the seminal paper~\cite{ALR87}, Aizenman, Lebowitz, and Ruelle first utilized the idea of cluster expansion in the celebrated Sherrington-Kirkpatrick (SK) spin glass model and obtained a rigorous solution of the model by proving a Gaussian central limit theorem of free energy. The results hold at high temperatures in the zero external field. The clusters they obtained are a collection of finite-sized simple loops in a complete graph. Similar results exist for the diluted SK model~\cite{Ko06} and Hopfield model~\cites{CKT04, ST92}. However, all these results are essentially restricted to the two-spin and zero external field cases. It was believed that the cluster expansion would not work in the presence of an external field~\cites{Tal11a, Tal11b}. Recently, in~\cite{DW21}, the authors made some progress by showing that the cluster expansion still works under weak external fields. A new cluster structure, a collection of simple paths with finite length, emerges instead. Based on this, they established a one-step transition behavior with a critical exponent of $1/4$ for the fluctuation of free energy under weak external fields.

In this paper, we extend this technique to a more general setting, the mixed $p$-spin glass models, which have highly complicated physical behavior compared to the simple SK model. In the mixed $p$-spin setting, the spin interactions are present not only among two sites but on any finite $p\ge3$ sites and their mixtures. To carry out the cluster expansion technique to the $p$-spin setting has the following main challenges. First, in order for the later counting procedure to be feasible, one needs to use second-moment control for the partition function to exclude the infinite-size cluster structures. In the SK case, a Gaussian trick was used to approximate the second moment due to the quadratic-type interactions. However, this is not possible for the general $p$-spin case. Besides that, In the SK case, the second-moment analysis immediately reduces the cluster structure to a finite size. Nonetheless, in the $p$-spin setting, the usual second-moment control can only reduce the hypergraph cluster size to $\log(N)$. To bypass this, we carefully analyze the structures with size at most $\log(N)$ and show that under the correct scaling, contributions of large-size structures tend to zero in probability. Second, it is not even clear what should be the geometry of cluster structures. In the SK case, finite size ``loop" and ``path" are competitively contributing to the partition function. However, there is no proper notion of ``loop" and ``path" on general hypergraphs. Another fundamental question would be how many different possible cluster structures are contributing. Third, analyzing the mixed $p$-spin model can fundamentally differ from the pure $p$-spin case, where $p$ is fixed. SK model can be treated as the pure $2$-spin model. In the mixed $p$-spin setting, where $p$ is not fixed, this makes the challenges addressed more significant. Later one can see that the results for mixed $p$-spin and pure $p$-spin models are indeed quite different.

Despite the above challenges, we are able to give a complete characterization of the dominant cluster structures for all pure and mixed $p$-spin models under different external fields. As a result, we obtain new fluctuation results for the partition function in general mixed $p$-spin glass models and new {\it multiple-transition} behavior under weak external fields. This naturally induces the fluctuation results for the free energy. Besides that, our results also have the following interesting implications. First, for fixed $p\ge3$, the fluctuation behaviors, accordingly the cluster structure geometry, are quite different for even and odd $p$. In many classical results of $p$-spin glass models, it has been a mystery why the odd and even $p$-spin model behave quite differently. For example, many important results~\cites{MTal06,CHL18,Chen14,Chen14b,CDP17,GJ21} can only be proven in even $p$ case rigorously. Our results give a combinatorial explanation of this folklore. Second, a byproduct of our analysis is a simple characterization of a new critical inverse temperature $\gb_f$ for general spin glass models.
The threshold is derived from an optimal second-moment estimate of the partition function. It is worth noting that this new threshold has a very simple form. To our knowledge, it has never been found in the literature. Some numerical simulation has been given in~Figure~\ref{fig:betac}, suggesting that this new threshold is neither the usual static phase transition point nor the dynamical phase transition point. It is of great interest to investigate the physical nature of $\gb_f$; we left this for future work. Third, a multi-scale fluctuation phenomenon was originally conjectured in~\cite{BS22} for the pure $p$-spin model with no external field. From our analysis, it can be immediately seen that this phenomenon holds for the general mixed $p$-spin model with external fields.
Finally, for mixed $p$-spin models, in terms of fluctuation and cluster geometry, there are three new regimes w.r.t $p$. Two behave more like a pure $p$-spin model, while the other behaves more like a mixture model. This can be treated as a criterion for the classification of spin glass models.

We also emphasize that our idea is robust enough to extend to other variants, such as the multi-species mixed $p$-spin models. Due to the fundamental different results for the pure $p$-spin model and general mixed $p$-spin model, we give the main results in two different sections. Before presenting the main results, let us first introduce the definition of $p$-spin glass models.

\subsection{Setting}
We start with the pure $p$-spin glass model. Consider the $N$ dimensional hypercube $\gS_N:= \{-1,+1 \}^N$. For a configuration $\mvgs:=(\gs_1,\gs_2,\ldots,\gs_N)\in\gS_N$, the pure $p$-spin \emph{Hamiltonian} is given by
\begin{align}\label{def:pspin}
	H_{N,p}(\mvgs) := \frac{1}{N^{(p-1)/2}}\sum_{\mvi\in\sE_{N,p}}J_{\mvi}\cdot \prod_{j=1}^p\gs_{i_j},
\end{align}
where the random variables $(J_{\mvi} )_{\mvi\in \sE_{N,p}}$, indexed by the set of $p$-tuples $\sE_{N,p}:=\{\mvi=(i_1,i_2,\ldots, i_p)\mid 1\le i_1<i_2< \cdots < i_p\le N\})$, are i.i.d.~symmetrically distributed with variance one. For convenience, we define $\gs_{\mvi}:=\prod_{j=1}^p\gs_{i_j}$. 

One can also consider a general mixed $p$-spin model, where the Hamiltonian is a mixture of $H_{N,p}$ over $p\ge 2$.

\begin{ass}[\bfseries Mixed $p$-spin Hamiltonian]\label{ass:mph}
	The mixed $p$-spin Hamiltonian is given by
	\begin{align}\label{def:mixedp}
		H_{N}(\mvgs) := \sum_{p\ge 2} \theta_p \cdot H_{N,p}(\mvgs),
	\end{align}
	where $\mvtheta= (\theta_{p})_{p\ge2}$ satisfies $\sum_{p\ge 2}\theta_{p}^2/(p-2)!<\infty$.
\end{ass}

It is easy to calculate the covariance of the Hamiltonian, for $\mvgs, \mvgt \in \gS_N$, as
\begin{align*}
	\E H_N(\mvgs) H_N(\mvgt) = N\cdot \xi_{N}(R_{\mvgs,\mvgt}),
	\text{ where } R_{\mvgs, \mvgt}:= \frac{1}{N} \sum_{i} \gs_i \gt_i
\end{align*}
and the function $\xi_{N}$\footnote{One can explicitly write down $\xi_N(x)$ as $\sum_{p\ge 2} \frac{\theta_p^2}{p!}\cdot (-N)^{-p} B_p(-Nx, -1!\cdot N, -2!\cdot Nx, \ldots, -(p-1)!\cdot Nx^{\ind\{p\text{ odd}\}})$ where $B_p$ is the $p^{\text{th}}$ complete exponential Bell polynomial~\cite{andr84}. It is easy to check that $\norm{\xi_N-\xi}_{\infty}\le \sum_{p\ge 2} \frac{\theta_p^2}{p!} (N^p-(N)_p)/N^p.$ Moreover, it is easy to check that for any $p\ge 2$, $(-\sqrt{N})^{-p} B_p(-\sqrt{N}x, -1!\cdot N, -2!\sqrt{N}\cdot x, \ldots)$ converges to the Hermite polynomial $H_{p}(x)=B_{p}(x,-1,0,\ldots)$ as $N\to\infty$.}
satisfies $\norm{\xi_{N}-\xi}_{\infty}\le \xi''(1)/2N$ where
\begin{align*}
	\xi(x): = \sum_{p\ge 2} \frac{\theta_p^2}{p!}\cdot x^p \text{ for } x\in[-1,1]
\end{align*}
and $\xi''(1)=\sum_{p\ge 2}\theta_{p}^2/(p-2)!.$
The function $\xi_{N}$ is known as the structure function and can encode mean-field spin glass models. The values $\theta_{k}=\ind_{\{k=p\}}$ corresponds to the pure $p$-spin model.

For all these models, adding an external magnetic field $h\sum_{i=1}^N \gs_i$ could create drastic changes to the system. Using the cluster expansion technique, we obtain a detailed transition picture of those changes in the mixed $p$-spin and other generalized models.

We recall the standard definitions first. The partition function at inverse temperature $\gb>0$ and external field $h\ge 0$ is defined as
\begin{align*}
	Z_{N}(\gb,h) :=  \sum_{\mvgs \in \gS_N} e^{\gb H_N(\mvgs)+h\sum_{i=1}^{N}\gs_i}
\end{align*}
and the free energy as
\begin{align*}
	F_{N}(\gb,h) := \frac{1}{N} \log Z_{N}(\gb,h).
\end{align*}
For notational convenience, we may drop the dependence on terms such as $N,\gb,h$ unless it causes any confusion. It is easy to prove the following decomposition of the partition function
\begin{align}\label{eq:decom}
	Z_{N}(\gb,h) =(2\cosh h)^{N}\cdot \bar{Z}_N(\gb)\cdot  \hat{Z}_N(\gb,h)
\end{align}
where
\begin{align}
	\bar{Z}_N(\gb)   & := \prod_{p\ge 2}\prod_{\mvi\in\sE_{N,p}}\cosh(\gb \theta_{p}\cdot  J_{\mvi}/N^{(p-1)/2})
	\text{ and }\label{def:ZNbar} \\
	\hat{Z}_N(\gb,h) & :=\E_{h} \prod_{p\ge 2}\prod_{\mvi\in\sE_{N,p}} (1+\gs_{\mvi}\tanh(\gb \theta_{p}\cdot  J_{\mvi}/N^{(p-1)/2})).\label{def:ZNhat}
\end{align}
Here the expectation $\E_{h}$ is w.r.t.~$(\gs_i)_{1\le i\le N}$ being i.i.d.~with $\pr_{h}(\gs_{1}=\pm1)=e^{\pm h}/(e^{h}+e^{-h})$ and mean $\hh:=\tanh h$. Note that,
\begin{align*}
	\frac1N\log \bar{Z}_N(\gb) = \frac1N \sum_{p\ge 2}\sum_{\mvi\in\sE_{N,p}}\log\cosh(\gb \theta_{p}\cdot  J_{\mvi}/N^{(p-1)/2}) \to \frac{1}2\gb^2\sum_{p\ge 2}\theta_{p}^2/p!=\frac{1}2\gb^2\xi(1)
\end{align*}
as $N\to\infty$ almost surely and in expectation; and
\begin{align*}
	N^{p_m-2}\cdot \var(\log \bar{Z}_N(\gb)) \to \frac{\gb^4\theta_{p_m}^4}{4p_m!}\var(J^2)
\end{align*}
where $p_m:=\min\{p\ge 2\mid \theta_{p}\neq0\}$.
We will mainly focus on analyzing the behavior of  $\hat{Z}_N(\gb,h)$ as $N\to\infty$.

\begin{ass}[\bfseries Moment assumptions on Disorder]\label{ass:moment}
	We will assume that the random variables $(J_{\mvi})_{ \mvi\in\cup_{p\ge 2}\sE_{N,p}}$ are i.i.d.~symmetrically distributed with mean zero, variance one and finite fourth moment.
\end{ass}

\begin{ass}[\bfseries Weak external field]\label{ass:wef}
	In the rest of the article, we will assume the external field is of the form
	\begin{align*}
		h=h_{N}=\rho N^{-\ga},
	\end{align*}
	for some $\rho>0$, $\ga \in [\quar,\infty]$. In particular, we have $hN^{\quar}\le \rho$ for all $N\ge 1$.
\end{ass}
Note that $\ga=\infty$ corresponds to the zero external field case. In this paper, we focus on the fluctuation and distributional behavior of the free energy for the mixed $p$-spin model under a weak or no external field condition. Our approach extensively generalizes the cluster expansion and Stein's method framework in~\cite{DW21} for SK model. This approach does not extend to the regime $\ga \in [0,\quar)$ due to the same diverging second-moment issue as in~\cite{DW21}. We will briefly discuss the details later. We state the main results as follows.

\subsection{Main results}
As mentioned before, a byproduct of our results is a new critical threshold. Before presenting its precise form, let us recall the existing critical thresholds in spin glass theory. For SK model, it is known that the static phase transition threshold is $\gb=1$. The model exhibits the so-called replica symmetric phase for $\gb<1$, and the replica symmetry breaking phase for $\gb>1$. Replica symmetry (breaking) is defined via the asymptotic distribution of the inner product of two configurations sampled from the SK Gibbs measure; if it is concentrated at some fixed point, then the model is called replica symmetry, symmetry breaking otherwise. Note that, in the seminal paper of Aizenman et.al~\cite{ALR87}, the threshold for the second moment of the partition function is also $\gb=1$. On the other hand, there is another threshold known as the dynamical phase transition point, which separates the fast and slow mixing regions for the associated Glauber dynamics. For the SK model, the dynamic threshold is also believed to be $\gb=1$~\cite{AJ24}. 

SK model is special in the sense that all thresholds coincide with each other. However, it is more involved in the general $p$-spin case. For pure $p$-spin models with $p\ge 3$, to the best of our knowledge, there is no explicit characterization of the static phase transition threshold yet. Some existing  literature~\cite{Chen19} has implicit form obtained via Parisi formula~\cite{MTal06,Pan14}, but in general it is hard to compute directly. Besides that, the dynamic phase transition threshold is widely believed to be strictly smaller than the static phase transition point, see~\cite{AJ24}. In the current paper, we derived a new threshold $\gb_f$ from a sharp second-moment estimate of $\hat{Z}_N(\gb,h)$. Namely, for $\gb<\gb_f$, the second moment stays bounded and diverges for $\gb>\gb_f$. Numerical simulation given in Figure~\ref{fig:betac} suggests that this is neither the static nor the dynamic phase transition threshold.

\subsubsection{A new critical temperature}
 Our first main result is the following characterization of $\gb_f(\xi)$ for general spin glass models $\xi$, including mixed $p$-spin and multi-species mixed $p$-spin models. Let
\begin{align}\label{eq:beta_c}
	\gb_f=\gb_f(\xi):= \min_{x\in (0,1)} \sqrt{\frac{I(x)}{\xi(x)}},
\end{align}
where $\xi(x)$ is the structure-function used to define the spin glass model in~\eqref{def:mixedp} and
$$I(x):= \frac{1}2(1+x)\log(1+x) + \frac{1}2(1-x)\log(1-x)$$ is the binary entropy function.

Note that,
\begin{align*}
	\gb_f^2\cdot \xi(x) \le I(x)\text{ for }  x \in [-1,1].
\end{align*}
Subag~\cite{Sub21} obtained a similar, but different, characterization in the spherical setting for multi-species mixed $p$-spin glass models. We present a version for the multi-species SK model in Section~\ref{sec:multi-p}. Formally, we have the following result.
\begin{thm}\label{thm:beta_c}
	Assume~\ref{ass:mph},~\ref{ass:moment} and~\ref{ass:wef}. For $\gb<\gb_f(\xi)$, we have for all $N$ large
	\begin{align*}
		\E [(\hat{Z}_N(\gb,h))^2] \le \frac{C}{\sqrt{\gb_f^2-\gb^2}}\exp\left(\frac{\rho^4\gb^2}{\gb_f^2-\gb^2} \right)
	\end{align*}
	where $C>0$ is a universal constant. In particular, we have
	\begin{align}
		\frac1N \log Z_N(\gb) \to \frac{1}2\gb^2\xi(1) \text{ in probability as } N\to\infty.
	\end{align}
\end{thm}

\begin{rem}
	The above convergence result is based on the fact that $\hat{Z}_N(\gb) = 1+ o_{\pr}(1)$. The convergence in probability can be strengthened to almost sure convergence under stronger moment assumption than Assumption~\ref{ass:moment}.
\end{rem}

The proof of the above theorem and further discussions are deferred to Section~\ref{sec:exp} with the proof of the large hypergraphs decay regime, where the characterization~\eqref{eq:beta_c} was derived.
Recall that in the pure $p$-spin case, $\xi_p(x) = x^p/p!$ and we have the following result.
\begin{lem}\label{lem:betac-p}
	The function $\phi(x):=xI'(x)/I(x)$ for $x\in (0,1)$ and $\phi(0)=2$ is continuous and strictly increasing on $[0,1)$ with range $[2,\infty)$. Moreover, for $p\ge 2$ and $\xi_{p}(x)=x^{p}/p!$ we have
	\begin{align*}
		\gb_{f}(p)=\sqrt{I(x_{\star})/\xi_p(x_{\star})}, \text{ where } x_{\star}:=\phi^{-1}(p).
	\end{align*}
\end{lem}

\begin{rem}\label{rem:threshold}
	Talagrand~\cite{Tal00} and~\cite{BKL02} obtained a lower bound $\gb_p(p): = \min_{x \in [0,1]}\sqrt{(1+x^{-p})\cdot I(x)}$. As shown in Figure~\ref{fig:betac}, the difference between $\gb_f(p)$ and $\gb_p(p)$ brings up a fascinating phenomenon. The simulation was done under the following Hamiltonian:
 \[
 H_N(\mvgs) = \sqrt{\frac{p!}{2N^{p-1}}} \sum_{i_1<i_2<\cdots <i_p} J_{i_1,i_2,\ldots i_p} \gs_{i_1} \cdots \gs_{i_p}.
 \]
 Equivalently, $\xi(x) = x^p/2$. It is known that as $p\to\infty$, the pure $p$-spin model converges to the random energy model. In~\cite{BKL02}, two different thresholds $\sqrt{\ln 2}$ and $2\sqrt{\ln 2}$ were found, where the first one is the Gaussian and non-Gaussian fluctuation threshold, and the second one is the static phase transition thresholds separating the replica symmetry and symmetry breaking phase. It is interesting to notice that $\gb_f(p) \to \sqrt{2\ln 2}$ and $\gb_p(p) \to 2\sqrt{\ln 2}$  as $p\to \infty$. Besides, comparing with existing simulations of the static critical inverse temperature $\gb_c$~\cite{Chen19,PWB20}, neither $\gb_p(p)$ nor $\gb_f(p)$ agree with the simulation, but $\gb_p(p)$ gets closer to the simulation results for large $p$. Given the explicit form of $\beta_f$ and the fact that $\gb_f \sqrt{2\ln 2}$, it is very interesting to understand the physical meaning of this new threshold.
\end{rem}

\begin{figure}
    \centering
    \includegraphics[scale=0.55]{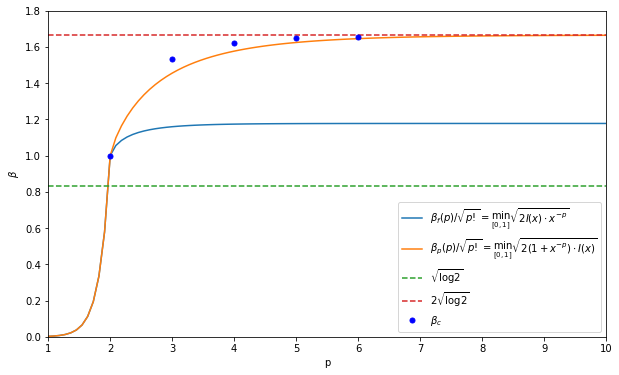}
    \caption{Under the structure function $\xi(x) = x^p/2$. Plot of various critical thresholds vs.~$p$. The yellow curve below is the threshold from~\cite{BKL02,Tal00}. The blue curve corresponds to our threshold $\gb_f$. The blue dots are some numerical simulations from~\cite{Chen14} for the static phase transition point $\beta_c$ for pure $p$-spin models. As $p\to\infty$, $\beta_{f}(\xi_{p})\to \sqrt{2\ln2}$, and $\gb_p(p)\to 2\sqrt{\ln 2}$. It is known that the pure $p$-spin model converges to the random energy model (REM) as $p\to \infty$. In~\cite{BKL02}, it was shown that  for REM, $2\sqrt{\ln 2}$ is the static phase transition point, and $\sqrt{\ln 2}$ is the Gaussian fluctuation threshold.}
    \label{fig:betac}
\end{figure}

Next, we present the fluctuation results up to $\gb_f$ for pure $p$-spin models. The results for mixed $p$-spin glass are quite different, thus deferred to the next Section. Similarly for the multi-species variants.

\subsubsection{Pure $p$-spin glass models at $h=0$}
At the log-partition function level, the fluctuation results at $h=0$ were known about twenty years ago in~\cites{BKL02, KL05}, which were proved at a high enough temperature (small enough $\gb$). The proof ideas were based on several different approaches, and  in particular, $\hat{Z}_N(\gb,0)$ has never been well understood (see~\cite{KL05} for more details). We first restate the results in~\cite{KL05} as follows.

\begin{thm}[cf.~\cite{KL05}*{Theorem 1.2}]\label{thm:lowe}
	Assume that $h=0$, $J_{\mvi}$'s are i.i.d.~symmetrically distributed with variance one and finite exponential moment. There exists $\tilde{\gb}$ depending on $p$ such that for $\gb<\tilde{\gb}$, we have
	\begin{align}
		N^{\frac{p-2}2} (\log Z_N(\gb,0) - \log \E Z_N (\gb,0)) \to \N\left(0,\frac {\gb^4} {4p!} \var(J^2)\right)
	\end{align}
	$\text{in distribution as } N \to \infty.$
\end{thm}
The above theorem is based on the decomposition of the partition function given in~\eqref{eq:decom} for $h=0$. The key point is that the variance order of $\log\hat{Z}_N$ is much smaller than $\log\bar{Z}_N$. Under the scaling $N^{(p-2)/2}$ of $\log\bar{Z}_N$, the fluctuation of $\log\hat{Z}_N$ is negligible.
However, it is not even clear what is the correct variance order of $\log\hat{Z}_N$ for $h=0$. The authors used concentration bounds to prove that the variance order of $\log\hat{Z}_N$ is dominated by $\log\bar{Z}_N$. Besides, the result holds at very high temperatures due to the lack of understanding about $\hat{Z}_N$. In the following part, we can see that $\hat{Z}_N$ contains fruitful information about the model. Once $\hat{Z}_N$ can be fully understood, we can prove the fluctuation results and extract the information about the critical inverse temperature $\gb_f(\xi)$. We state our first fluctuation result for $\hat{Z}_N$ up to $\gb_f$ in the zero external field case, \ie~$\ga=\infty$.

\begin{thm}[Pure $p$-spin at $h=0$]\label{thm:h0}
	Assume that $h=0$ and assumption~\ref{ass:moment} holds. For $\gb<\gb_f$, we have
	\begin{align}
		N^{\frac{1}4{\ell_p(p-2)}} \cdot  \log \hat{Z}_N(\gb,0) \to \N(0, v_{p}(\gb)^2)
		\text{ in distribution as } N \to \infty,
	\end{align}
	where
	$$
		\ell_p := \begin{cases}3 & \text{if }p\text{ is even,} \\
             4 & \text{if}  \ p \text{ is odd,}
		\end{cases}
	$$ and
	\begin{align}\label{eq:var-h0}
		v_{p}(\gb)^2:=
		\begin{cases}
			\frac{ \gb^{6} }{ 3!p!^{3/2}} \E H_3(H_{p}(\eta)/\sqrt{p!}) = \frac{ \gb^{6} }{ 3! p!^3} \E (H_{p}(\eta)^3), & \text{ for $p$ even}, \\
			\frac{ \gb^{8} }{ 4!p!^2} \E H_4(H_{p}(\eta)/\sqrt{p!})) = \frac{ \gb^{8} }{ 4!p!^4} \E (H_{p}(\eta)^4 - 3p!^2),  & \text{ for $p$ odd}.
		\end{cases}
	\end{align}
	where $H_{k}(\cdot)$ is the Hermite polynomial of degree $k$ and $\eta\sim\N(0,1)$.
\end{thm}

\begin{rem}\label{rem:bs22}
	We have the following remarks:
	\begin{enumeratea}
		\item This result is stronger than what was obtained in~\cite{KL05}, where the author only showed that the variance order of $\log\hat{Z}_N$ is smaller than $\log\bar{Z}_N$. Besides that, Theorem~\ref{thm:h0} holds up to $\gb_f$ with general weights having a finite fourth moment.
		\item Bovier and Schertzer~\cite{BS22} recently obtained similar results to Theorem~\ref{thm:h0}. Their proof is based on truncated moment estimates with a different decomposition of the partition function. They split $\log Z_N$ into two parts: $J_N:= \frac12 \sum_{\mvi \in \sE_{N,p}} \gb^2 J_{\mvi}^2/N^{p-1}$ and $\log Z_N - J_N$, and obtained similar limit theorem for $\log Z_N - J_N$.
	\end{enumeratea}
\end{rem}
The above result suggests that pure even and odd $p$-spin models behave differently at $h=0$. Later in the proof, it will be clear why there is such a discrepancy. This discrepancy persists until the weak external field becomes strong enough. This will be touched in Theorem~\ref{thm:h1}.

Let us briefly discuss the strategy of the proof for Theorem~\ref{thm:h0}. Starting with the decomposition in~\eqref{eq:decom}, we will focus on the term $\hat{Z}_N$. Applying $\E_{\mvgs}$, we have
\begin{align}\label{eq:zhat}
	\hat{Z}_N(\gb,0) & = \sum_{\gC \subseteq \sE_{N,p}} \hat\go(\gC),\\
 \text{ where }
	\hat\go(\gC) = \prod_{\mvi \in \gC} \hat\go_{\mvi}, \quad&
	\hat\go_{\mvi}= \tanh(\gb J_{\mvi}/N^{(p-1)/2}), \text{ for }\mvi\in\sE_{N,p},\notag
\end{align}
and $\gC$ is a sub-hypergraph of the simple complete $p$-uniform hypergraph. We further decompose  the sum depending on the cluster structure, \ie~the number of hyperedges used and the number of odd-degree vertices in $\gC$. The main idea is that the contribution of large cluster structures to $\hat{Z}_N$ decays fast in the regime $\gb<\gb_f$. It allows us to reduce the sum in $\hat{Z}_N$ to smaller-sized clusters. This will be proved and explained in Section~\ref{sec:exp}. In contrast to SK model~\cite{DW21}, the first reduction can only reduce the size of clusters to size $\cO(\log N)$, which is still not amenable to analysis using Stein's method and counting argument in the later steps. 

To get tractable finite-size dominant clusters, we use a form of chaos expansion. Built on this, we can make a further reduction to identify the leading clusters. It turns out that even and odd $p$-spin cases have different leading cluster structures. We will use the notation $((k,p),\ell)$ to denote the cluster structure with $k$ many $p$-hyperedges and $\ell$ many odd-degree vertices. We may also use $(k,\ell)$ notation instead of $((k,p),\ell)$  if there is no confusion. 

We define the sub-hypergraphs corresponding to cluster structures for even and odd $p$, respectively, as follows:
\begin{align}\label{str:pure-p-0}
	\begin{split}
		\cS_{(3,p),0} &:= \{ \gC \subseteq \sE_{N,p}: (\abs{\gC}, \abs{\partial \gC})= (3,0)\} \text{ and } \\
		\cS_{(4,p),0} &:= \{ \gC \subseteq \sE_{N,p}: (\abs{\gC}, \abs{\partial \gC})= (4,0)\},
	\end{split}
\end{align}
where $\abs{\gC}, \abs{\partial \gC}$ are, respectively, the number of hyperedges and odd-degree vertices in $\gC$. Example of those structures is given in Figure~\ref{fig:min-struct-odd} and~\ref{fig:min-struct-even}. Note that, $\cS_{(3,p),0}$ is non-empty only when $p$ is even. The results of determining the leading cluster structures are presented in Section~\ref{sec:dominant}.

\begin{rem}[On the multi-scale fluctuation phenomenon]\label{rem:multi-scale-pure}
Following the Remark~\ref{rem:bs22} part (b), Bovier and Schertzer~\cite{BS22} posed a conjecture that once the leading terms are subtracted from $\log Z_N - J_N$, on a smaller scale, there appears yet another limit theorem. This type of result is particularly clear in our framework. In our setting, $\log \hat{Z}_N$ plays the role of $\log Z_N - J_N$. Once the leading clusters' contribution in~\eqref{str:pure-p-0} are subtracted from $\log \hat{Z}_N$, the dominant contribution will be from larger size clusters (with smaller variance). Take even $p$-spin as an example; the next leading cluster after subtracting $(3,0)$ cluster will be cluster $(4,0)$. The variance order will be in a smaller scale $\Theta(N^{p-2})$. One can obtain a similar limit theorem under this scale, which can be repeated on a smaller scale. We further remark that this type of result can also be obtained if some weak external field is present, where different dominant cluster structures appear with different variance orders. This will be addressed in Theorem~\ref{thm:h1}. Similar things should happen in general mixed $p$-spin models, where the dominant clusters become quite complicated; details are presented in Section~\ref{sec:mixed-p}.
\end{rem}

To prove the central limit theorem, we use multivariate Stein's method for exchangeable pairs~\cite{RR09} to analyze the obtained leading clusters. This is in a similar spirit to~\cite{DW21} but is more involved than the SK case since the cluster structures are now hypergraphs. The complete details are presented in Section~\ref{sec:stein}. Note that an explicit convergence rate is also obtained while applying Stein's method. We further remark that from Figure~\ref{fig:min-struct-odd}, it is easy to see that odd $p$-spin cluster structures are more complicated, which in some sense explains why in many classical results, the odd $p$-spin case is usually more difficult to deal with. We defer the further discussion of using Stein's method to prove joint distributional convergence after Theorem~\ref{thm:h1} in Section~\ref{ssec:stein}.

\begin{figure}[htbp]
	\begin{subfigure}[b]{\columnwidth}
		\centering
		\includegraphics[width=3.5cm,page=7]{hypergraphs.pdf}
		\includegraphics[width=3.5cm]{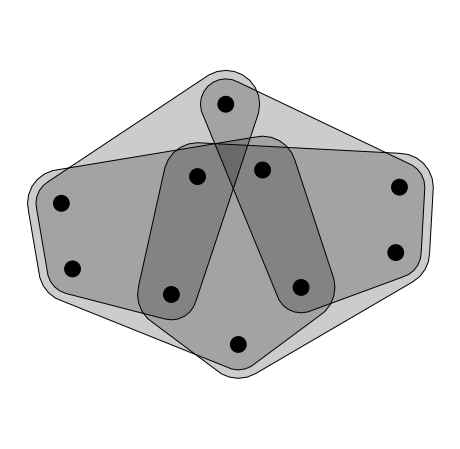}
		\includegraphics[width=3.5cm,page=9]{hypergraphs.pdf}
		\caption{Variants of $((4,5),0)$ cluster, \ie~the leading cluster for $p=5, h=0$.}
		\label{subfig:a}
	\end{subfigure}
	\begin{subfigure}[b]{\columnwidth}
		\centering
		\includegraphics[width=3.5cm,page=6]{hypergraphs.pdf}
		\includegraphics[width=3.5cm,page=5]{hypergraphs.pdf}
		\includegraphics[width=3.5cm,page=4]{hypergraphs.pdf}
		\caption{Left one corresponds to $((3,5),1)$ cluster, middle one $((2,5),2)$ cluster and right one $((1,5),5)$ cluster.}
		\label{subfig:b}
	\end{subfigure}
	\caption{Example of leading clusters for $p=5$ appearing when $\ga\ge \quar$. }
	\label{fig:min-struct-odd}
\end{figure}

\begin{figure}[htbp]
	\centering
	\includegraphics[width=3.5cm,page=3]{hypergraphs.pdf}
	\includegraphics[width=3.5cm,page=2]{hypergraphs.pdf}
	\includegraphics[width=3.5cm,page=1]{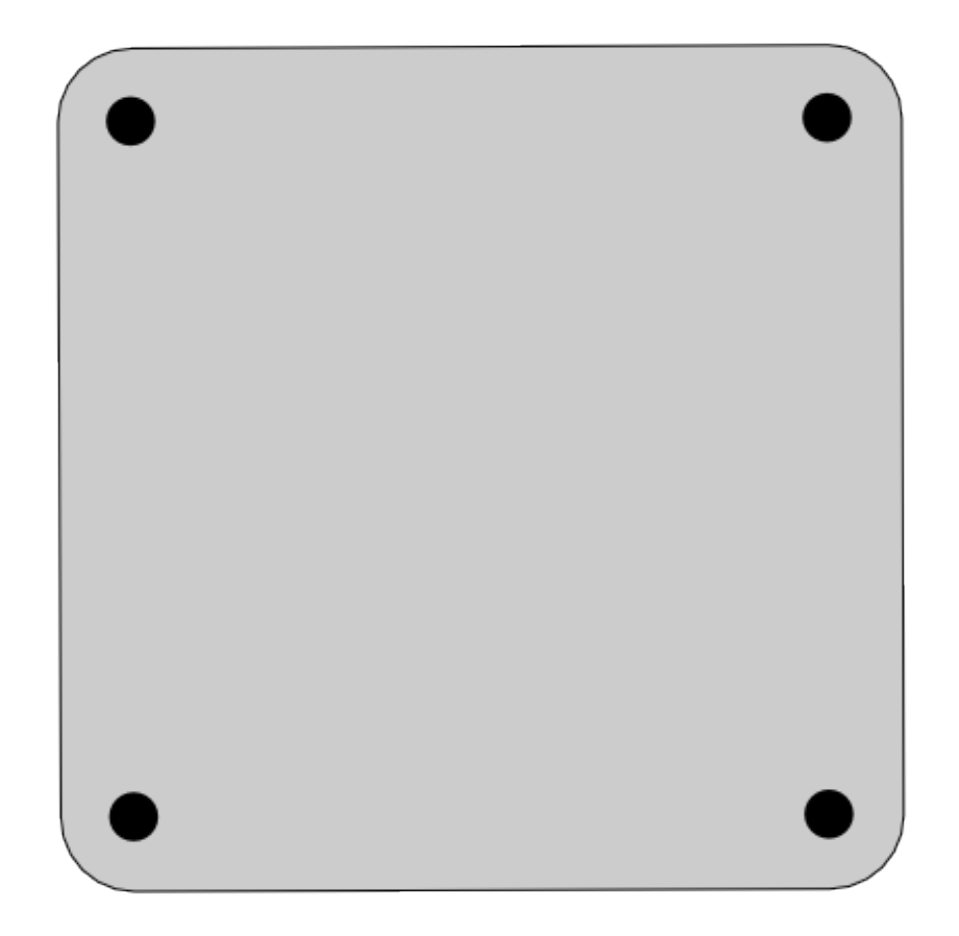}
	\caption{Example of leading clusters for $p=4$ appeared when $\ga \ge \quar$. The left one is the $((3,4),0)$ cluster when $\ga = \infty$. The middle one is  the $((2,4),2)$ cluster. The right one is the $((1,4),4)$ single hyperedge. }
	\label{fig:min-struct-even}
\end{figure}

\subsubsection{Pure $p$-spin glass models under weak external field}

Next, we state the results for the model under a weak external field. Compared to the results in SK model~\cite{DW21}, the picture becomes more complicated and exhibits a new multiple-transition phenomenon. Here we point out that we can only obtain the fluctuation results for $\ga\ge \quar$ due to a  fundamental difficulty similar to the SK case. Roughly speaking, we do not enjoy the large hypergraph decay results for $\ga< \quar$, \ie~when the external field is strong enough. The large hypergraph decay results under weak external field are presented in Section~\ref{sec:exp}. 

Let us first introduce some necessary notations. Depending on whether $p$ is even or odd, we define
\begin{align} \label{eq:alphac}
	\ga_c =\ga_c(p) :=
	\begin{cases}
		\quad {p}/{8}, & p\text{ even}, \\
		(p-1)/4,       & p \text{ odd}.
	\end{cases}
\end{align}
We also introduce the following three regimes:
\begin{align}\label{eq:regime}
	\begin{split}
		\cR_1  =\cR_1 (\ga_c)  :=[\ga_c,\infty),\quad
		\cR_2  =\cR_2 (\ga_c)  := [\half, \ga_c), \text{ and } \quad
		\cR_3  :=  [\quar, \half).
	\end{split}
\end{align}
The cases $\ga = \half$ and $\ga_c$ correspond to critical transition points, where some new leading structures emerge. For each of the regimes, we define the corresponding variance exponents
\begin{align}\label{eq:gamma}
	\gc_1 := \frac{1}2\ell_p(p-2), \quad
	\gc_2(\ga) :=  \gc_{1} - 4(\ga_c(p)-\ga)   \text{ and }
	\gc_3(\ga) := 2\ga p-1.
\end{align}

We state the main results when $\ga \in [\quar,\infty)$ for pure $p$-spin models. We use the following notation for representing the limiting variance
\begin{align}\label{eq:limit-var}
	v_{(k,p),\ell}(\gb,\rho)^2 := \frac{\gb^{2k} \rho^{2\ell}}{k! \cdot \ell! \cdot p!^{k/2}}  \E ( H_k(H_p(\eta)/\sqrt{p!})\cdot H_\ell(\eta)), \text{ where } \eta \sim \N(0,1).
\end{align}

 Note that this is the variance associated with the cluster structure $ ((k,p),\ell)$, which is summarized in Table~\ref{tab:trans}. The above limiting variance result is proved in Lemma~\ref{lem:varlim}. In the rest, we will drop the dependence on $p$ and write it as $v_{k,\ell}(\gb,\rho)^2$ for convenience. 

 In the $h=0$ case, it can be checked that the limiting variances in Theorem~\ref{thm:h0} are $v_{3,0}(\gb,0)^2$ and $v_{4,0}(\gb,0)^2$ for even and odd $p$ respectively. In this case, we will write them as $v_{3,0}(\gb)^2,v_{4,0}(\gb)^2$ for simplicity.

\begin{thm}[Pure $p$-spin under weak external field]\label{thm:h1}
	Assume that $\ga \ge \quar$, and assumption~\ref{ass:moment} holds. For $\gb<\gb_f$ with the notations in~\eqref{eq:gamma}, we have the following results.
	\begin{enumeratea}
		\item \textsc{\bfseries Sub-critical regime}: If $\ga \in \cR_1=[\ga_c, \infty)$, we have
		\begin{align}
			N^{\gc_1/2} \cdot \log \hat{Z}_N(\gb,h) \to \N(0,v_1^2)
		\end{align}
		in distribution as $N \to \infty$ where
		\begin{align*}
			v_1^2 :=\begin{cases}
				        v_p(\gb)^2,   & \ga>\ga_c, \\
				        v_p^c(\gb)^2, & \ga  = \ga_c,.
			        \end{cases}
		\end{align*}
		with $v_p(\gb)$ as defined in~\eqref{eq:var-h0} and
		\begin{align*}
			v_p^c(\gb)^2 := \begin{cases}
				                v_{3,0}(\gb)^2 + v_{2,2}(\gb,\rho)^2,                         & \text{$p$ even}, \\
				                v_{4,0}(\gb)^2 + v_{3,1}(\gb,\rho)^2 + v_{2,2}(\gb,\rho)^2, & \text{$p$ odd}.
			                \end{cases}
		\end{align*}
		\item \textsc{\bfseries Unified regime I}: If $\ga \in \cR_2=[\half,\ga_c)$, we have
		\begin{align}
			N^{\gc_2(\ga)/2} \cdot \log \hat{Z}_N \to \N(0,v_2^2)
		\end{align}
		in distribution as $N \to \infty$ where
		\begin{align*}
			v_2^2 :=\begin{cases}
				v_{2,2}(\gb,\rho)^2 + v_{1,p}(\gb,\rho)^2, & \ga  = \frac12, \\
				v_{2,2}(\gb,\rho)^2,                         & \frac12<\ga<\ga_c.
			\end{cases}
		\end{align*}
		\item \textsc{\bfseries Unified regime II}: If $\ga \in \cR_3=[\quar,\half)$, we have
		\begin{align}
			N^{\gc_3(\ga)/2} \cdot  \log \hat{Z}_N \to \N(0,v_3^2)
		\end{align}
		in distribution as $N \to \infty$ where $v_3^2 = v_{1,p}(\gb,\rho)^2$.
	\end{enumeratea}
\end{thm}

\begin{rem}
	The even and odd $p$-spin models behave consistently once the external field is in the unified regimes $\cR_2$ and $\cR_3$.
\end{rem}

The above theorem presents a new multiple transition phenomenon, accordingly summarized in Figure~\ref{fig:transition}. Note that the picture is more complicated than the SK ($p=2$) model, where only the regime $\cR_1$ appears, see~\cite{DW21} for further details. Besides that, $p=3,4$ are also special, we can check that $\ga_c = \half$ in these cases. It means that the regime $\cR_2$ disappears. All three regimes exist for $p \ge 5$, exhibiting multiple transitions. We summarize this discussion in Figure~\ref{fig:transition} and also in Table~\ref{tab:trans}.

\begin{figure}[htbp]
	\centering
	\includegraphics[height=3cm]{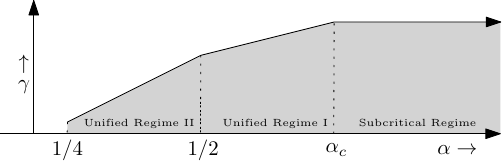}
	\caption{Multiple transitions in the pure $p$-spin models. The horizontal axis is the component of the external field $h$. The vertical axis is the exponent $\gc$ in $\var(\hat{Z}_{N})\approx N^{-\gc}$.}
	\label{fig:transition}
\end{figure}

The proof of Theorem~\ref{thm:h1} gets through in a similar manner as Theorem~\ref{thm:h0} once the following two key ingredients are established.
\begin{enumeratea}
	\item First, we need to prove that even if the external field, of the scale $N^{-\ga},\ga\ge \quar$, is present, the large hypergraph decay property still holds.
	\item Second, once we reduce the hypergraph clusters to a finite size, it needs to identify the dominant clusters depending on the strength of the external field. This needs careful analysis of the variance order of many possible cluster structures due to external fields. We rephrase it as a discrete optimization problem over the cluster size $\abs{\gC}$ and the number of odd-degree vertices $\abs{\partial \gC}$. The goal is to find the clusters with the largest variance. 
\end{enumeratea}

Let us first define those new dominant clusters that emerged after adding the weak external fields:
\begin{align}\label{str:pure-p-h}
	\begin{split}
		\cS_{(3,p),1}: = \{ \gC \subseteq \sE_{N,p}: (\abs{\gC}, \abs{\partial \gC})= (3,1)\}, \\ \cS_{(2,p),2}: = \{ \gC \subseteq \sE_{N,p}: (\abs{\gC}, \abs{\partial \gC})= (2,2)\},\\
		\cS_{(1,p),p}: = \{ \gC \subseteq \sE_{N,p}: (\abs{\gC}, \abs{\partial \gC})= (1,p)\}.
	\end{split}
\end{align}
There is still a discrepancy between even and odd $p$ in some regimes.

As addressed in Table~\ref{tab:trans}, for even $p$, in the sub-critical regime, the contributions from clusters $(2,2),(1,p)$ are dominated by the $(3,0)$ cluster. Thus, the fluctuation order and variance are essentially the same as zero external field cases. Only when $\ga = \ga_c$, the contributions of $(2,2)$ and $(3,0)$ are of the same order, this changes the asymptotic variance in the CLT. While in the unified regime I, the cluster $(2,2)$ dominates, as a result, the fluctuation order and variance both change. When the external field gets even stronger, \ie~in the unified regime II, the cluster $(1,p)$ dominates. Note that $\ga = 1/2$ is the transition point where the contribution of clusters $(2,2)$ and $(1,p)$ are in the same scale. For the odd $p$-spin case, the transition picture is essentially the same as the even $p$ case, except there is a new companion cluster $(3,1)$ for $(2,2)$. The example of cluster structures and the associated transition behavior can be found in Figure~\ref{fig:min-struct-odd} and Table~\ref{tab:trans}.

Since the behavior of $\hat{Z}_N$ is now well understood, it is not hard to deduce the transitional results for the log-partition function $\log Z_N$, because $\log \bar{Z}_N$ as an i.i.d.~sum is always easy to analyze. In particular, we have
\begin{align}
	\log \bar{Z}_N(\gb,h) & = \sum_{\mvi \in \sE_{N,p}} \log \cosh(\gb J_{\mvi}/N^{(p-1)/2}) \\
	& \approx \frac{\gb^2\binom{N}{p} }{2 N^{p-1}} - \frac{\gb^4  \binom{N}{p}}{12 N^{2p-2}}\cdot \E J^4 + \N\left(0,\frac{\gb^4  \binom{N}{p}}{4 N^{2p-2}}\cdot \var(J^2)\right). \label{eq:zbar-approx}
\end{align}
Here we use $\approx$ for distributional approximation after appropriate centering and scaling.
Due to the competitive behavior of $\hat{Z}_N$ and $\bar{Z}_N$, we obtained the following results at the free energy level. The threshold of the external field, in this case, changed. The variance order $\log Z_N$ in the supercritical regime aligns with $\log \hat{Z}_N$. While in the sub-critical case, $\bar{Z}_N$ dominates the contribution in $\log Z_N$.

\begin{cor}[Fluctuations for the log-partition function]\label{cor:logz}
	Assume that $\ga \ge \quar$, and $J_{\mvi},\mvi\in\sE_{N,p}$ satisfy Assumption~\ref{ass:moment}. Then for $\gb<\gb_f$, we have
	\begin{enumeratea}
		\item If $\ga > \frac{p-1}{2p} $, we have
		\begin{align}
			N^{(p-2)/2}\cdot \left(\log Z_N - \binom{N}{p}\E\log\cosh(\gb J/N^{(p-1)/2}) - N\log (2\cosh(h))\right) \to \N(0,v_1^2),
		\end{align}
		in distribution as $N\to\infty$, where $v_1^2 :=\gb^4\var(J^2)/4p! $.
		\item If $\ga=\frac{p-1}{2p} $, we have
		\begin{align}
			N^{(p-2)/2}\cdot \left(\log Z_N - \binom{N}{p}\E\log\cosh(\gb J/N^{(p-1)/2} -N\log (2\cosh(h))\right) \to \N(0,v_2^2),
		\end{align}
		in distribution as $N\to\infty$, where $v_2^2 := v_{1,p}(\gb,\rho)^2 + v_1^2. $
		\item If $\frac{1}4 \le \ga < \frac{p-1}{2p} $,we have
		\begin{align}
			N^{(2\ga p -1)/2}\cdot \left(\log Z_N - \binom{N}{p}\E\log\cosh(\gb J/N^{(p-1)/2} -N\log (2\cosh(h)) \right) \to \N(0,v_3^2)
		\end{align}
		in distribution as $N\to\infty$, where $v_3^2:=v_{1,p}(\gb,\rho)^2$.
	\end{enumeratea}
\end{cor}

\begin{rem}
	We remark on the following facts.
	\begin{enumeratea}
		\item    One can check that in the $p=2$ case, the critical transition threshold becomes $\frac14$, which was established in~\cite{DW21} for SK model. In the SK case, there is no third regime shown in part c).
		\item In general, for  $p\ge 3$, it can be checked that $\frac14 < \frac{p-1}{2p} < \frac12$. Clearly $\lim_{p\to\infty} \frac{p-1}{2p} = \frac12$. It was known~\cite{Derr80} that the random energy model (REM) is the limit of pure $p$-spin model as $p \to \infty$. Thus it is natural to conjecture that the critical threshold for the fluctuation in the REM with weak external field is $\ga_c = \frac12$ at high temperature. One can heuristically check this is indeed the case. In the REM case, the energy of each configuration is i.i.d.~distributed; it needs $\Theta(N^{1/2})$ external field to beat the i.i.d.~sum of energies by classical CLT.
	\end{enumeratea}
\end{rem}

Corollary~\ref{cor:logz} naturally follows from  Theorem~\ref{thm:h1}. Recall that in~\cite{KL05}, Kn\"opfel and L\"owe obtained results stated in Theorem~\ref{thm:lowe}, which basically says that at $h=0$ case, the fluctuation of free energy in pure $p$-spin model is dominated by $\log \bar{Z}_N$. Corollary~\ref{cor:logz} generalizes this result in two ways. One is by establishing the transitional behavior under a weak external field. The other is by extending the temperature range up to $\gb_f$.

In summary, the results for the pure $p$-spin case already exhibit complex transition phenomena, as illustrated in the above theorems. It is natural to ask what happens in the general mixed $p$-spin models. Following a similar idea, the main difficulty lies in identifying the leading cluster structures and controlling the large hypergraph decay. Unlike the pure $p$-spin model, now there is enormous freedom to form all kinds of clusters in the mixed $p$-spin case since $p$ is not fixed anymore. Even if the cluster size is fixed, there still could be infinitely many possible different structures. Due to this further complication of the mixed $p$-spin model, we defer the main results and discussions into Section~\ref{sec:mixed-p}. We further remark that our approach can be generalized to establish similar results for multi-species mixed $p$-spin models. Again, the main results and discussions will be postponed to Section~\ref{sec:multi-p}.

\begin{table}[htbp]
	\centering
	\caption{Summary of transition behavior for pure $p$-spin models}
	\begin{tabular}{@{} ccll @{}}
		\toprule
		External field $h$                      & Cond.    & $\gc : \var(\hat{Z}_N)\approx N^{-\gc}$ & Cluster structures $(\abs{\gC},\abs{\partial \gC}) $ \\
		\midrule
		\multirow2{*}{$\ga \in (\ga_c,\infty]$} & $p$ even & $3(p-2)/2$                              & $(3,0)$ \\
		                                        & $p$ odd  & $2(p-2)$                                & $ (4,0)$ \\
		\midrule
		\multirow2{*}{$\ga = \ga_c$}            & $p$ even & $3(p-2)/2$                              & $(3,0),(2,2)$ \\
		                                        & $p$ odd  & $2(p-2)$                                & $ (4,0),(2,2),(3,1)$ \\
		\midrule
		$\ga \in (1/2,\ga_c)$                   &          & $p+4\ga-3$                              & (2,2) \\
		\midrule
		$\ga =1/2$                              &          & $p-1$                                   & (2,2),(1,$p$) \\
		\midrule
		$\ga \in [1/4,1/2)$                     &          & $2\ga p-1$                              & (1,$p$) \\
		\bottomrule
	\end{tabular}
	\label{tab:trans}
\end{table}

\subsection{Prior results}\label{sec:prior}

This section briefly summarizes the previous fluctuation results in the $p$-spin models and the study of transitional behavior while changing the external field. The first fluctuation result for pure $p$-spin models can be traced back to 2002, in which Bovier, Kurkova, and L\"owe~\cite{BKL02} used martingale approach to prove a central limit theorem for the free energy up to a $\gb_p<\gb_c$.
The Gaussian disorder played a key role in their proof. Later Kn\"opfel and L\"owe~\cite{KL05} adapted the idea of Aizenman et.~al.~\cite{ALR87} to obtain a decomposition in~\eqref{eq:decom}, where the term $\log\bar{Z}_N$ is essentially a sum of i.i.d.~random variables, and they proved that the variance order of $\log\hat{Z}_N$ is smaller than that of $\log\bar{Z}_N$ using concentration results. Therefore, the fluctuation of limiting free energy is dominated by $\log\bar{Z}_N$.

The proof worked for general symmetric disorders with finite exponential moments. Again, this result also holds up to some $\tilde{\gb}_p<\gb_c$. In short, those two results were both established in the zero external field setting for pure $p$-spin models and a strictly sub-region of the whole high-temperature regime. More importantly, it was unclear what happens in the term $\hat{Z}_N$ in both works. On the other hand, in 2017, Chen, Dey, and Panchenko~\cite{CDP17} extended the fluctuation results to the mixed even $p$-spin models in the presence of an external field, and they proved that in this case, the Gaussian fluctuation of free energy persists for all $\gb<\infty$. The key idea is based on Stein's method and superconcentration results using the Parisi formula. However, one weakness is the even $p$ restriction, which seems a common hurdle in many other problems, see~\cites{Chen14, CHL18, MTal06, Chen14b}. Recently, Banerjee and Belius~\cite{BB21} used the weighted cycle counting technique to obtain the fluctuation results for mixed $p$-spin models at zero external fields with a non-vanishing 2-spin term. Note that this technique is different from the loop counting arising in cluster expansion. One notable weakness is that it works up to some implicitly high temperature and is also restricted to zero field cases.

Although there has been much progress in understanding the fluctuation of free energy at zero and positive external field, the transition behavior between them has long been elusive, especially in the Ising spin case. In the SK model ($p=2$) at a high-temperature regime, the authors~\cite{DW21} utilized various approaches to prove that in the high-temperature regime, there are three sub-regimes concerning the strength of the weak external field $h=\rho N^{-\ga}$ for some $\rho>0, \ga \in[0,\infty]$. In the sub-critical case, $\ga>\quar$, the fluctuation is nearly the same as $h=0$. In the critical case $\ga = \quar$, the fluctuation order does not change, but the asymptotic mean and variance now have some extra corrections. In the supercritical case, the fluctuation order becomes linear. One of their approaches is based on cluster expansion. They also established similar results in some other generalized models, such as the multi-species and diluted SK models. In the spherical SK model, Baik et.al.~\cite{BCLW21} studied the transition behavior using the steepest descent method originating in random matrix theory, and some results are computer-aided. The results in the high-temperature regime are consistent with~\cite{DW21} for the Ising spin case. A concurrent work by Landon and Sosoe~\cite{LS20} obtained similar results by conducting a fully rigorous analysis, and the idea is essentially the same. All those results are restricted to the $p=2$ spin case. In particular, extending the random matrix technique used in the spherical SK model to the spherical $p$-spin setting is challenging, and nothing is known. In summary, the transition behavior from zero to positive field for the general $p$-spin model is a complete mystery. Besides that, establishing results up to the critical inverse temperature is another challenge.

\subsection{Multivariate Stein's method for exchangeable pairs}\label{ssec:stein}

In the previous part, we briefly sketch how to decompose $\hat{Z}_N(\gb,h)$ into a collection of different clusters in different regimes. To establish the limit theorems described in Theorem~\ref{thm:h0},~\ref{thm:h1} and Corollary~\ref{cor:logz}, we apply the following version of multivariate Stein's method for exchangeable pairs. Stein's method is a modern way to prove normal approximation; we refer interested readers for the background to~\cite{CGS11}.

\begin{thm}[{\cite{RR09}*{Theorem 2.1}}]\label{thm:rr-mvstein}
	Assume that $(\mvW,\mvW')$ is an exchangeable pair of $\dR^d$ valued random vectors such that
	$
		\E \mvW = 0, \E \mvW\mvW^T = \gS
	$
	where $\gS \in \dR^{d \times d}$ is a symmetric positive definite matrix and $\mvW$ has finite third moment in each coordinate. Suppose further that
	\begin{align}\label{eq:stein-linear}
		\E(\mvW'-\mvW\mid \mvW) = -\gL \mvW\quad \text{a.s.},
	\end{align}
	for an invertible matrix $\gL$. If $\mvZ$ has $d$-dimensional standard normal distribution, we have for every three times differentiable function $f$,
	\begin{align*}
		\abs{\E f(\mvW) - \E f(\gS^{1/2}\mvZ)} \le \frac{1}4\abs{f}_2A + \frac{1}{12}\abs{f}_3B
	\end{align*}
	where $\abs{f}_2:=\sup_{i,j}\norm{\partial_{x_ix_j} f}_\infty, \abs{f}_3:=\sup_{i,j,k}\norm{\partial_{x_ix_jx_k} f}_\infty$, $\gl^{(i)}: = \sum_{m=1}^d \abs{(\gL^{-1})_{m,i}}$,
	\begin{align*}
		A & := \sum_{i,j=1}^d \gl^{(i)} \sqrt{\var \E((W_i'-W_i)(W_j'-W_j)\mid \mvW)} \\
		\text{ and }
		B & := \sum_{i,j,k=1}^d \gl^{(i)} \E \abs{(W_i'-W_i)(W_j'-W_j)(W_k'-W_k)}.
	\end{align*}
\end{thm}

In our setting, the random vector we will work on is a collection of random variables indexed by different cluster structures. For example, in the regime $\ga = \ga_c$ for pure even $p$-spin model,
\begin{align*}
	\mvW = \left(\sum_{\gC \subseteq S_{(3,p),0}} \go(\gC), \sum_{\gC \subseteq S_{(2,p),2}} \go(\gC)\right).
\end{align*}
Moreover, the linearity condition~\eqref{eq:stein-linear} is satisfied with $\gL$ of the form $\gl D$, where $\gl=o(1)$ is a scalar, and $D$ is a deterministic matrix.
The details for checking the associated conditions are left to Section~\ref{sec:stein}. We remark that Stein's method proof presents a more straightforward and systematic picture than the moment-based method and comes with a convergence rate under minimal moment assumptions.

\subsection{Notations}\label{ssec:notat}

\begin{enumeratei}
	\item $\gc$ denotes the exponent in  $\var(\hat{Z}_N)=\Theta( N^{-\gc})$.
	\item $\ga_c$ is the first critical transition exponent of the weak external field $h = \rho N^{-\ga}$.
	\item $p_e, p_o$ denote the minimum even and odd $p\ge 3$ interactions appearing the mixed $p$-spin model, and $p_m = \min\{p_o,p_e\}$.
	\item $H_{k}(\cdot)$ is the Hermite polynomial of degree $k$.
	\item $\hat{h} = \tanh(h)$ and $\tilde{h}^2 = \tanh^{-1}(\tanh(h)^2) \approx h^2$ for $h$ small.
    \item For a hyperedge $\mvi \in \sE_{N,p}$, we define
	$
		\hat{\go}_{\mvi} = \tanh(\gb \theta_p J_{\mvi}/N^{(p-1)/2})
	$
	and its normalized version
	$
		\go_{\mvi} = N^{(p-1)/2}\cdot  \hat{\go}_{\mvi}\approx \gb J_{\mvi}.
	$

    \item For a finite sub-hypergraph $\gC$, we define $\hat{\go}(\gC)=\prod_{\mvi\in \gC}\hat{\go}_{\mvi}$ and $\go(\gC) = \prod_{\mvi\in\gC} \go_{\mvi}$.

    \item $\cR_i$ for $i=1,2,3$ denotes the sub-critical regime: $ [\ga_c,\infty]$, unified regime I: $ [\half, \ga_c)$ and unified regime II: $ [\quar, \half)$ respectively. Note $\ga_c$ has different forms in pure and mixed $p$-spin models.
    
    \item We use $c:=((a_1,p_1),\ldots,(a_t,p_t),\ell)$ to denote the type of cluster structures. Specifically, $\cS_c$ denotes the set of sub-hypergraph corresponding to the cluster structure $c$, \ie~the set of hypergraphs with $a_s$ many $p_s$-hyperedges for $s=1,\ldots, t$ and $\ell$ many odd-degree vertices.  
	
    \item For mixed $p$-spin model, $\cS_{a_e,a_o,b}$ denote the set of hypergraphs with cluster structure $c = ((a_e,p_e),(a_o,p_o), b)$ with $a$ different $p_e$-hyperedges and $b$ different $p_o$-hyperedges and $c$ odd-degree vertices. For pure $p$-spin, we use $\cS_{a,b}$ to present the clusters with structure $c=((a,p),b)$.
    
	\item Following the last notation, we use $v_{a,b,p}(\gb,\rho)^2$ to denote the limiting variance for clusters in $\cS_{(a,p),b}$. Similarly for mixed $p$-spin, let $v_{a_e,a_o,b}(\gb,\rho)^2$ be the limiting variance for clusters in $\cS_{a_e,a_o,b}$.
	
    \item The weight for the cluster $c$ is $V_c:=\hat{h}^{\ell}\cdot\sum_{\gC\in \cS_c} \hat{\go}(\gC)$. The variance for the cluster $c$ is $\var(V_c)=\Theta(N^{\cX(c)})$, where $\cX(c)$ denotes the variance exponent. A similar notation is used in the random vector case for multivariate Stein's method computations.

    \item We define the random variable $W_c:= \sum_{\gC\in \cS_c} \go(\gC)$ for a given cluster structure $c$.  
\end{enumeratei}

\subsection{Organization of the paper}
This paper is structured as follows. The main results for pure $p$-spin models under zero and weak external field are stated in Section~\ref{sec:intro}. For clarity, we state the corresponding results for general mixed $p$-spin models in Section~\ref{sec:mixed-p}. Besides that, we also discuss and state the results for multi-species mixed $p$-spin models in Section~\ref{sec:multi-p}. The proof details are distributed into three different sections. In Section~\ref{sec:exp} the large hypergraph decay results for general mixed $p$-spin (including pure $p$ case) are presented, where we also include the proof of the simple characterization of critical inverse temperature for general spin glass models in the replica symmetric regime. We also include the large $p$-hyperedge truncation results there, which is a core challenge in the mixed $p$-spin models. In Section~\ref{sec:dominant}, we present the general strategy and the proof details for identifying the dominant cluster structures in various regimes for different models. In Section~\ref{sec:stein}, we show the details of using multivariate Stein's method for exchangeable pairs to analyze the dominant clusters obtained in Section~\ref{sec:dominant}. Finally, in Section~\ref{sec:main-pf}, we collect everything together to prove the main theorems present in Section~\ref{sec:intro} and~\ref{sec:mixed-p} for the pure $p$-spin and mixed $p$-spin models respectively. In Section~\ref{sec:stein-mix}, we devote ourselves to illustrating how to carry out the computations for applying multivariate Stein's method for mixed $p$-spin model. Finally, in Section~\ref{sec:open}, we present future questions and further discussions about the results of this paper.

\section{Results for Mixed \texorpdfstring{$p$}{p}-Spin Models}\label{sec:mixed-p}
In Section~\ref{sec:intro}, Theorem~\ref{thm:h0} and~\ref{thm:h1} gives the leading cluster structures and the fluctuation results for $\log \hat{Z}_N$ in the pure $p$-spin model. This naturally gives rise to the fluctuation results for the free energy in Corollary~\ref{cor:logz}. In this section, we state similar results in a more general setting: the mixed $p$-spin models. Compared to the pure $p$-spin case, which corresponds to the spin glass models on $p$-uniform hypergraphs, the mixed $p$-spin model is a mixture of all possible even and odd $p$-uniform hypergraphs. This creates several further challenges in the analysis, and interestingly a two-parameter multiple transition appears. Specifically, the mixed $p$-spin model also exhibits transition with respect to $p$. It turns out that the dominant clusters in the mixed $p$-spin case only depend on the minimum effective even and odd $p$-spins. We define
\begin{align*}
	p_e:=\min\{\text{even } p\ge 4: \theta_p>0\} \quad p_o := \min\{\text{odd } p\ge 3 : \theta_p>0\}.
\end{align*}

Now we state the main theorems for mixed $p$-spin glass models. Similarly, as in the pure $p$-spin case, we introduce the following notation for the limiting variance in mixed $p$-spin models.
\begin{align*}
	 & v_{(a_e,p_e),(a_o,p_o),b}(\gb, \rho)^2 \\
	 & := \frac{\gb^{2(a_e+a_o)} \cdot \theta_{p_e}^{2a_e} 
  \cdot
 \theta_{p_o}^{2a_0} 
 \cdot \rho^{2b}}{a_e! \cdot a_o! \cdot b! \cdot (p_e!)^{a_e/2} \cdot (p_o!)^{a_o/2}}  \E ( H_{a_e}(H_{p_e}(\eta)/\sqrt{p_e!}) \cdot H_{a_o}(H_{p_o}(\eta)/\sqrt{p_o!})\cdot H_b(\eta)),
\end{align*}
where $a_e,a_o,b  \in \bN$. This corresponds to the scaled variance for the cluster structure with $a_e$ many $p_e$-hyperedges, $a_o$ many $p_o$-hyperedges, and $b$ many odd-degree vertices; these structures are summarized in Table~\ref{tab:transm}. Again we will drop the dependence on $p_e,p_o$ for convenience by using $v_{a_e,a_o,b}(\gb,\rho)^2$. In the case $c=0$ corresponding to $h=0$ case, we further drop the dependence on $\rho$ as in pure $p$-spin case, \ie~we will use $v_{a_e,a_o,0}(\gb)^2$.

\begin{thm}[Mixed $p$-spin at $h=0$]\label{thm:mixed0}
	Under the same assumptions as in Theorem~\ref{thm:h0},
	\begin{itemize}
		\item If $p_e < p_o$, then we have
		      \begin{align*}
			      N^{3(p_e-2)/4} \cdot \log \hat{Z}_N \to \N(0,v_{\textup{mix},1}^2)
		      \text{  in distribution as $ N \to \infty$,}  \end{align*} where $v_{\textup{mix},1}^2 := v_{3,0,0}(\gb)^2 $ as defined in~\eqref{eq:var-h0}.
		\item If $p_o < p_e < 2p_o$, then we have
		      \begin{align*}
			      N^{(2p_o + p_e-6)/4} \cdot \log \hat{Z}_N \to \N(0,v_{\textup{mix},2}^2)
		      \text{  in distribution as $ N \to \infty$,}  	\end{align*} 
        where
		      \begin{align*}
			      v_{\textup{mix},2}^2:= \begin{cases}
				      v_{1,2,0}(\gb)^2,                            & p_o<p_e<2(p_o-1), \\
				      v_{1,2,0}(\gb)^2 + v_{0,4,0}(\gb)^2, & p_e = 2(p_o-1).
			      \end{cases}
		      \end{align*}
		\item If $p_e \ge 2 p_o$, then we have
		      \begin{align*}
			      N^{(p_o-2)} \cdot \log \hat{Z}_N \to \N(0,v_{\textup{mix},3}^2)
		      \text{  in distribution as $ N \to \infty$,}  \end{align*} 
        where $v_{\textup{mix},3}^2 := v_{0,4,0}(\gb)^2$.
	\end{itemize}
\end{thm}

The above theorem implies that if $p_e<p_o$, the mixed $p$-spin model behaves like a pure $p_e$-spin model, similar in the case $p_e\ge 2p_o$. Only in the regime $p_o<p_e<2p_o$, the mixed model behaves like a mixture. Furthermore, from the variance part and associated cluster structures summarized in Table~\ref{tab:transm}, one can see that $p_e = 2(p_o-1)$ is a transition point from mixture to pure $p_o$-spin case.

In order to state the results under a weak external field, we need to first introduce some necessary notations. Define
\begin{align}\label{eq:alphac-mix}
	\ga_c(p_e,p_o) :=
	\begin{cases}
		p_e/8,     & \text{ if } 2<p_e<p_o-1, \\
		(p_e-2)/4, & \text{ if } 2<p_e = p_o - 1, \\
		p_e/8,     & \text{ if } p_o<p_e <2( p_o - 1), \\
		(p_o-1)/4, & \text{ if } p_e\ge 2(p_o-1).
	\end{cases}
\end{align}
A simple remark is that for $p_e=2$, the transition threshold is $\ga_c = \frac14$, which was obtained in~\cite{DW21}.
Similarly we next define the associated 3 regimes $\cR_1(\ga_c(p_o,p_e)),\cR_2(\ga_c(p_o,p_e)),\cR_3$ as in~\eqref{eq:regime}. The corresponding scaling exponents are defined as
\begin{align}
	\gc_1 := \begin{cases}
		         \frac32 (p_e-2),       & \text{ if } p_e<p_o, \\
		         p_o + \frac12 p_e - 3, & \text{ if } p_o<p_e<2p_o, \\
		         2(p_o-2),              & \text{ if } p_e\ge 2p_o.
	         \end{cases}
\end{align}
and
\begin{align}
	\gc_2:=\gc_2(\ga)=
	\begin{cases}
		(p_e-3+4\ga), & \text{ if } p_e<p_o-1, \\
		(p_e-2+2\ga), & \text{ if } p_e = p_o - 1, \\
		(p_o-3+4\ga), & \text{ if } p_e = p_o + 1, \\
		(p_o-3+4\ga), & \text{ if } p_e> p_o+1.
	\end{cases}
\end{align}
and
\begin{align}
	\gc_3:=\gc_3(\ga)= \begin{cases}
		                   (2\ga p_e-1), & \text{ if } p_e<p_o, \\
		                   (2\ga p_o-1), & \text{ if } p_e>p_o.
	                   \end{cases}
\end{align}

The heuristic picture for the above components is as follows. In the zero external field, the leading clusters are $(3_{p_e},0),(1_{p_e},2_{p_o},0)$, and $(4_{p_o},0)$. To find $\ga_c$ after adding a weak external field, we have to compute the contribution from the new clusters with odd-degree vertices and find when their contribution is in the same order as the clusters at $h=0$. The specialty is due to the $(1_{p_{e}},1_{p_{o}},1)$ cluster shown in Figure~\ref{fig:min-struct-mixed}, and it seems to be a discontinuity for the transition w.r.t $p$. The details of those cluster structures and variance components are summarized in Table~\ref{tab:transm}. In terms of external field strength, there are three regimes for $\log \hat{Z}_N(\gb,h)$ similarly as in pure $p$ case. 

\begin{figure}[htbp]
	\centering
	\includegraphics[width=3.5cm,page=11]{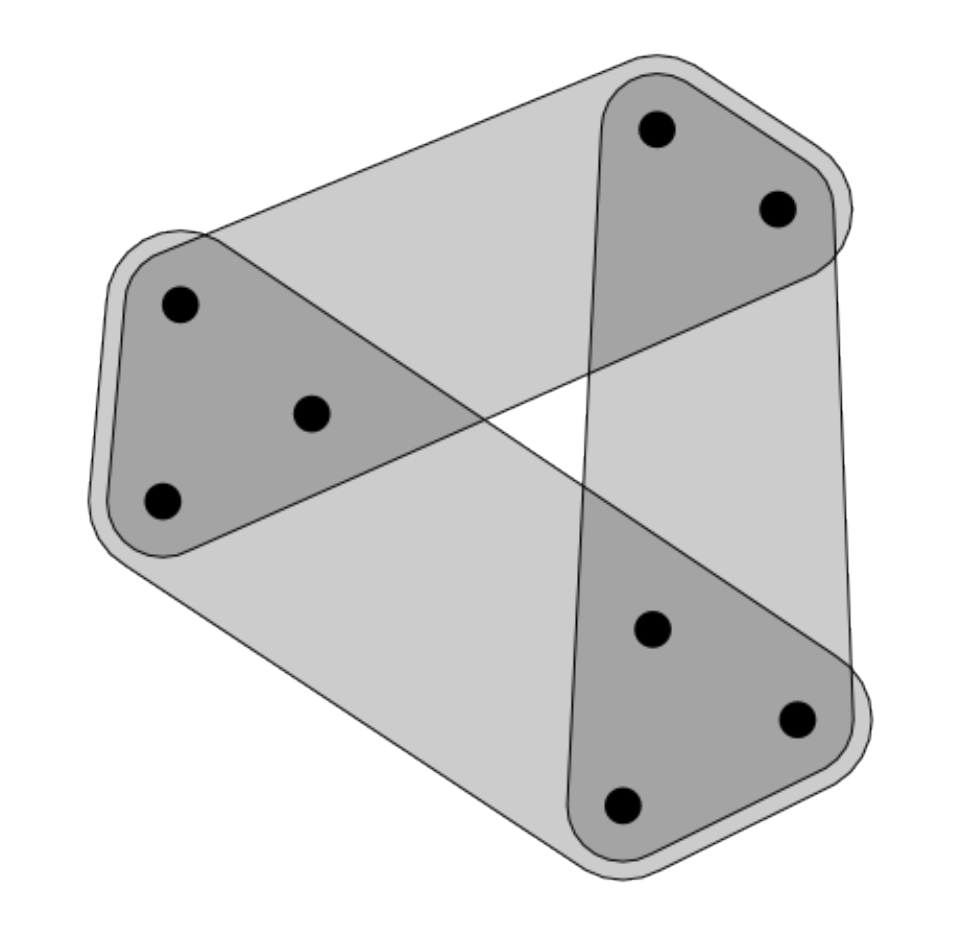}
	\includegraphics[width=3.5cm,page=10]{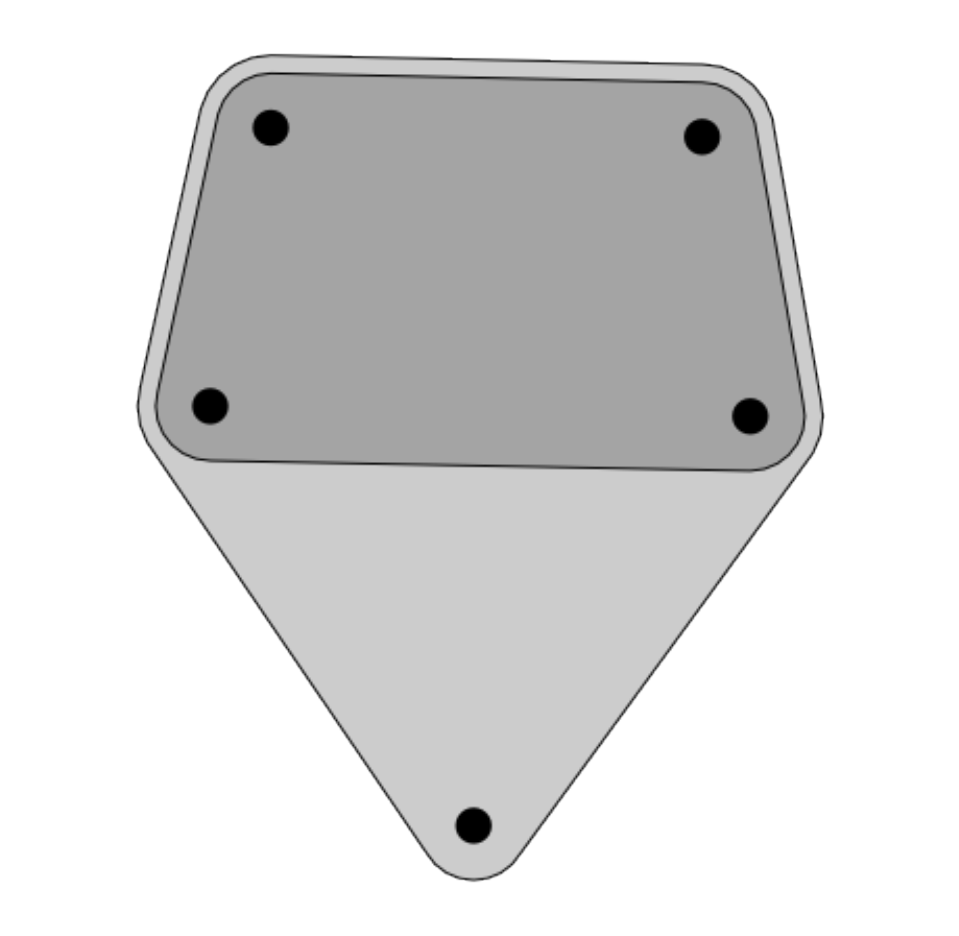}
	\caption{New leading clusters in the mixed $p$-spin models. For example, with $p_o=5$, the left one corresponds to the $(1_{p_o+1},2_{p_o},0)$ cluster. The right one is the $(1_{p_o-1},1_{p_o},1)$ cluster.}
	\label{fig:min-struct-mixed}
\end{figure}

We state the main theorem for mixed $p$-spin models under weak external field in Theorem~\ref{thm:mixedh}. We will focus on the case $\theta_2=0$. When $\theta_2>0$, the $p=2$ clusters dominate, and one can replicate the proof for the SK model (see~\cite{DW21}) in this case.

\begin{thm}[Mixed $p$-spin model weak external field]\label{thm:mixedh}
	Under the same assumptions in Theorem~\ref{thm:h1}, we have
	\begin{enumeratea}
		\item \textsc{\bfseries Sub-critical regime}: If $\ga \in [\ga_c(p_e,p_o), \infty)$, we have
		\begin{align}
			N^{\gc_1/2} \cdot \log \hat{Z}_N \to \N(0,u_1^2)
		\end{align}
		in distribution as $N \to \infty$ where
		\begin{align*}
			u_{1}^2:= \begin{cases}
				\hat{u}_1^2,     & \ga < \ga_c(p_e,p_o), \\
				\hat{u}_{1,c}^2, & \ga  = \ga_c(p_e,p_o).
			\end{cases}
		\end{align*}
		with
		\begin{align*}
			\hat{u}_1^2 := \begin{cases}
				               v_{\textup{mix},1}^2, & p_e < p_o, \\
				               v_{\textup{mix},2}^2, & p_o < p_e< 2p_o, \\
				               v_{\textup{mix},3}^2, & p_e \ge 2p_o.
			               \end{cases}
		\end{align*}
		and
		\begin{align*}
			\hat{u}_{1,c}^2 :=
			\begin{cases}
				v_{3,0,0}(\gb)^2 + v_{2,0,2}(\gb,\rho)^2,                                                              & p_e < p_o-1, \\
				v_{3,0,0}(\gb)^2 + v_{1,1,1}(\gb,\rho)^2,                                                              & p_e = p_o-1, \\
				v_{1,2,0}(\gb)^2 + v_{1,1,1}(\gb,\rho)^2,                                                              & p_e=p_o+1 , \\
				v_{1,2,0}(\gb)^2 + v_{0,2,2}(\gb,\rho)^2,                                                              & p_o+1<p_e<2(p_o-1) , \\
				v_{1,2,0}(\gb)^2 + v_{0,4,0}(\gb)^2 + v_{0,2,2}(\gb,\rho)^2 + v_{0,3,1}(\gb,\rho)^2, & p_e=2(p_o-1) , \\
				v_{0,4,0}(\gb)^2 + v_{0,2,2}(\gb,\rho)^2 + v_{0,3,1}(\gb,\rho)^2,                            & p_e \ge 2p_o.
			\end{cases}
		\end{align*}
		\item \textsc{\bfseries Unified regime I}: If $\ga \in [\half,\ga_c(p_e,p_o))$, we have
		\begin{align*}
			N^{\gc_2/2} \cdot \log \hat{Z}_N \to \N(0,u_2^2)
		\end{align*}
		in distribution as $N \to \infty$ where
		\begin{align*}
			u_2^2 := \begin{cases}
				\hat{u}_2^2,     & \ga \in (1/2, \ga_c(p_e,p_o)), \\
				\hat{u}_{2,c}^2, & \ga = 1/2.
			\end{cases}
		\end{align*}
		with
		\begin{align*}
			\hat{u}_2^2 := \begin{cases}
				               v_{2,0,2}(\gb,\rho)^2, & p_e < p_o-1, \\
				               v_{1,1,1}(\gb,\rho)^2, & p_e = p_o-1, \\
				               v_{1,1,1}(\gb,\rho)^2, & p_e = p_o+1, \\
				               v_{0,2,2}(\gb,\rho)^2, & p_e > p_o+1.
			               \end{cases}
        \end{align*}
        and 
        \begin{align*}
            \hat{u}_{2,c}^2:=
			\begin{cases}
				v_{2,0,2}(\gb,\rho)^2 + v_{1,0,p_e}(\gb,\rho)^2, & p_e < p_o-1, \\
				v_{1,1,1}(\gb,\rho)^2 + v_{1,0,p_e}(\gb,\rho)^2, & p_e = p_o-1, \\
				v_{1,1,1}(\gb,\rho)^2, + v_{0,1,p_o}(\gb,\rho)^2 & p_e = p_o+1, \\
				v_{0,2,2}(\gb,\rho)^2 + v_{0,1,p_o}(\gb,\rho)^2, & p_e > p_o+1.
			\end{cases}
		\end{align*}
		\item \textsc{\bfseries Unified regime II}: If $\ga \in [\quar,\half)$, we have
		\begin{align*}
			N^{\gc_3/2} \cdot\log \hat{Z}_N \to \N(0,u_3^2)
		\end{align*}
		in distribution as $N \to \infty$ where $u_3^2 :=\begin{cases} v_{1,0,p_e}(\gb,\rho)^2, & p_e<p_o, \\ v_{0,1,p_o}(\gb,\rho)^2, & p_e> p_o.\end{cases} $
	\end{enumeratea}
\end{thm}

We will reuse the notation $\cR_1,\cR_2,\cR_3$ as in Section~\ref{sec:intro} to denote the sub-critical regime, unified regime I, and unified regime II, respectively. In the $h=0$ case, we discussed the criterion of classifying the mixed $p$-spin models in terms of the transition results w.r.t $p$. For the model under a weak external field, it is interesting that the mixed model behaves more like a pure model if the weak external field is strong enough. In particular, for $\ga \in \cR_3$, if $p_e<p_o$, the mixed model behaves like a pure $p_e$-spin glass and pure $p_o$-spin otherwise. 

\begin{rem}[Connection to the triviality of geometry]
Notice from Table~\ref{tab:transm}, as the external field gets strong enough, for example, for $\ga \in \cR_3$, the dominant cluster becomes trivial, which is a single hyperedge. On the other hand, in the study of geometry for spherical mixed $p$-spin glass models, it is known~\cite{DJSM21} that in the zero external field case, the Hamiltonian, as a random function, has nontrivial geometry in the sense that it has exponentially many critical points. However, if the external field is strong enough, then the Hamiltonian only has two critical points and thus has a trivial geometry. It will be interesting to explore the connections between those two triviality behaviors. We leave this to future work. There is a long line of works on the triviality of the geometry in spherical spin glass models, one can check~\cites{FL14,Fyo15,RBBC19} and references therein.
\end{rem}

With a good understanding of $\log \hat{Z}_N(\gb,h)$, we can similarly deduce the transitional results for free energy. The expansion of $\log \bar{Z}_N$ is similar to in pure $p$-spin case, as we know that as $p$ gets larger, the variance order gets smaller. Therefore, the dominant term comes from the smallest one $p_m:=\min\{p_e,p_o\}$ and we have
\begin{align}
	\log \bar{Z}_N(\gb,h) & = \sum_{p\ge 3}\sum_{\mvi \in \sE_{N,p}} \log \cosh(\gb\theta_p \cdot J_{\mvi}/N^{(p-1)/2}) \\ 
	                     &  \approx \frac{\gb^2\theta_{p_m}^2\binom{N}{p_m} }{2 N^{p_m-1}} - \frac{\gb^4 \theta_{p_m}^4  \binom{N}{p_m}}{12 N^{2p_m-2}}\cdot \E J^4 + \N\left(0,\frac{\gb^4 \theta_{p_m}^4 \binom{N}{p_m}}{4 N^{2p_m-2}}\cdot \var(J^2)\right).\label{eq:zbar-approx-mix}
\end{align}

\begin{cor}\label{thm:mixed-log}
	Recall $p_m =\min\{p_e,p_o\}$. Assume that $\ga \ge \quar$, and $J_{\mvi},\mvi\in \cup_{p\ge 3}\sE_{N,p}$ satisfy Assumption~\ref{ass:moment}. Then for $\gb<\gb_f$,

	\begin{enumeratea}
		\item If $\ga > \frac{p_m-1}{2p_m} $, we have
		\begin{align*}
			N^{(p_m-2)/2}\cdot \left(\log Z_N - \binom{N}{p_m}\E\log\cosh(\gb\theta_{p_m}J/N^{(p_m-1)/2})-N\log (2\cosh(h)) \right) \to \N(0,v_1^2),
		\end{align*}
		where $v_1^2 = \gb^4 \theta_{p_m}^4\var(J^2)/ 4p_m!$.
		\item If $\ga=\frac{p_m-1}{2p_m} $, we have
		\begin{align*}
			N^{(p_m-2)/2}\cdot \left(\log Z_N - \binom{N}{p_{m}}\E\log\cosh(\gb\theta_{p_m}J/N^{(p_m-1)/2}) -N\log (2\cosh(h))\right) \to \N(0,v_2^2),
		\end{align*}
		where $v_2^2 :=v_1^2 + v_{1,p_m}(\gb,\rho)^2 $ with $ v_{1,p_m}(\gb,\rho)^2 := \begin{cases}  v_{1,0,p_e}(\gb,\rho)^2, & \text{if} \ p_m = p_e, \\ v_{0,1,p_o}(\gb,\rho)^2, & \text{if} \ p_m = p_o. \end{cases}$
		\item If $\frac{1}4 \le \ga < \frac{p_m-1}{2p_m} $,we have
		\begin{align*}
			N^{(2\ga p_m -1)/2}\cdot\left(\log Z_N - \binom{N}{p_{m}}\E\log\cosh(\gb\theta_{p_m}J/N^{(p_m-1)/2})-N\log (2\cosh(h)) \right) \to \N(0,v_3^2)
		\end{align*}
		where $v_3^2 := v_{1,p_m}(\gb,\rho)^2$.
	\end{enumeratea}
\end{cor}

\begin{rem}[multi-scale fluctuation for the free energy]
As we pointed out in Remark~\ref{rem:multi-scale-pure}, the multi-scale fluctuation phenomenon holds for $\log \hat{Z}_N$ in the pure $p$-spin model. However, this is not the case on the free energy level due to the dominance of $\log \bar{Z}_N$. In the mixed $p$-spin setting, even on the free energy level, there still exists a similar multi-scale fluctuation phenomenon due to the two-parameter multiple transition picture. From the above corollary, if we subtract the $\Theta(N^{-(p_m-2)})$ terms, there should be another limit theorem in a smaller scale $\Theta(N^{-(q_{m}-2)})$, where $q_{m}:=\min\{p\neq p_m: \theta_p >0\}$ is the second smallest effective $p$-spins.
\end{rem}

\begin{table}[htbp]
	\centering
	\caption{Summary of the transition behavior for mixed $p$-spin models}
	\begin{tabular}{@{}cccl@{}}
		\toprule
		External field $h$     & Condition            & $\gc$ & Cluster Structure \\
		\midrule
		\multirow{3}{*}{$\ga \in (\ga_c,\infty]$} & $p_e<p_o$            & $3(p_e-2)/2$                            & $(3_{p_e},0)$ \\
		& $p_o<p_e< 2(p_o-1)$  & $p_o-3+p_e/2$                           & $ (1_{p_e},2_{p_o},0)$ \\
		& $p_e = 2(p_o-1)$     & $p_o-3+p_e/2$                           & $(1_{p_e},2_{p_o},0),(4_{p_o},0)$ \\
		& $p_e\ge2p_o$         & $2(p_o-2)$                              & $ (4_{p_o},0)$ \\
		\midrule
		\multirow{5}{*}{$\ga = \ga_c(p_o,p_e)$}   & $p_e<p_o-1$          & $3(p_e-2)/2$                            & $(3_{p_e},0),(2_{p_e},2)$ \\
		& $p_e=p_o-1$          & $3(p_e-2)/2$                            & $(3_{p_e},0),(1_{p_e},1_{p_o},1)$ \\
		& $p_e=p_o+1$          & $p_o-3+p_e/2$                           & $(1_{p_e},2_{p_o},0),(1_{p_e},1_{p_o},1)$ \\
		& $p_o+1<p_e<2(p_o-1)$ & $p_o-3+p_e/2$                           & $(1_{p_e},2_{p_o},0),(2_{p_o},2)$ \\
		& $p_e=2(p_o-1)$       & $p_o-3+p_e/2$                           & $(1_{p_e},2_{p_o},0),(4_{p_o},0)$,\\
        & & & $(2_{p_o},2),(3_{p_o},1)$ \\
		& $p_e\ge 2p_o$        & $2(p_o-2)$                              & $(4_{p_o},0),(2_{p_o},2),(3_{p_o},1)$ \\
		\midrule
		\multirow4{*}{$\ga \in (1/2,\ga_c)$}      & $p_e<p_o-1$          & $p_e-3+4\ga$                            & $(2_{p_e},2)$ \\ & $p_e = p_o-1$ & $p_e-2+2\ga$ & $ (1_{p_e},1_{p_o},1)$ \\& $p_e=p_o+1$ & $p_o-3+4\ga$ & $ (1_{p_e},1_{p_o},1)$ \\
		& $p_e>p_o+1$          & $p_o-3+4\ga$                            & $ (2_{p_o},2)$ \\
		\midrule
		\multirow4{*}{$\ga = 1/2$}                & $p_e<p_o-1$          & $p_e-1$                                 & $(2_{p_e},2),(1_{p_e},p_e)$ \\
		& $p_e = p_o-1$        & $p_e-1$                                 & $ (1_{p_e},1_{p_o},1),(1,p_e)$ \\& $p_e=p_o+1$& $p_o-1$ & $ (1_{p_e},1_{p_o},1),(1,p_o)$ \\
	& $p_e>p_o+1$          & $p_o-1$                                 & $ (2_{p_o},2),(1_{p_o},p_o)$ \\
		\midrule
		\multirow2{*}{$\ga \in [1/4,1/2)$}        & $p_e<p_o$            & $2\ga p_e-1$                            & $(1_{p_e},p_e)$ \\
	& $p_e>p_o$            & $2\ga p_o-1$                            & $(1_{p_o},p_o)$ \\
		\bottomrule
	\end{tabular}
	\label{tab:transm}
\end{table}

\subsection{A word on the proof}
We briefly discuss the extra challenges for the proof of the main results in mixed $p$-spin setting and our strategy to deal with it. The first and core challenge is caused by the flexibility of $p$-hyperedges. We must show that the dominant clusters can not include large $p$-hyperedges. This is similar to the large hypergraph decay results in Proposition~\ref{prop:first-red}. The high-level idea  for establishing this result is to decompose the large $p$-hyperedge into several smaller hyperedges, this will increase the cluster size, and then we use the large hypergraph decaying results to conclude. Once we know that the dominant clusters will have finitely many hyperedges and all the hyperedges are finite tuples, the next challenge will be identifying the leading clusters. Although the cluster enjoys the above nice property, there still can be many possibilities due to various combinations of different $p$-hyperedges. This part needs very careful analysis. In particular, the leading clusters change w.r.t.~two parameters, external field strength, and the mixture of hyperedges. Finally, after obtaining all the dominant clusters, we use multivariate Stein's method in Theorem~\ref{thm:rr-mvstein} to establish the joint convergence of different clusters to the normal limit. The computations there become more involved in the pure $p$-spin case. We highlight the ideas in Section~\ref{sec:stein-mix}.

\section{Results for Multi-Species Mixed \texorpdfstring{$p$}{p}-Spin Models}\label{sec:multi-p}

The multi-species model was originally introduced in~\cite{BCMT15}. It can be regarded as a natural inhomogeneous extension of the classical mean field spin glass models. In recent years, this model has received much attention due to its connections to the neural network~\cites{ACM21, ABCM20, ACCM21} and its inherent fundamental challenges in the non-convex case~\cite{Mou21}. Let us first define multi-species mixed $p$-spin models. Consider a set $\sS$ of species with $\abs{\sS}=k<\infty$, for the number of spins $N>0$, let
\begin{align*}
	\{1,2,\ldots, N\} = \bigcup_{s \in \sS} I_s, \quad \text{and} \ I_s \cap I_t = \emptyset \ \text{for} \ s\neq t.
\end{align*}
Assume the density ratio of species $s \in \sS$ is $\gl_s: = \lim_{N\to\infty} \frac{\abs{I_s}}{N}$. We denote 
\[
\gL:=\mathrm{diag}(\gl_1,\gl_2,\cdots, \gl_k).
\]
For each $p>3$, let
$$
	\gD_p^2:=(\gD_{s_1,\ldots,s_p}^2)_{s_1,\ldots,s_p \in \sS} \in \dR_+^{\abs{\sS}^p}
$$
be a symmetric $p$-dimensional tensor. For $\mvgs \in \{-1,+1\}^N$, the Hamiltonian of multi-species pure $p$-spin model can be defined as follows:
\begin{align}
	H_{N,\gD_p}(\mvgs) = \frac{1}{N^{(p-1)/2}} \sum_{\mvi\in\sE_{N,p}} \gD_{s(\mvi)} J_{\mvi} \gs_{\mvi},
\end{align}
where $s(\mvi)$ maps each coordinate of the $p$-tuple $\mvi$ to a species in $\sS$. The Hamiltonian of multi-species mixed $p$-spin model is
\begin{align}
	H_{N,\gD}(\mvgs) := \sum_{p\ge 3} H_{N,\gD_p}(\mvgs).
\end{align}
To make sure the model is well-defined, it needs the following fast decaying assumption,
\begin{align}
	\sum_{p \ge 3}  \norm{\gL^{\otimes p}\gD_p^2}_{\infty}/(p-2)! < \infty.
\end{align}
Here the limiting structure function is
\begin{align*}
	\xi(\mvx):= \sum_{p\ge 3} \frac{1}{p!}\la \gL^{\otimes p}\gD_{p}^2,\mvx^{\otimes p}\ra \text{ for } \mvx\in [-1,1]^{k}.
\end{align*}

Parisi formula for the multi-species SK model was first established in~\cite{Pan15} under the positive semi-definite assumption on $\gD$. For general indefinite $\gD$, computing the limiting free energy still remains a challenging open question. The fluctuation results for the general multi-species SK model at a high-temperature regime were obtained in~\cite{DW20}. In the spherical setting, some progress has been made in evaluating the limiting free energy. Recently, for multi-species mixed $p$-spin spherical models, Bates and Sohn established the Crisanti-Sommers formula in~\cites{BateS22a, BateS22b} under the positive semi-definite assumption on $\gD$. Shortly after these works, Subag~\cites{Sub21b,sub21c} utilized the Thouless-Anderson-Palmer (TAP) approach instead to compute the limiting free energy and ground state energy for multi-species pure $p$-spin case. Notably, this approach works for general indefinite $\gD$. The fluctuation results for both Ising and spherical multi-species mixed $p$-spin models are still missing. Using the intuition from the last section, we can easily obtain the fluctuation results for general indefinite multi-species Ising mixed $p$-spin models in the entire high-temperature regime.

First, let us introduce the characterization of the critical inverse temperature for multi-species models.
\begin{align*}
	\gb_{f}(\xi):=\inf_{\mvx\in (0,1]^{k}} \sqrt{\frac{\sum_{i=1}^{k}\gl_i I(x_i)}{\xi(\mvx)}}.
\end{align*}
Analogue of Theorem~\ref{thm:beta_c} can be easily established in the multi-species setting. One needs to carefully trace the structure-function.

\begin{thm}\label{thm:multi-beta_c}
	Under similar assumptions as~\ref{ass:mph},~\ref{ass:moment} and~\ref{ass:wef} for multi-species model. For $\gb<\gb_f(\xi)$, we have for all $N$ large
	\begin{align*}
		\E [(\hat{Z}_N(\gb,h))^2] \le \frac{C}{(\gb_f^2-\gb^2)^{k/2}}\exp\left(\frac{C\rho^4\gb^2}{\gb_f^2-\gb^2} \right)
	\end{align*}
	where $C>0$ is a universal constant. In particular, we have
	\begin{align}
		\frac1N \log Z_N(\gb) \to \frac{1}2\gb^2\xi(1) \text{ in probability as } N\to\infty.
	\end{align}
\end{thm}

\begin{rem}
	The proof of Theorem~\ref{thm:multi-beta_c} is morally the same as the single species case in Theorem~\ref{thm:beta_c}, whose proof is given in Section~\ref{sec:exp}. For multi-species case, one needs to work on the following vector of magnetization $S_N^{\sS}:=\left(\sum_{i\in I_s}\gs_i\right)_{s\in \sS}$ and also carefully trace the structure-function $\xi(\mvx) =\sum_{p\ge 3} \frac{1}{p!}\la \gL^{\otimes p}\gD_{p}^2,\mvx^{\otimes p}\ra $.
\end{rem}

 With the critical inverse temperature, we now state the multi-species fluctuation results up to $\gb_f$. Before that, some necessary notations are needed. Let
\begin{align*}
	p_e &:=\min\{\text{even } p\ge 4:\text{$\gD_p$ is not the zero $p$-tensor}\},  \\
    p_o &:= \min\{\text{odd } p\ge 3 : \text{$\gD_p$ is not the zero $p$-tensor}\}.
\end{align*}
By abusing the notation a bit, we continue to use the notation for the partition function $Z_{N}(\gb,h)$ and its similar decomposition in the multi-species setting. We state the analog of Theorem~\ref{thm:mixed0}.
For convenience, we assume that $\tr(\gD_{p_e}), \tr(\gD_{p_o})>0$, to make sure that the leading cluster structure remains the same as in the single-species case. In the general $\gD_p$ case, a similar result holds after identifying the leading cluster structure. 

\begin{thm}[multi-species mixed $p$-spin at $h=0$]\label{thm:multi0}
	Under the same assumptions as in Theorem~\ref{thm:h0}, 
	\begin{itemize}
		\item If $p_e < p_o$, then we have
		      \begin{align*}
			      N^{3(p_e-2)/4} \cdot \log \hat{Z}_N \to \N(0,v^2)
		      \text{  in distribution as $ N \to \infty$,}  	\end{align*} 
        for some $v^2 \in (0,\infty) $.
		\item If $p_o < p_e \le 2p_o$, then we have
		      \begin{align*}
			      N^{(2p_o + p_e-6)/4} \cdot \log \hat{Z}_N \to \N(0,v^2)
		      \text{  in distribution as $ N \to \infty$,}  	\end{align*} for some $v^2 \in (0,\infty) $.
		\item If $p_e > 2 p_o$, then we have
		      \begin{align*}
			      N^{(p_o-2)} \cdot \log \hat{Z}_N \to \N(0,v^2)
		      \text{  in distribution as $ N \to \infty$,}  \end{align*} 
        for some $v^2 \in (0,\infty) $.
	\end{itemize}
\end{thm}

The variances $v^2$ in each regime can be explicitly computed once the leading clusters are identified. We do not pursue the most general form here. The transitional results are also essentially same as in the single species case. The thresholds and corresponding scaling components can be defined similarly. We state the results as follows without repeating the definitions. The only difference lies in the explicit asymptotic variance.

\begin{thm}[multi-species mixed $p$-spin under weak external field]\label{thm:multi-h}
	Under the same assumptions in Theorem~\ref{thm:h1}, we have
	\begin{enumeratea}
		\item \textsc{\bfseries Sub-critical regime}: If $\ga \in [\ga_c(p_o,p_e),\infty)$, we have
		\begin{align*}
			N^{\gc_1/2} \cdot \log \hat{Z}_N \to \N(0,v^2)
		\text{  in distribution as $ N \to \infty$,}  	\end{align*} 
        for some $v^2 \in (0,\infty) $.
		\item \textsc{\bfseries Unified regime I}: If $\ga \in [\half,\ga_c(p_o,p_e))$, we have
		\begin{align*}
			N^{\gc_2/2} \cdot \log \hat{Z}_N \to \N(0,v^2)
		\text{  in distribution as $ N \to \infty$,}  	\end{align*} 
        for some $v^2 \in (0,\infty) $.
		\item \textsc{\bfseries Unified regime II}: If $\ga \in [\quar,\half]$, we have
		\begin{align*}
			N^{\gc_3/2} \cdot\log \hat{Z}_N \to \N(0,v^2)
		\text{  in distribution as $ N \to \infty$,}  	\end{align*} 
        for some $v^2 \in (0,\infty) $.
	\end{enumeratea}
\end{thm}

 After this result, one can similarly obtain the analog of Corollary~\ref{thm:mixed-log}. We do not state it formally here. The proof of the main results for the multi-species model is basically same as the single-species case. The steps for obtaining finite leading clusters are nearly same since the inhomogeneous variance profile $\gD$ has all finite entries and there are finitely many different species, one can always use a uniform bound to prove the analog of Proposition~\ref{prop:first-red},~\ref{prop:second-red} and~\ref{prop:p-reduction}. Identifying the leading clusters  needs only the variance order. Therefore, the only thing changing is the asymptotic variance forms. After this, one can repeat the multivariate Stein's method to establish the Gaussian limit as discussed in Section~\ref{ssec:stein}.


\section{Large Hypergraph Decay Regime}\label{sec:exp}

\subsection{Large hypergraph decay}

We begin with the proof of the property that large hypergraph contributions decay fast in the whole high-temperature regime. Compared to the SK ($p=2$) model, the Gaussian trick (see \cite{DW21}*{Lemma 4.1}) for controlling the second moments of $\hat{Z}_N$ is not available any more due to the loss of quadratic features. Instead, we seek an integral approximation of the targeted discrete sum, then directly control the integral upper bound.

We first look at the second moment of $\hat{Z}_N(\gb,h)$. Define $\abs{\mvi}=p$ for $\mvi\in\sE_{N,p}$. Recall that
\begin{align*}
	\hat{\go}_{\mvi} = \tanh(\gb \theta_{p}\cdot  J_{\mvi}/N^{(p-1)/2}) \text{ for } \mvi\in\cup_{p\ge 2}\sE_{N,p}
\end{align*}
and
\begin{align*}
	\hat{Z}_N(\gb,h)
	= \E_{h} \prod_{p\ge 2}\prod_{\mvi\in\sE_{N,p}} (1+\gs_{\mvi}\cdot \hat{\go}_{\mvi}).
\end{align*}

\begin{lem}\label{lem:zhatvar}
	Assume~\ref{ass:mph},~\ref{ass:moment} and~\ref{ass:wef}. For $\gb<\gb_f(\xi)$, we have for all $N$ large,
	\begin{align*}
		\E [(\hat{Z}_N(\gb,h))^2] \le \frac{c}{\sqrt{\gb_f^2-\gb^2}}\exp\left(\frac{\rho^4\gb^2}{\gb_f^2-\gb^2} \right)
	\end{align*}
	where $c>0$ is a universal constant.
\end{lem}
\begin{rem}\label{rem:zhatvar}
    In Lemma~\ref{lem:zhatvar} we do not need $\theta_2$ to be $0$. However, when $\theta_2=0$ one can get a better $N$-dependent upper bound.
Moreover, using $\E \tanh^2(\gb J/N^{(p-1)/2}) \le \gb^2/N^{p-1}$, one can replace $\E(\hat{Z}_N(\gb,h))^2$ by $\sum_{\gC \subseteq \cup_{p\ge 3}\sE_{N,p}} h^{2\abs{\partial \gC}}  \prod_{p\ge 3}\left(\theta_p^2\gb^2/N^{p-1}\right)^{\abs{\gC}_p}$ in Lemma~\ref{lem:zhatvar}, where $|\gC|_p$ is the number of $p$-hyperedges in $\gC$.
\end{rem}

\begin{proof}
	We notice that
	\begin{align*}
		\hat{Z}_N(\gb,h)^2
		= \E_{h} \prod_{p\ge 2}\prod_{\mvi\in\sE_{N,p}} (1+(\gs_{\mvi}+\gt_{\mvi})\cdot \hat{\go}_{\mvi} + \gs_{\mvi}\gt_{\mvi}\cdot \hat{\go}_{\mvi}^2).
	\end{align*}
	After taking the expectation with respect to the disorder, we have
	\begin{align*}
		\phi_N(\gb,h) := \E\hat{Z}_N(\gb,h)^2 & = \E_{h} \prod_{p\ge 2}\prod_{\mvi\in\sE_{N,p}} (1+ \gs_{\mvi}\gt_{\mvi} \E\hat{\go}_{\mvi}^2) \\
		                                      & =\E_{\tilde{h}^2} \prod_{p\ge 2}\prod_{\mvi\in\sE_{N,p}} (1+ \gs_{\mvi} \cdot\E \hat{\go}_{\mvi}^2)
		\le \E_{\tilde{h}^2} \exp\left(\sum_{p\ge 2}  \sum_{\mvi\in\sE_{N,p}} \E \hat{\go}_{\mvi}^2\cdot \gs_{\mvi}\right)
	\end{align*}
	where $\tilde{h}^2 := \tanh^{-1}(\tanh^2(h)) \le h^2$. Note that,
	\begin{align*}
		\E \hat{\go}_{\mvi}^2\le \gb^2\theta_{p}^2/N^{p-1} \text{ for }\mvi\in\sE_{N,p}
	\end{align*}
	and
	\begin{align*}
		\abs{\left(\sum_{i=1}^{N}\gs_{i}\right)^{p} - p!\sum_{\mvi\in\sE_{N,p}} \gs_{\mvi}} \le N^{p}-(N)_{p}=N^{p}\left(1-\prod_{k=1}^{p-1}(1-k/N)\right)\le \binom{p}2N^{p-1}.
	\end{align*}
	Thus
	\begin{align}
		\phi_N(\gb,h)
		 & \le  \E_{\tilde{h}^2} \prod_{p\ge 2} \exp\biggl(N\gb^2 \theta_{p}^2/p!\cdot |S_{N}/N|^{p} + \gb^2 \theta_{p}^2/2(p-2)!\biggr) \\
		 & = \exp\biggl(\frac12\gb^2\sum_{p\ge 2} \theta_{p}^2/(p-2)!\biggr)\cdot \E_{\tilde{h}^2} \exp\biggl(N\gb^2\cdot \xi(|S_N/N|)\biggr)
	\end{align}
	where $S_N: = \sum_{i=1}^N \gs_i$ and $\xi(x)=\sum_{p\ge 2}\theta_{p}^2x^{p}/p!$. Continuing from the r.h.s, we have
	\begin{align*}
		\E_{\tilde{h}^2} \exp\biggl(N\gb^2\cdot \xi(|S_N/N|)\biggr)
		 & = e^{-N\log \cosh\tilde{h}^2}\cdot \E_{0} \exp\biggl(N\gb^2\cdot \xi(|S_N/N|) + \tilde{h}^2 \cdot S_{N}\biggr) \\
		 & \le \E_{0} \exp\biggl(N\gb^2\cdot \xi(|S_N/N|) + \tilde{h}^2 \cdot S_{N}\biggr).
	\end{align*}
	Here  the expectation $\E_0$ is with respect to the i.i.d.~$\{-1,+1\}$ valued $\gs_i$ with mean $0$. Our goal is to derive a uniform bound on the r.h.s.~of the above expression.

	Now with $t=\gb/\gb_{f}\in (0,1)$, we have
	\begin{align*}
		\gb^2\cdot \xi(x) \le t^2 \cdot I(x) \text{ for all } x\in[0,1]
	\end{align*}
	where
	\begin{align*}I(x)=\frac12(1+x)\ln(1+x)+\frac12(1-x)\ln(1-x).
	\end{align*}
	Moreover, we will use the fact that
	\begin{align*}
		\pr(|S_{N}|\ge Nx)\le 2\exp(-NI(x)),
	\end{align*}
	and
	\begin{align*}
		x^2/2\le I(x) \le x^2/2\cdot (1 +x^2/2)
	\end{align*}
	for all $x\in[0,1]$. With
	\begin{align*}
		f_N(x)=N\gb^2
		\xi(x/N)+\log\cosh(\hh\cdot x) \le Nt^2\cdot I(x/N) +\log\cosh(\hh\cdot x),
	\end{align*}
	which is convex and increasing in $[0,N]$
	we can write
	\begin{align*}
		 & \E_{0} \exp\biggl(N\gb^2\cdot \xi(|S_N/N|) + \hh \cdot S_{N}\biggr) \\
		 & = \E_0\exp(f_N(|S_N|)) \\
		 & = \E_0\exp(f_N(|S_N|\wedge k_N)) + \E_0(\exp(f_N(|S_N|))- \E_0\exp(f_N( k_N)))\ind_{|S_N|>k_N}
	\end{align*}
	for some $k_N$, to be chosen later so that
	\begin{align*}
		t_N:=t\cdot \sqrt{1+k_N^2/2N^2}\le \sqrt{(1+t^2)/2}
	\end{align*}
	or
	\begin{align*}
		1-t_N^2 \ge (1-t^2)/2>0.
	\end{align*}
	Thus, we have with $\eta\sim\N(0,1)$,
	\begin{align*}
		\E_0\exp(f_N(|S_N|\wedge k_N))
		 & \le \E_0\exp(t_N^2 S_N^2/2N + \hh\cdot S_N) \\
		 & = \E \exp(N\log\cosh(\eta \cdot t_N +\sqrt{N\hh^2})/\sqrt{N})
		\le \E \exp(\half(\eta \cdot t_N+\rho^2)^2).
	\end{align*}
	Here, we used the Gaussian trick in the equality. Simplifying we get
	\begin{align*}
		\E_0\exp(f_N(|S_N|\wedge k_N))
		 & \le (1-t_N^2)^{-\half}\exp(\rho^4/2(1-t_N^2)).
	\end{align*}
	Thus, we have
	\begin{align*}
		 \E_0(\exp(f_N(|S_N|)) &- \E_0\exp(f_N( k_N)))\ind_{|S_N|>k_N} \\
		 & =\sum_{k=k_N+1}^N (\exp(f_N(k)) - \exp(f_N(k-1))) \cdot \pr(|S_N|\ge k) \\
		 & \le \sum_{k=k_N+1}^N f_N'(k)\exp(f_N(k)) \cdot 2\exp(-NI(k/N)).
      \end{align*}
      
Simplifying, we get
   \begin{align*}
		 & \le 2(\gb^2\xi'(1)+\hh)\cdot \sum_{k=k_N+1}^N \exp(-(1-t^2)NI(k/N) + \hh k) \\
		 & \le 2(\gb^2\xi'(1)+\hh)\cdot \sum_{k=k_N+1}^N \exp(-(1-t^2)k^2/2N + \rho^2 k/\sqrt{N}) \\
		 & \le c\sqrt{N} \int_{k_N/\sqrt{N}}^\infty \exp(-(1-t^2)x^2/2 + \rho^2 x)\ dx \\
		 & \le c\sqrt{N} \frac{1}{\sqrt{1-t^2}}\exp(\rho^4/2(1-t^2))\cdot \pr(\eta \ge \sqrt{1-t^2}\cdot (k_N/\sqrt{N} - \rho^2/(1-t^2))).
	\end{align*}
	By choosing
	\begin{align*}
		4\rho^4/(1-t^2)^2 + 8\log N/(1-t^2)\le k_N^2/N \le (1-t^2)N
	\end{align*}
	we get the stated result.
\end{proof}

From the analysis, it is relatively easy to derive the decay results for large graph contributions in $p$-spin models. We formulate this into the following Proposition. First denote that
\begin{align}
	\hat{Z}_{N,m} := \sum_{\gC \subseteq \cup_{p\ge 3}\sE_{N,p}, \abs{\gC} \le m} \hat{h}^{ \mid \partial \gC|}\cdot \hat{\go}(\gC).
\end{align}

\begin{prop}[Large hypergraph error decay]\label{prop:first-red}
	If $\gb< \gb_f$, for $m\ge O_{\gb,\rho}(\log N) $, we have
	\begin{align}
		N^{p_m-2} \cdot\norm{\hat{Z}_{N} - \hat{Z}_{N,m}}_2 \to 0 \text{ as } N\to\infty.
	\end{align}
\end{prop}
\begin{proof}[Proof of Proposition~\ref{prop:first-red}]
	Recall that $\gb_{N,p}^2=N^{(p-1)}\E \tanh^2(\gb J/N^{(p-1)/2})$ and $h= \rho N^{-\ga}$,
	\begin{align}
		\phi_N(\gb,h) = \E \hat{Z}_N(\gb,h)^2 = \sum_{\gC \subseteq \cup_{p\ge 3}\sE_{N,p}}\hat{h}^{2\abs{\partial \gC}} \prod_{p\ge 3}\left( \frac{\theta_p^2\gb_{N,p}^2}{N^{p-1}}\right)^{\abs{\gC}_p},
	\end{align}
	where $\abs{\gC}_p$ denote the number of $p$-hyperedges in the hypergraph $\gC$, and $\abs{\gC} = \sum_{p\ge 3} \abs{\gC}_p$.
	On the other hand, we have
	\begin{align*}
		\sum_{\gC \subseteq \cup_{p\ge 3}\sE_{N,p}, \abs{\gC}\ge m} \hat{h}^{2\abs{\gC}} \prod_{p\ge 3}\left( \frac{\theta_p^2\gb_{N,p}^2}{N^{p-1}}\right)^{\abs{\gC}_p}   \le e^{-2my} \phi_N(e^y\gb, h).
	\end{align*}
	Following the bound for $\phi_N(\gb,h)$ derived before in Lemma~\ref{lem:zhatvar} and the Remark~\ref{rem:zhatvar}, we have
	\begin{align}
		\phi_N(\gb,h) \le \frac{c}{\sqrt{\gb_f^2-\gb_{N,p}^2}}\exp\left(\frac{\rho^4\gb_{N,p}^2}{\gb_f^2-\gb_{N,p}^2} \right).
	\end{align}
	Then
	\begin{align*}
		N^{2(p_m-2)} \cdot \sum_{\gC \subseteq \cup_{p\ge 3}\sE_{N,p}, \abs{\gC}\ge m.}\hat{h}^{2\abs{\gC}} \cdot \prod_{p\ge 3}\left( \frac{\theta_p^2\gb^2}{N^{p-1}}\right)^{\abs{\gC}_p}             
        &\le c N^{2(p_m-2)} \cdot \exp(-2my),
	\end{align*}
	as long as $e^y\gb<\gb_f$ or $y<\log(\gb_f/\gb)$.
	From the last step, as long as $my-(p_m-2)\log N \to \infty$ we have the required result.
\end{proof}

	In the last proposition, we proved that the very large size clusters' contribution to $\hat{Z}_N$ decays exponentially fast. This allows us to deal with much smaller hypergraphs and simplify the problem significantly. From the proof, it can be seen that this large graph decaying regime intrinsically characterizes the critical inverse temperature $\gb_f$. In the SK model ($p=2$ case)~\cite{DW21}, this type of result is strong enough to obtain a reduction of the clusters to finite size since the structures in the ordinary graph are relatively simple. However, in the $p$-spin case, this reduction can at most reduce the cluster size to $\log N$, which is due to the complicated structures in the hypergraph. In the following part, we are devoted to establishing a further reduction, namely reducing the $\log N$ size clusters to finite size. The key technique we utilized is a form of chaos expansion and second-moment computations. The strategy is as follows. 

	\begin{prop}\label{prop:second-red}
		For given mixed $p$-spin models $\xi$, in the regime $\gb<\gb_f(\xi)$ and external field strength $\ga \ge \frac14$, let
		\begin{align}
			\hat{Z}_{N,4}:=\sum_{\abs{\gC}\le4, \abs{\partial \gC}\ge 0} \hat{h}^{\abs{\partial \gC}} \cdot \hat{\go}(\gC),
		\end{align}
		we have \begin{align}
			N^{p_m-2} \cdot \norm{ \hat{Z}_{N,m} - \hat{Z}_{N,4}}_2 \to 0
		\end{align}
		as $N \to \infty$.
	\end{prop}

	\begin{rem}
		The rescaling $N^{p_m-2}$ corresponds to the smallest variance cluster (see Table~\ref{tab:transm}). The goal is to show that even under this strongest scaling, any clusters of size greater than 4 are in subleading order. Thus, in all regimes, it suffices to consider clusters of size less than or equal to 4.
	\end{rem}

	\begin{proof}
		The proof is divided into two steps. First, we establish the results in the setting of pure $p$-spin models, then we prove a helpful lemma that suffices to establish the results for mixed $p$-spin models.
		For the pure $p$-spin model, the second moment of the LHS is the sum over $\ell=5,6,\ldots,m$ of
		\begin{align*}
			 N^{2(p-2)}\cdot\sum_{\abs{\gC}=\ell, \abs{\partial \gC} \ge 0} \hat{h}^{2\abs{\partial \gC}}\E \prod_{i \in \gC} \tanh^2\left(\frac{\gb J_i}{N^{(p-1)/2}}\right)
			 & \le N^{2(p-2)}\cdot \sum_{\abs{\gC}=\ell, \abs{\partial \gC} \ge 0} \hat{h}^{2\abs{\partial \gC}}\cdot  \frac{\gb^{2\ell}}{N^{(p-1)\ell}}.
    \end{align*}
    Now the RHS can be bounded by 
    \begin{align*}
			 N^{2(p-2)}\cdot \frac{\gb^{2\ell}}{N^{(p-1)\ell}} \frac{1}{\ell!} \E_{\hh} \left(\sum_{\mvi \in \sE_{N,p}} \gs_{\mvi } \right)^\ell 
			 & \le  N^{2(p-2)}\cdot \frac{\gb^{2\ell}}{N^{(p-1)\ell}} \cdot \frac{1}{\ell! (p!)^\ell} \E_{\hh}\left(\sum_{i=1}^N \gs_i\right)^{p\ell}
    \end{align*}
	Recall that, $\mvgs = (\gs_1, \gs_2,\ldots, \gs_N)$ are i.i.d.~with mean $\hat{h} = \tanh(h)$. Let $\bar{\mvgs} = (\bar{\gs}_1, \bar{\gs}_2,\ldots, \bar{\gs}_N)$ be the centered version of $\mvgs$. Thus, we can further bound the RHS by
    \begin{align*}
			 & =  N^{2(p-2)}\cdot \frac{\gb^{2\ell}}{N^{(p-1)\ell}} \cdot \frac{1}{\ell! (p!)^\ell} \cdot N^{p\ell/2}\cdot \E_{\bar{\mvgs}} \left(\frac{1}{\sqrt{N}}\sum_{i=1}^N   \bar{\gs}_i + \sqrt{N\hat{h}^4}\right)^{p\ell} \\
			 & \le  N^{2(p-2)}\cdot\frac{\gb^{2\ell}}{N^{(p-1)\ell}} \cdot \frac{1}{\ell! (p!)^\ell} \cdot N^{p\ell/2} \cdot \left(\sqrt{cp\ell}+\sqrt{\rho}\right)^{p\ell}
			 \approx  N^{(p-2)(2-\ell/2)} \cdot (cp\ell)^{p\ell/2},
	\end{align*}
 
  In the above derivation, $\E_{\bar{\mvgs}}$ is the expectation with respect to $\bar{\mvgs}$, which is basically same as $\E_0$ used in the proof of Lemma~\ref{lem:zhatvar}. Since $h$ is in the regime $\ga \ge \frac14$, that is, $\sqrt{N\hat{h}^4} = O(1)$, the constant $c$ in the last two steps depends on $\rho$. The inequality follows from the classical results about moment estimates for i.i.d sum of bounded random variables~\cites{Pin15,Lata97}. Notice that $p\ge 3$, if $4 < \ell < m$, where $m$ is some finite constant independent of $N$, then the polynomial term $N^{(p-2)(1-\ell/2)}$ decays much faster. Note that even if $m$ can be as large as $O(\log N)$, the polynomial decay still dominates.

		Next, we argue why the above results also suffice for mixed $p$-spin models. For mixed $p$-spin models, we need to bound
		\begin{align}
			N^{2(p_m-2)} \cdot \sum_{\substack{ \gC \subseteq \cup_{p\ge3} \sE_{N,p}; \\ \abs{\gC}=k, \abs{\partial \gC}\ge 0}} \hat{h}^{2 \abs{\partial \gC}} \prod_{p \ge 3} \E\left(\tanh^2(\gb \theta_p J/N^{(p-1)/2})\right)^{\abs{\gC}_p},
		\end{align}
		where $\abs{\gC}_p$ is the number of $p$-hyperedges in $\gC$, and $\abs{\gC} = \sum_{p\ge 3} \abs{\gC}_p$. Continuing the above expression
		\begin{align*}
			 \gb^{2k} N^{2(p_m-2)}  \sum_{\substack{ \gC \subseteq \cup_{p\ge3} \sE_{N,p}; \\ \abs{\gC}=k, \abs{\partial \gC}\ge 0}}\hat{h}^{2 \abs{\partial \gC}} \prod_{p \ge 3} \frac{\theta_p^{2\abs{\gC}_p}}{N^{(p-1)\abs{\gC}_p}}
			 & \le  \gb^{2k} N^{2(p_m-2)+k} \cdot \E_{\mvgs} \left( \sum_{p\ge3} \theta_p^2 \cdot  \frac{(\sum_{i=1}^N \gs_i)^p}{N^p p!}\right)^k \\
			 & =  N^{2(p_m-2)+k} \cdot \E_{\mvgs} \left(\gb^2 \xi\left(\frac1N \sum_{i=1}^N \bar{\gs}_i + \hat{h}^2 \right) \right)^k.
		\end{align*}

		Notice that for mixed $p$-spin models, $\xi\left(\frac1N \sum_{i=1}^N \bar{\gs}_i + \hat{h}^2 \right) \le C \cdot \left(\frac1N \sum_{i=1}^N \bar{\gs}_i + \hat{h}^2 \right)^{p_m}$ for some constant $C<\infty$, that is, mixed $p$-spin is dominated by the leading pure $p$-spin model up to a constant factor. Thus the mixed $p$-spin results can be deduced from the pure $p$ case established before.
	\end{proof}

	In the mixed $p$-spin case, there is a further roadblock that even for a fixed small-size cluster, one could have infinitely many different clusters. This is due to the flexibility over $p$ and any $p$-hyperedge can appear. In the next proposition, we will prove that for a fixed-size cluster, the total contribution of clusters with large $p$-hyperedge is still small.

	\begin{prop}\label{prop:p-reduction}
		For a mixed $p$-spin model with structure function $\xi$ and for $\gb<\gb_f(\xi)$, there exists $k<\infty$ such that
		\begin{align}
			N^{p_m-2} \cdot  \sum_{\gC: e_{\max}\ge k, |\gC|\le 4} \hat{h}^{\abs{\partial \gC}} \go(\gC)\to 0
		\end{align}
		in probability as $N\to \infty$, where $e_{\max}$ is the size of the hyperedge in $\gC$ with maximum number of vertices.
	\end{prop}
	The above results basically say that for the small-sized clusters obtained in the mixed $p$-spin setting, the leading contributions can not contain any large $p$-hyperedges. This enables us to focus on analyzing the clusters of small size and small hyperedges. This proof of this proposition basically follows from the idea that any large $p$-hyperedges can be decomposed as multiple small $p$-hyperegdes. In this way, the cluster size gets increased and we know from previous results that the large size cluster contribution is small.

	\begin{proof}[Proof of Proposition~\ref{prop:p-reduction}]
		Without loss of generality, we focus on the case $\abs{\gC}=2$, since $\abs{\gC}$ is uniformly bounded by 4, one can always carry out the proof to the other cases. Assuming $\gC$ is composed of $p_1$-hyperedge and $p_2$-hyperedge, where $p_2>p_1$. By computing the second moment of the l.h.s, we get the bound
		\begin{align*}
			N^{2(p_m-2)} \cdot  \left( \sum_{\gC: p_2\ge k} \hat{h}^{2\abs{\partial \gC}} \prod_{i=1}^2\frac{\theta_{p_i}^2\gb^2}{N^{p_i-1}}\right) .
		\end{align*}
		Notice that for the larger hyperedge, one can always decompose it into smaller hyperedges as follows. Formally, for some $t<\infty$, we have
		\begin{align}
			p_2 = \sum_{j=1}^t a_j q_j, \quad \text{where} \ a_i\in \bN,
		\end{align}
		and $\{q_j\}_{j=1}^{t}$ satisfies $q_j\ge 3$ are distinct and $\text{g.c.d}\{q_1,q_2,\cdots,q_t\} = \text{g.c.d}\{p\ge 3: \theta_p>0\}$.
		Then one can rewrite the edge weights for $p_2$-hyperedge,
		\begin{align}
			\frac{\theta_{p_2}^2 \gb^2}{N^{p_2-1}} = \prod_{j=1}^t \left( \frac{\gb^2 \theta_{q_j}^2}{N^{q_j-1}}\right)^{a_j} \cdot \frac{1}{\gb^{2\sum_j a_j -2}} \cdot \frac{\theta_{p_2}^2}{\prod_{j=1}^t \theta_{q_j}^{2a_j}} \cdot \frac{1}{N^{\sum_j a_j -1}},
		\end{align}
		Notice that under the condition $\gb<\gb_f(\xi)$ implies that the term $\gb^{2-2\sum_j a_j} \prod_{j=1}^t \theta_{q_j}^{-2a_j}$ is of constant order, we denote it as $c$. Then we have
		\begin{align}
			    & N^{2(p_m-2)} \cdot  \left( \sum_{\gC: p_2\ge k} \hat{h}^{2\abs{\partial \gC}} \prod_{i=1}^2\frac{\theta_{p_i}^2\gb^2}{N^{p_i-1}}\right) \\
			\le & N^{2(p_m-2)} \cdot \left( \sum_{\tilde{\gC}:\sum_{j=1}^ta_jq_j \ge k} \hat{h}^{2\abs{\partial \gC}} \frac{\theta_{p_1}^2\gb^2}{N^{p_1-1}} \cdot  \prod_{j=1}^t \left( \frac{\gb^2 \theta_{q_j}^2}{N^{q_j-1}}\right)^{a_j} \cdot \frac{c}{N^{\sum_j a_j -1}} \sum_{p_2\ge k} \frac{\theta_{p_2}^2 }{p_2!}\right),\label{eq:edge-decomp}
		\end{align}
		where $\tilde{\gC}$ are the hypergraphs obtained by decomposing the $p_2$-hyperedge in $\gC$ as a collection of $a_j$ number of $q_j$-hyperedges. Let \begin{align*}q_{\max}:=\max_{1\le j \le t}q_j,\end{align*}
		then it is clear that $k\le p_2 \le q_{\max}\cdot  \sum_{j=1}^t a_j$, if $k/q_{\max} -1 \ge 2$, then $\frac{c}{N^{\sum_j a_j -1}} \sum_{p_2\ge k} \frac{\theta_{p_2}^2 }{p_2!} < \infty$ combining with the mixed $p$-spin assumption~\ref{ass:mph}. On the other hand, if $4\le k/q_{\max} \le \sum_{j}a_j$, then the r.h.s.~of~\eqref{eq:edge-decomp} will converge to 0 as $N \to \infty$ by using the large graph decaying results established before. Since there are more than 4 hyperedges after decomposing the $p_2$-hyperedge as 4 smaller hyperedges. In summary, one needs to choose $k = 4 q_{\max}$.
	\end{proof}

	\subsection{Proof of Lemma~\ref{lem:betac-p}}
		First, we show that the function $\phi(e^{t})$ is strictly increasing in $t\in(-\infty,0]$. Note that
		\begin{align*}
			\phi(e^{t})=\left( \log(I(e^t)/I(1)) \right)' \text{ and } I(e^{t})/I(1)=\E(e^{tX}), t\le 0
		\end{align*}
		where $X$ is an even-integer values random variable with $$\pr(X=2k)=1/(\log2\cdot 2k(2k-1)), k\ge 1.$$
		Thus,
		\begin{align*}
			\frac{d}{dt}\phi(e^{t}) = \left( \log(I(e^t)/I(1)) \right)'' >0 \text{ for } t<0.
		\end{align*}
		In particular, for $p>2$ fixed, $\phi^{-1}(p)\in (0,1)$ is well defined.

		Now, for $\xi_{p}(x)=x^{p}/p!$, note that at $\gb=\gb_{f}(\xi_{p})$, the function $\frac12\gb^2x^p/p!-I(x)$ has two global maxima in $[0,1]$, one at $0$ and some point $u\in (0,1)$. In particular, we have
		\begin{align*}
			\frac12\gb^2u^{p}/p!=I(u) \text{ and } \frac12\gb^2u^{p-1}/(p-1)! = I'(u),
		\end{align*}
		which implies that
		\begin{align*}
			\phi(u) = \frac{uI'(u)}{I(u)}=p \text{ and } u=\phi^{-1}(p).
		\end{align*}
		Thus, we have
		$
			\gb_{f}^2/p! = {2I(u)}/{u^{p}}$ and we are done.\qed

\subsection{Proof of Theorem~\ref{thm:beta_c}}

Notice that
\begin{align*}
	\frac1N \log \bar{Z}_N(\gb,0) \to \frac{1}2 \gb^2 \xi(1) \quad \text{in probability}\quad \text{as} \ N \to \infty.
\end{align*}
For the replica symmetric solution, we need to prove that $\frac1N \log \hat{Z}_N(\gb,0) \to 0 $ in probability. This is clear from the above control of $\hat{Z}_N(\gb,0)$, where we established that the variance converges to 0 for $\gb<\gb_f(\xi)$ as $N\to\infty$. This implies that $\hat{Z}_N(\gb,0) \to 1$ in probability as $N\to \infty$. \qed

\section{Dominant Clusters for \texorpdfstring{$p$}{p}-Spin Models}\label{sec:dominant}

In the previous section, we reduced the size of clusters contributing to $\hat{Z}_N(\gb,h)$ to 4. In this section, we will identify the leading clusters in various regimes among those finite size cluster structures. The competing influences of those leading clusters naturally give rise to the multiple transition picture described in Section~\ref{sec:intro} and~\ref{sec:mixed-p}. 

In the pure $p$-spin models, the underlying hypergraph is uniform and the leading clusters can change depending on strength of the weak external field. However, in the mixed $p$-spin case, the underlying graph is a general non-uniform hypergraph, where all possible hyperedges can appear. The leading clusters now can also change with respect to $p$. This is the fundamental reason of the two parameter multiple transition summarized in Table~\ref{tab:transm}. This extra freedom on selecting any $p$-hyperedges creates further roadblocks for identifying the leading clusters. Even for a fixed-size cluster, there still can be many different cluster structures since one can always choose any $p$-hyperedges in the cluster. Although the mixed $p$-spin model looks complicated, it turns out the ideas obtained while analyzing the pure $p$-spin case are still helpful in the mixed case. We start with pure $p$-spin model results. The general strategy of those analyses can be reformulated as a discrete optimization problem over the set of possible leading clusters. 

We will use $c=((a_1,p_1),\ldots,(a_t,p_t),\ell)$ to denote a cluster structure  with $a_i$ many $p_i$ hyperedges for $i=1,2,\ldots,t$ and $\ell$ many odd-degree vertices. Let
\begin{align*}
V_c
:=\sum_{\gC\subseteq \cS_c} \hh^\ell \cdot \hat{\go}(\gC) = \hh^\ell \cdot N^{-\sum a_i(p_i-1)/2}\cdot  \sum_{\gC} 
\go(\gC)
= \hh^\ell \cdot N^{-\sum a_i(p_i-1)/2}\cdot  W_c.
\end{align*}
We have,
\begin{align*}
\hat{Z}_N-1=\sum_{c\in\sC} V_c
\end{align*}
where the sum is over all cluster structures $\sC$.
Recall that ${W}_c$ is the sum of $\hat{\go}(\gC)$ over all hypergraphs $\gC \in \cup_{p\ge 3}\sE_{N,p}$ with a given cluster structure $c$. We define 
\begin{align*}
    \cX(c):=-2\ga \ell -\sum_i a_i(p_i-1) +\frac12\left(\sum_i a_ip_i +\ell\right)
\end{align*}
as the real number such that $\var(V_c)=\Theta(N^{\cX(c)})$. Here is the general abstract framework, we need to find
$$
	\sC_\star:=\argmax_{c \in \sC}{\cX(c)}\text{ and } \gc_\star=\max_{c \in \sC}{\cX(c)}
$$
where $\sC$ is the collection of all possible cluster structures contributing to the variance of the partition function. In particular, we have
\begin{align*}
N^{\gc_\star/2}\cdot (\hat{Z}_N-1)\approx N^{\gc_\star/2}\cdot \sum_{c\in\sC_\star}  V_c.
\end{align*}

Thus the goal is to find the cluster structure(s) having the highest order of variance in each regime. We will instantiate this framework while analyzing different cases in the following subsections.

\subsection{Clusters for pure $p$-spin model at $h=0$}

We first deal with the case of $h=0$ for pure $p$-spin models. The key tool is the second-moment method. In particular, we calculate the variance order of all possible clusters and identify the dominant ones. If $h=0$, by the decomposition of $\hat{Z}_N$, we know that the effective clusters should have zero odd-degree vertices. Since the cluster size has been reduced to 4 in Section~\ref{sec:exp}, there are only a few cases left. Specifically, we find that the leading clusters for odd and even $p$ are $((4,p),0)$ and $((3,p),0)$, respectively. This discrepancy explains why even and odd $p$-spin models behave differently, as suggested in the past literature. Besides that, the extra hyperedge in odd $p$-spin case creates quite delicate structures. This can be seen from the edge intersection mechanism in Figure~\ref{fig:min-struct-odd}. Later, this subtlety will become even more apparent when we analyze the case with a weak external field and apply Stein's method to establish the CLT. Accordingly, the analysis in the odd $p$ case is more involved.

Recall that the leading structures are already given in~\eqref{str:pure-p-0}. Examples of the cluster structures are presented in Figures~\ref{fig:min-struct-odd} and~\ref{fig:min-struct-even}. Also, recall that $\ell_p$ is defined as
\begin{align}\label{def:l_p}
	\ell_p=\begin{cases}
		       4 & \text{ if $p$ is odd, } \\
		       3 & \text{ if $p$ is even.}
	       \end{cases}
\end{align}
We state this result formally in the following proposition.
\begin{prop}\label{prop:smallest-h0}
	For pure $p$-spin models with $p\ge 3$ fixed and $h=0$, we have
	\begin{align}\label{eq:pureh0-domin}
		N^{\frac{\ell_p}4(p-2)}\left( \hat{Z}_{N,4} - 1- \sum_{\gC \in \cS_{(\ell_p,p),0} }\hat{\go}(\gC) \right) \to 0
	\end{align}
	in probability as $N \to \infty$.
\end{prop}
Note that $N^{-\ell_p(p-2)/2}$ is the variance order of clusters $(\ell_p,0)$. This proposition essentially says that the cluster $(\ell_p,0)$ is the leading contribution to $\hat{Z}_N(\gb)$. Recall that the notation $\cS_{(\ell_p,0)}$ for the set of cluster structures were defined in Section~\ref{ssec:notat}.

\begin{proof}[Proof of Proposition~\ref{prop:smallest-h0}]
	Notice that $\ell_p(p-2)/4 \le (p-2)$, by Proposition~\ref{prop:second-red}, we know that the contribution of all clusters with $\abs{\gC}>4$ will converge to 0 in probability under the given rescaling. All we need to consider is the contributions from cluster sizes 3 and 4. Since given the even degree restriction, single hyperedge ($\abs{\gC}=1$) can not form a cluster with zero odd-degree vertices ($\abs{\partial \gC}=0)$. Neither do two hyperedges without forming multi-edges. In order to prove~\eqref{eq:pureh0-domin}, one needs to identify the dominant clusters in $\hat{Z}_{N,4}$. Therefore it naturally leads to the following optimization problem,
	\begin{align*}
		\argmax_{a=3,4} \cX(((a,p),0) ) \text{ such that } 2\text{ divides } ap.
	\end{align*}
	In general, the zero external field condition requires each cluster to have no odd-degree vertices. Besides that, for a fixed-size cluster, it is clear that the leading contribution would be the case where each vertex's degree is exactly 2, otherwise, the number of possible structures is much smaller thus the contribution is in subleading terms.

	Now let us compute the contribution of leading clusters. For even $p$, by the structural restriction, we know that $a=3,4$. In general, for $\abs{\gC} = a$, since the total weights on the hyperedges are fixed, one needs to count how many structures of size $a$ are. In this case, there are $N^{a p/2}$ different ways to form the $(a,0)$ cluster; equivalently, we need to choose $a p /2$ vertices among $N$. Each of such structures has a variance of order $N^{-a(p-1)}$. Therefore, the total contribution is of order $N^{-a(p-2)/2}$. It is clear that the dominant cluster is $a=3$. Therefore, under the scaling $N^{3(p-2)/4}$, the second moment of all other structures converges to 0 as $N\to \infty$. The odd $p$ case is much more simpler, since the restriction $\frac{ap}{2} \in \bN$ needs $a$ even, thus $a=4$. The variance order, in this case, is $\Theta(N^{-2(p-2)})$. The desired results follow similarly as in the even case.
\end{proof}

\subsection{Clusters for pure $p$-spin model with $h\neq 0$}

Once an external field is present, the even degree restriction on clusters is no longer needed. Although the large hypergraph decay results in Proposition~\ref{prop:first-red} and~\ref{prop:second-red} reduce the cluster size to 4, there are still a lot of possible structures by allowing odd-degree vertices. It is observed that the new clusters with odd-degree vertices must have smaller cluster sizes. This observation seems natural but actually has a bit of subtlety. Take $(3,0)$ and $(3,1)$ cluster as an example for any particular structure of those two kinds, allowing odd-degree vertices clearly decrease the variance order. However, the number of cluster structures for $(3,1)$ also increases, making it less obvious what is the concluding effect and needs more precise quantitative control. By explicit computation, we observed that the new dominant cluster has to be of a smaller size to boost the variance in order to compete with the cluster contribution at $h=0$. A similar but more quantitative issue is whether the system prefers more or fewer odd-degree vertices if there are any. More odd-degree vertices will diminish the contribution of each particular structure more significantly but also increases the number of possible structures. A similar trade-off makes the analysis quite delicate. Regarding this, it is natural to see that the strength of the external field  essentially controls how many odd-degree vertices should appear in the new leading cluster. This fundamentally characterizes the multiple transition picture in Table~\ref{tab:trans}.

Guided by the above observations,  when $\ga$ decreases from $\infty$ (corresponds to $h=0$ case), some new clusters beyond $(3,0),(4,0)$ in~\eqref{str:pure-p-0} will appear, an immediate question is when those new clusters can have same variance order as the $h=0$ clusters. Based on the above observations, we found that the new dominant clusters are $(3,1), (2,2),(1,p)$ in~\eqref{str:pure-p-h}. To find the threshold $\ga_c$, we need to compute the variance order of those new clusters. One can see in Table~\ref{tab:trans}, the clusters at $\ga=\ga_c$ for even and odd $p$ are still different. However, if the external field gets stronger for $\ga \in \cR_2$, specifically at $\ga = \frac12$,\ie~the second threshold, the dominant clusters become unified for both even and odd $p$. Finally, in the second unified regime $\cR_3$, the leading cluster becomes even simpler and just a single hyperedge. As it indicates,  the external field gets stronger and stronger, the cluster size becomes smaller and smaller.

To explain the details, we will use the notation $\cC(\ga)$ to represent the set of leading cluster structures for a given $\ga$, as listed in Table~\ref{tab:trans}:
\begin{align}
\cC(\ga):=\{\text{cluster structures in the last column of Table~\ref{tab:trans} for given $\ga$}\}.
\end{align}
To compute the first threshold $\ga_c$, we follow the insights obtained in the SK ($p=2$) case~\cite{DW21}, where the transition is very simple and only two clusters appear and compete with each other. The threshold $\ga_c$ is the point where the contributions from two clusters achieve a balance. In the $p$-spin case, since the cluster structures are more involved, we need to first identify all possible dominant clusters and then compute their contributions. Depending on the strength of external fields, the final step is to find when those leading clusters' contributions are in the same order.

We summarize the main results in the following Proposition. The proof details will give a formal explanation for the above heuristic discussions.

\begin{prop}\label{prop:critical}
	Let $p\ge 3$ and $\ga \ge \quar$ be fixed. We have
	\begin{align}
		N^{\gc}\left(  \hat{Z}_{N,4} -1- \sum_{c \in \cC(\ga)}\sum_{\gC \in \cS_{c}}\hat{\go}(\gC) \hat{h}^b \right) \to 0
	\end{align}
	in probability as $N \to \infty$, where $\gc$ depending on $\ga$ is as given in the third column of Table~\ref{tab:trans}.
\end{prop}

\begin{proof}[Proof of Proposition~\ref{prop:critical}]
	Let us start with instantiating the optimization framework first. Recall $\gc_i\le (p-2)$ uniformly for $i=1,2,3$, for same reason, it suffices to consider the clusters of size $a\le 4$.
	As we did in the $h=0$ case, now consider the following optimization problem
	\begin{align*}
		\argmax_{\substack{1\le a \le 4,\\ a+b>2, \\ 0\le b \le ap}} \cX\left(((a,p),b)\right) \text{ s.t. } 2\mid(ap + b), a>0,  b\ge 0.
	\end{align*}
	We briefly explain all those constraints. By allowing an external field, the degree of vertices can be both even and odd. Thus the cluster size $a$ can be any integer from 1 to 4 inclusive without creating multi-edges. The number of odd-degree vertices $b$ can be ranging from 0 to $ap$. As in the $h=0$ case, notice that the possible leading clusters must have the following property. Whenever the degree is even, it must be exactly 2. Whenever the degree is odd, it must be exactly 1. Because for each fixed $(a,b)$, in this way the number of possible structures is maximal and gives dominant contributions. Thus there are $N^{(ap-b)/2}$ different ways to choose degree 2 vertices, and $N^{b}$ different ways to choose the degree one vertices. This explains the restriction $2\mid ap+b$.

	For $p$ even, it is clear that $b$ needs be even by the restriction $\frac{ap+b}{2}\in\bN^+$. Similarly, for $p$ odd, it needs $a,b$ to be even or odd in the meantime. In general, for fixed $(a,b)$ in a given regime $\ga \in \cR_{i}$, the variance order of the corresponding structure is $N^{(ap+b)/2} \cdot N^{-a(p-1)-2b\ga} =N^{-\frac a2 (p-2)- \frac b2(4\ga -1)} $.

	First, let us consider the case $\ga = \ga_c(p)$. Recall the definition of $\ga_c$ in~\eqref{eq:alphac}, for even $p$, the variance is of order $N^{-(p-2)(\frac a2  +\frac b4)}$. In this case, the restriction is $a=1,2,3,4$ and $b$ is a non-negative integer. It is not hard to see that clusters $(2,2),(3,0)$ have the largest variance. For odd $p$ case, the variance is of order $N^{-(p-2)(\frac a2  +\frac b2)}$. Now the restriction becomes $a=1,2,3,4$ and $b\in\bN$ are even or odd at the same time. The leading clusters, in this case, are $(4,0),(2,2),(3,1)$.
	Furthermore, if $\ga > \ga_c$, the new leading clusters' contribution will be in subleading order. The dominant clusters are the same as the $\ga = \infty$ case. The variance order in this regime is $\Theta(N^{-\ell_p(p-2)/2})$.

	Now we analyze the case $\ga = \frac12$ similarly. For even $p$, the variance is of order $N^{-a(p-2)/2-b/2}$. The even $b$ restriction gives the leading clusters $(2,2),(1,p)$. For odd $p$, the restriction becomes $a,b$ are both even or odd. The leading clusters are also $(2,2),(1,p)$. If $\ga \in (\frac12, \ga_c)$, since the variance order in general is $N^{-\frac a2 (p-2)- \frac b2(4\ga -1)}$, it is clear that
	\begin{align*}
		-(p-2)-(4\ga -1) >- \frac12 (p-2) -\frac p2(4\ga -1)
	\end{align*}
	for $\ga \in(\frac12, \ga_c)$. The l.h.s and r.h.s.~are the exponent of the variance of $(2,2),(1,p)$ clusters respectively. One can also perform similar computations for the other cases. In summary, the leading cluster for $\ga \in (\frac12, \ga_c)$ is $(2,2)$ for both even and odd $p$. The variance order, in this case, is $N^{3-4\ga - p}$.

	The analysis in the regime $\ga \in [\frac14, \frac12)$ is very similar to the previous case. The leading cluster is $(1,p)$, and the variance order is $N^{1-2\ga p}$. The computation details are very similar to the previous cases and are left to interested readers.
\end{proof}

\subsection{Clusters for mixed $p$-spin model with $h=0$}
\label{ssec:mixed-p0-dom}
In the pure $p$-spin case, one needs to identify the leading cluster structures on complete $p$-uniform hypergraphs, where all hyperedges are of $p$-tuples for fixed $p\ge3$. In the mixed $p$-spin case, we need to deal with the general hypergraph where all possible $p$-hyperedges can appear. This additional flexibility over $p$ creates a new transition with respect to $p$ and further challenges in the analysis. As we discussed, although the results in Proposition~\ref{prop:first-red} and~\ref{prop:second-red} reduce the size of possible leading clusters to 4 or less, there still can be infinitely many different structures even for a fixed-size cluster. For example, in $h=0$ case, consider the clusters of size $\abs{\gC}=4$, the 4 hyperedges technically can be any $p_1,p_2,p_3,p_4\ge 3$ tuples with total degree $\sum_{i=1}^4p_i$ even. To get over this issue, it basically needs a similar type of decaying result with respect to $p$ as dealing with the external field in Proposition~\ref{prop:first-red} and~\ref{prop:second-red}.

We introduce the following strategy to deal with it. Notice that for a fixed finite size cluster, if there is any large $p$-hyperedge, one can always decompose it into several smaller hyperedges. For example, a $9$-hyperedge can be decomposed into a $4$-hyperedge and a $5$-hyperedge. In this way, the cluster size would increase. By the results in Proposition~\ref{prop:second-red}, any clusters with a size greater than 4 are in subleading order. Thus we need to prove that any finite-size cluster with a large $p$-hyperedge must be in subleading order. The formal statement of this result is presented in Proposition~\ref{prop:p-reduction}. 

With the above result, it suffices to consider clusters that have finitely many hyperedges, and those hyperedges are finite $p$-tuples. Notice that the variance order of cluster structures has negative dependence on $p$ in the exponent, thus it suffices to consider to minimum effective even and odd $p$-spins $p_e,p_o$ for identifying the leading clusters in the mixed setting.  The rest of the analysis will be similar to the pure $p$-spin case, except that the cluster structures are now more complicated due to the extra variation on $p$. The analysis has to be very careful. We start analyzing the zero external field case in this part. Before stating the main results, we recall some necessary notations,
\begin{align*}
	p_e = \min\{p\ge 3, p \ \text{even and} \ \theta_p\neq 0\}, \  p_o = \min\{p\ge 3, p \ \text{odd and} \ \theta_p \neq 0\}, \  \ p_m = \min\{p_e, p_o\}.
\end{align*}
The associated cluster structure for $\ga = \infty$ is
\begin{align*}
	\cC_{\textup{mix},0} := \{  \text{the cluster structures in the last column of Table~\ref{tab:transm} for $\ga = \infty$}\}.
\end{align*}
In the mixed setting, we use the 3-tuples $(a_e,a_o,b)$ to characterize the cluster structure (check the notation in Section~\ref{ssec:notat}). Note that $a_e,a_o,b$ represent the number of $p_e$-hyperedges, $p_o$-hyperedges, and the number of odd-degree vertices in the associated cluster. The scaling exponent $\gc_1$ is
\begin{align*}
	\gc_1 = \begin{cases}
		        \frac34 (p_e-2),                     & \text{ if } p_e<p_o, \\
		        \frac12 p_o + \frac14 p_e - \frac32, & \text{ if } p_o<p_e<2p_o, \\
		        (p_o-2),                             & \text{ if } p_e\ge 2p_o.
	        \end{cases}
\end{align*}

Here is the formal statement of the results. The proof strategy is similar to the pure $p$-spin case by instantiating the optimization framework.
\begin{prop}\label{pro:mix-zero}
	Let $p_m \ge 3$ be fixed and $h=0$. We have
	\begin{align}
		N^{\gc_1} \cdot \left(\hat{Z}_{N,4}-1 - \sum_{c \in \cC_{\textup{mix},0}}\sum_{\gC \in \cS_{c}}\hat{\go}(\gC) \right) \to 0
	\end{align}
	in probability as $N \to \infty$.
\end{prop}

\begin{proof}
	Following a similar strategy as before, first, by the definition of $\gc_1$, it is easy to see that $\gc_1\le p_m-2$ for $p_m\ge 3$. Thus by Proposition~\ref{prop:second-red}, it suffices to consider cluster sizes of 3 and 4. By Proposition~\ref{prop:p-reduction}, one needs to consider the clusters with $p$-hyperedges with finite $p$. Notice that the variance component has negative dependence on $p$, thus it suffices to focus on the clusters with $p_e,p_o$-hyperedges. In particular, we consider the following optimization problem,

	\begin{align*}
	   \argmax_{\substack{a_e+a_o=3,4; \\a_e,a_o\ge 0.}} \cX\left(((a_e,p_e),(a_o,p_o),0)\right) \text{ s.t. }  \frac{a_ep_e+a_op_o}{2} \in \bN^+.
	\end{align*}

	Let us first explain the constrains above. The size $3,4$ restriction is due to simple sub-hypergraphs condition with even degree vertices. The constraints $\frac{a_ep_e+a_op_o}{2} \in \bN^+$ is by a similar reason as before, \ie~ the possible dominant clusters enjoy the property that every vertex has degree 2, otherwise, it's in subleading order. From this condition, it is easy to notice that $a_o$ can not be odd.

	We first consider $a_o=0$, then $a_o+a_e = 3,4$ implies $a_e=3,4$, a simple calculation of the variance order indicates that $a_e=3,a_o=0$ is one possible dominant cluster, its variance is of order $N^{3p_e/2}\cdot N^{-3(p_e-1)} = N^{-3(p_e-2)/2}$. On the other hand, for $a_o=2$, similarly we get $a_e = 1,2$, and clearly $a_e=1$ dominants, thus $a_e=1,a_o = 2$ is another possible dominant cluster, the variance is $N^{(p_e+2p_o)/2}\cdot N^{-(p_e-1)-2(p_o-1)} = N^{-p_o-p_e/2+3}$. The last possible dominant cluster is $a_e=0,a_o=4$, whose variance is of order $N^{2p_o} \cdot N^{-4(p_o-1)} = N^{-2(p_o-2)}$. For all those 3 possibilities, we identify the dominant one among various regimes depending on the relation of $p_e,p_o$.

	If $p_e<p_o$, it is clear that \begin{align*}\frac{-3(p_e-2)}{2} +p_o+\frac{p_e}{2} -3 = -\frac{p_e}{2} +p_o>0,\end{align*}
	and similarly \begin{align*}\frac{-3(p_e-2)}{2} +2(p_o-2) =2p_o-\frac{3p_e}{2}-1 \ge 2p_e+2 - \frac{3p_e}{2}-1>0 . \end{align*}
	Therefore, the dominant cluster in this regime is the $(3,0,0)$ cluster. Similarly one can identify the leading clusters in the other three cases: $p_o<p_e<2(p_o-1), p_e = 2(p_o-1), p_e\ge 2p_o$. In the case $ p_o<p_e<2(p_o-1)$, the dominant cluster is $(1,2,0)$. While for $p_e\ge 2p_o$, the dominant cluster is $(0,4,0)$. One can regard $p_e = 2(p_o-1)$ as the transition point with respect to $p$, where clusters $(1,2,0),(0,4,0)$ coexist.
\end{proof}

\subsection{Clusters for mixed $p$-spin model with $h\neq 0$}
\label{ssec:mixed-p-dom}
The analysis in the $h=0$ case for the mixed $p$-spin model is already delicate, we now turn to a more challenging task, which is analyzing the transitional behavior by adding some weak external field. As indicated in Table~\ref{tab:transm}, the transitional behavior for mixed $p$-spin case becomes quite complicated. Thanks to the Proposition~\ref{prop:p-reduction}, which allows us to focus on the clusters with finite size hyperedges. Despite this, there are still many possible combinations to form different leading clusters when an external field is present.

However, with the important observations in the pure $p$-spin case and the results for mixed $p$-spin at $h=0$, we are still able to derive all the leading clusters by careful analysis. To make life easier, we introduce the following set to denote the dominant clusters in various regimes. We will use the notation $\cC_{\text{mix}}(\ga)$ to represent the set of leading cluster structures for a given $\ga$, as listed in Table~\ref{tab:transm}:
\begin{align}
\cC_{\text{mix}}(\ga):=\{  \text{the cluster structures in the last column of Table~\ref{tab:transm} for a given $\ga$.}\}
\end{align}
One can always check the cluster structures listed in Table~\ref{tab:transm}. Now we state the main results.

\begin{prop}\label{pro:mix-second}
	Let $p_m \ge 3$ be fixed and $\ga \in[\quar,\infty)$.  We have
	\begin{align}
		N^{\gc}\left(\hat{Z}_{N,4}-1 - \sum_{c \in \cC_{\text{mix}}(\ga)}\sum_{\gC \in \cS_{c} }\hat{h}^{b}\hat{\go}(\gC) \right) \to 0
	\end{align}
	in probability as $N \to \infty$, where $\ga$ depending on $\ga, p_e,p_o$ is as given in the third column of Table~\ref{tab:transm}.
\end{prop}

\begin{proof}
	One can easily verify the fact that $\gc_i \le p_m-2$ for $i=1,2,3$, thus it suffices to focus on the clusters of size $a_e+a_o$ no greater than 4. Due to the effect of external field, now the even degree restriction no longer holds. In particular, we have the following instance of the optimization problem to identify dominant clusters in different regimes.

	\begin{align*}
		\argmax_{\substack{a_e+a_o=1,2,3,4; \\a_e,a_o\ge 0,  b\ge0.}} \cX\left(((a_e,p_e),(a_o,p_o),b)\right) \quad \text{s.t.} \quad \frac{a_ep_e+a_op_o+b}{2} \in \bN^+.
	\end{align*}
	Let us first briefly explain the constraints. Notice the size of clusters is $\abs{\gC} = a_e+a_o$, \ie the sum of the number of $p_e$-hyperedges and $p_o$-hyperedges. For the same reason as in the proof of Proposition~\ref{prop:p-reduction}, we focus on the clusters with only $p_e,p_o$ hyperedges. Due to the possible presence of a weak external field, the size of the cluster can be $\{1,2,3,4\}$. On the other hand, the possible dominant clusters must have the property that every even-degree vertex has degree exactly 2, and odd-degree vertices must have degree 1. This was argued in the previous proofs. Therefore, it needs $\frac{a_ep_e+a_op_o+b}{2}\in \bN^+$. From this restriction, one can easily see that $b$, the number of odd-degree vertices, must be even or odd as $a_o$ in the meantime.

	Note that in general the variance order of $(a_e,a_o,b)$ cluster is
	\begin{align}\label{eq:tmp-varorder}
		N^{\frac{a_ep_e+a_op_o+b}{2}} \cdot N^{-a_e(p_e-1)-a_o(p_o-1) - 2b\ga} = N^{-\frac{a_e(p_e-2)}{2}-\frac{a_o(p_o-2)}{2}-\frac{b(4\ga-1)}{2}}.
	\end{align}

	Let us first analyze the regime $\cR_1$. In particular, for $\ga = \ga_c(p_e,p_o)$, recall the definition of $\ga_c$ in~\eqref{eq:alphac-mix},
	\begin{align*}
		\ga_c =
		\begin{cases}
			p_e/8,     & \text{ if } 2<p_e<p_o-1, \\
			(p_e-2)/4, & \text{ if } 2<p_e = p_o - 1, \\
			p_e/8,     & \text{ if } p_o<p_e <2( p_o - 1), \\
			(p_o-1)/4, & \text{ if } p_e\ge 2(p_o-1).
		\end{cases}
	\end{align*}
	We drop the dependence on $p_e,p_o$ for convenience if not causing any confusion. If $p_e<p_o-1$, then $\ga_c = p_e/8$. Plugging into~\eqref{eq:tmp-varorder}, the variance order for general $(a_e,a_o,b)$ cluster is $N^{-(\frac{a_e}{2}+\frac b4)(p_e-2) - \frac{a_o}{2}(p_o-2)}$. Recall the observation $a_o$ and $b$ must be even or odd in the meantime. It's easy to see that for fixed $b$, by the condition $p_e<p_o-1$, we know that the dominant cluster prefers $p_e$-hyperedges in order to maximize variance. This implies that the dominant cluster should have $a_o=0$. This reduces the problem to the pure $p_e$-spin model case, as we did, the dominant clusters at $\ga = \ga_c$ should be $a_e=2, b=2$ and $a_e =3, b=0$.

	Now if $p_e = p_o-1$, in this case $\ga_c = (p_e-2)/4$. The variance in~\eqref{eq:tmp-varorder} becomes $N^{-\frac{a_e}{2}(p_e-2)-\frac{a_o}{2}(p_o-2) - \frac{b}{2}(p_e-3)}$. Since $p_e<p_o$, clearly $a_e=3, a_o=0,b=0$ is one dominant cluster. The variance order is $N^{-\frac{3(p_e-2)}{2}}$. One can also check that $a_e=1,a_o=1,b=1$ is another cluster with variance order $N^{-\frac{3(p_e-2)}{2}}$. This implies that all clusters with $a_e+a_o>2, b>1$ are in subleading order. To find other possible dominant clusters, it needs to decrease either $a_e+a_o$ or $b$. However, by the structural property, we know $a_e=1,a_o=1,b=0$, and $a_e+a_o=1,b=1$ can not happen.

	If $p_e =p_o+1$, then $\ga_c = p_e/8$. The variance order is $N^{-(\frac{a_e}{2}+\frac b4)(p_e-2) - \frac{a_o}{2}(p_o-2)}$. Similarly we know one dominant cluster without odd-degree vertices ($b=0$) is $a_o=2,a_e=1,b=0$ with variance order $N^{-p_o+3-p_e/2}$. It is easy to find that the variance order of the cluster $a_e=1,a_o=1,b=1$ is also $N^{-p_o+3-p_e/2}$. Similarly as in previous case, we know that other possible dominant clusters are $a_e=1,a_o=1,b=0$ and $a_e+a_o=1,b=1$, which can not appear due to the hypergraph structural restrictions.

	If $p_o+1<p_e<2(p_o-1)$, $\ga_c$ is still $p_e/8$. This case is very similar to the last case. The zero odd vertices dominant cluster is still $a_o=2,a_e=1,b=0$. While the cluster $a_e=1,a_o=1,b=1$ can not appear, since now the difference $p_e-p_o\ge 2$ implies $b \ge 2$. Still by $p_e>p_o+1$, it's natural to check $a_e=0,a_o=2,b=2$ cluster. The variance is indeed dominant one $N^{-p_o-p_e/2+3}$.

	If $p_e = 2(p_o-1)$, then $\ga_c = (p_o-1)/4$. In general the variance is of order $N^{-\frac{a_e}{2}(p_e-2) - \frac{a_o+b}{2}(p_o-2)}$. Clearly since $p_e>p_o$, one dominant cluster is $a_o=4,a_e=0,b=0$. The variance order is $N^{-p_o+3-p_e/2} = N^{-2(p_o-2)}$. Note that any clusters of size 4 with odd-degree vertices are in subleading order. To find other possible dominant clusters, it needs to decrease the size of cluster. For size 3 cluster with no odd-degree vertices, the only choice is $a_e=1,a_o=2,b=0$. It can be seen that the variance is $N^{-p_o+3-p_e/2} = N^{-2(p_o-2)}$. By allowing odd-degree vertices, consider $b=1$, recall $b$ and $a_o$ must be even together, thus $a_o=3,a_e=0,b=1$ is another dominant cluster with the same variance order. Similarly one can find the last dominant cluster in this case $a_e=0,a_o=2,b=2$.

	Finally for $p_e\ge 2p_o$, $\ga_c$ stays the same as the last case. The only thing changing is that $p_e\ge 2p_o$ make the cluster $a_e=1,a_o=2,b=0$ impossible. The rest of analysis are exactly same as the previous case.

	Eventually, when $\ga>\ga_c$, it's clear that all the new clusters with odd-degree vertices are in subleading order. Thus the dominant clusters will be same as $\ga = \infty$ case.

	In a similar fashion, one can analyze the other regimes, such as $\ga=\frac12$ and $\frac14$. We will not replicate all the details here. The final dominant clusters in each regime are as shown in Table~\ref{tab:transm}.
\end{proof}

\section{CLT using Stein's Method}\label{sec:stein}

After reducing $\hat{Z}_N$ to contribution from finite sized clusters, we apply the multivariate Stein's method to establish the joint central limit theorem. The computation details for using Stein's method are presented below.

In this section, we analyze the case of pure $p$-spin glass models.  Fix $p\ge 3$. The case $p=2$ has already been analyzed in~\cite{DW21}. For simplicity, we suppress the dependence on $p,N$ in the computations below, unless needed to avoid confusion.

Note that, we only have to consider contributions from clusters $\gC$, with $(\abs{\gC},\abs{\partial\gC})$ in the set
\begin{align}
	\sS_{p}:= \begin{cases}
		          \{(3,0),(2,2),(1,p)\}       & \text{ if $p$ is even} \\
		          \{(4,0),(3,1),(2,2),(1,p)\} & \text{ if $p$ is odd}.
	          \end{cases}
\end{align}
Before stating the main theorem, let us first fix some notations. Recall that
\begin{align}
	\go_{\mvi}:=N^{(p-1)/2}\tanh(\gb J_{\mvi}\cdot N^{-(p-1)/2})\approx \gb J_{\mvi},\qquad \mvi\in \cN_{p}
\end{align}
are i.i.d.~with mean 0 and variance
\begin{align}
	\gb_{N,p}^2 = N^{(p-1)}\E \tanh^2(\gb J/N^{(p-1)/2}) \approx \gb^2.
\end{align}
Moreover,
\begin{align*}
	\go(\gC)=\prod_{\mvi\in\gC}\go_{\mvi}
\end{align*}
for a sub-hypergraph $\gC\subseteq \sE_{N,p}$. Given two integers $k\ge 1,\ell\ge 0$, with $kp\equiv \ell\mod 2$, we define
\begin{align*}
	\cS_{k, \ell}:=\{\gC\subseteq \sE_{N,p}\mid \abs{\gC}=k, \abs{\partial\gC}=\ell\}
\end{align*}
as the set of all sub-hypergraphs with $k$ many hyperedges and $\ell$ many odd-degree vertices. We also define
\begin{align*}
	t_{k, \ell}:= (kp+\ell)/2\in \dN,
\end{align*}
so that $\abs{\cS_{k, \ell}}=\Theta(N^{t_{k, \ell}})$. Note that, the maximum number of vertices present in a subgraph $\gc\in \cS_{k, \ell}$ is $\ell+(kp-\ell)/2=(kp+\ell)/2$. We also define
\begin{align*}
	\cS_{k, \ell}^{-\mvi}:=\{\gC\setminus\{\mvi\}\mid \mvi\in \gC\in \cS_{k, \ell}\}
\end{align*}
as the set of sub-hypergraphs $\Pi$ with $(k-1)$ hyperedges such that adding the hyperedge $\mvi$ to $\Pi$ makes it an element in $\cS_{k, \ell}$. Note that
\begin{align*}
	\abs{\cS_{k, \ell}^{-\mvi}} = k\abs{\cS_{k, \ell}}/N_{p} \text{ for all } \mvi.
\end{align*}

Now, we consider the random variables
\begin{align*}
	W_{k, \ell} := \sum_{\gC\in \cS_{k, \ell}} \go(\gC).
\end{align*}
Note that, we have $\E W_{k, \ell} =0$ and
\begin{align*}
	\nu_{k, \ell}^2:=\var(W_{k, \ell}) = \abs{\cS_{k, \ell}} \cdot \gb_{N,p}^{2k}.
\end{align*}

\begin{lem}[Limiting variance] \label{lem:varlim}
	For $k\ge 1, \ell\ge 0$, we have
	\begin{align*}
		u_{k,\ell}(p)^{2}:=\lim_{N\to\infty}\abs{\cS_{k,\ell}} \cdot N^{-(kp+\ell)/2} = \frac{1}{k!\cdot \ell!\cdot p!^{k/2}}\E\left(H_{\ell}(\eta)\cdot H_{k}\left(H_{p}(\eta)/\sqrt{p!}\right)\right),
	\end{align*}
	where $\eta\sim\N(0,1)$. In particular,
	\begin{align*}
		u_{3,0}(p)^{2}=\frac1{3!\cdot (p/2)!^3},\
		u_{3,1}(p)^{2}=\frac{1}{(p+1)\cdot ((p-1)/2)!^{3}} ,\
		u_{2,2}(p)^{2}= \frac{p}{2\cdot p!},\
		v_{1,p}(p)^{2}=\frac1{p!}.
	\end{align*}
\end{lem}
\begin{proof}
The proof relies on two key ingredients. First we rewrite $\abs{\cS_{k,\ell}}$ as an expectation of the product of spin variables. Recall $\cS_{k,\ell}$ represents the number of cluster structures with $k$ hyperedges and $\ell$ odd-degree vertices. We have
\begin{align}
\abs{\cS_{k,\ell}} = \frac{1}{k!\ell!}\sum_{\substack{\mvi_1, \mvi_2, \ldots, \mvi_k \text{ distinct} \\  j_1,\cdots, j_\ell \text{ distinct}}}  \E_{\mvgs, \mvgt} \left( \hat{\mvgs}_{\mvi_1} \cdots \hat{\mvgs}_{\mvi_k} \gt_{j_1} \cdots \gt_{j_\ell}\right),
\end{align}
where for each $r = 1,2,\cdots, k$, $\hat{\mvgs}_{\mvi_r} := \prod_{s=1}^p \gs_{\mvi_r(s)}$ for $\mvgs_{\mvi_r} \in \{-1,+1\}^p$, and the expectation $\E_{\mvgs, \mvgt}$ is with respect to all $(\gs_{i}, \gt_i, 1\le i\le N)$ are i.i.d.~uniformly distributed over $\{+1,-1\}$. It is clear that whenever there is a parity between the sites of $\mvgt$ and $\mvgs_{\mvi_1}, \cdots \mvgs_{\mvi_k}$ and all the rest of sites in $\mvgs$ appear even number of times, the contribution in the expectation will be 1, otherwise, the contribution will be zero. The next key step is to connect the r.h.s.~with the Hermite polynomial. We note the following fact:
\begin{align}\label{eq:bell-poly}
&\frac{p!}{N^{p/2}} \sum_{1\le j_1< \cdots j_p\le N} x_{j_1} \cdots x_{j_p} \notag\\
&\qquad = \frac{(-1)^p}{N^{p/2}} B_p(-\sqrt{N}\tilde{x}, -N, -2!\sqrt{N}\tilde{x}, \cdots, -(p-1)!N\cdot (N\tilde{x})^{\frac12 1_{p \ \text{odd}}} )
\end{align}
where $B_p(\cdot)$ is the $p$-th complete exponential Bell polynomial~\cite{andr84}, and $\tilde{x} = \frac{1}{\sqrt{N}} \sum_{i=1}^N x_i$. As $N \to \infty$, the r.h.s. tends to $B_p(x,-1,0,0,\cdots, 0) = H_p(x)$, where $H_p(x)$ is Hermite polynomial with $x:= \lim_{N\to \infty} \tilde{x}$. Using this fact, we have 
\begin{align*}
 \abs{\cS_{k,\ell}} \cdot N^{-(kp+\ell)/2} &= \frac{1}{N^{(kp+\ell)/2}} \E_{\mvgs, \mvgt} \sum_{\substack{\mvi_1, \mvi_2, \ldots, \mvi_k, \\ 1\le j_1<\cdots < j_\ell \le N}} \hat{\mvgs}_{\mvi_1} \cdots \hat{\mvgs}_{\mvi_k} \gt_{j_1} \cdots \gt_{j_\ell} \\
 & = \frac{1}{k! p!^{k/2}} \left(\frac{-1}{N^{p/2}}\right)^{k} B_k(-N^{p/2} \sqrt{p!}\tilde{\mvgs}_{\mvi_1}, \cdots) \cdot \frac{1}{\ell!} \frac{(-1)^{\ell}}{N^{\ell/2}} B_{\ell}(-\sqrt{N} \tilde{\mvgt}, \cdots )
\end{align*}
where $\tilde{\mvgs} := \frac{1}{N^{p/2}} \sum_{\mvi} \hat{\mvgs}_{\mvi}$ and $\tilde{\mvgt} := N^{-\ell/2} \sum_{1\le j_1<\cdots j_{\ell} \le N} \gt_{j_1} \cdots \gt_{j_\ell}$. Following the last step, one can also use the fact in~\eqref{eq:bell-poly} to rewrite $\tilde{\mvgs}_{\mvi}$ in terms of $p$-th Bell polynomial, then let $N\to \infty$, by Central limit theorem, we have the r.h.s.~converges to 
\begin{align*}
\E_{\eta} \frac{1}{k! \cdot \ell! \cdot p!^{k/2}} \left( H_k\left(\frac{1}{\sqrt{p!}} H_p(\eta)\right) \cdot H_\ell(\eta)\right) \quad \text{where} \ \eta \sim \N(0,1).
\end{align*}
This completes the proof.
\end{proof}

We also define
\begin{align*}
	\widehat{W}_{k, \ell} := W_{k, \ell}/\nu_{k, \ell}
\end{align*}
as the cantered and scaled version of $W_{k, \ell}$ with mean zero and variance one. We define the random vectors
\begin{align*}
	\mvW:= (W_{k, \ell})_{(k,\ell)\in \sS_{p}} \text{ and } \widehat{\mvW}:= (\widehat{W}_{k, \ell})_{(k,\ell)\in \sS_{p}}.
\end{align*}
Note that, for $p$ even there are three co-ordinates and for $p$ odd there are four co-ordinates in $\mvW$. It is easy to check that the variance-covariance matrix of $\widehat{\mvW}$ is the identity matrix. We will work with $\widehat{W_{k\ell}}$ for simplicity and scale them appropriately at the end. Here is the formal statement of using multivariate Stein's method to prove the normal limit. 

\begin{thm}\label{thm:stein-main}
For the $\abs{\sS_p}$-dimensional random vector $\mvW= (W_{k, \ell})_{(k,\ell)\in \sS_{p}}$ with mean and variance structure $\gS$ specified above in the Lemma~\ref{lem:varlim}. For any thrice differentiable function $f$, we have 
\begin{align*}
\abs{\E f(\mvW) - \E f(\gS^{1/2}\mvZ)} \le C \left( \abs{f}_2 N^{-p} + \abs{f}_3N^{-p/2}\right),
\end{align*}
where $C$ is some constant depending only on $\gb$ and the distribution of disorder $J_{12}$.
\end{thm}

Now we apply Theorem~\ref{thm:rr-mvstein} to prove Theorem~\ref{thm:stein-main}. To do that, we will use the following Glauber dynamics. First we select  a $p$-hyperedge ${\mvI}$ from $\sE_{N,p}$ uniformly at random, then replace $\go_{\mvI}$ by an i.i.d.~copy $\go_{\mvI}'$. We denote the corresponding new random variables $W_{k, \ell}, \mvW$ as $W'_{k, \ell},\mvW'$, respectively. Define, $\gD W_{k, \ell}:=W_{k, \ell}'-W_{k, \ell}$ and $\gD\mvW:=\mvW'-\mvW$.

First we check the linearity condition. Define
\begin{align*}
	\gl_{k, \ell}:=k\cdot N_{p}^{-1},\quad \text{ where } N_{p}:=\binom{N}{p}.
\end{align*}
\begin{lem}[Linearity of conditional expectation]\label{lem:lin}
	We have
	\begin{align*}
		\E\left(\gD\mvW\mid \mvW \right) = -\gL\mvW
	\end{align*}
	where $\gL=\diag(\gl_{k, \ell}, (k,\ell)\in \sS_{p})$.
\end{lem}
\begin{proof}
	It is enough to check the linearity condition for any $(k,\ell)\in\sS_{p}$. Using the fact that $\go_{\mvi}'$ has mean zero, we have
	\begin{align*}
		\E( \gD W_{k, \ell} \mid \mvW) & = \E\left( (\go_{\mvI}' - \go_{\mvI})\cdot  \sum_{\Pi \in \cS_{k, \ell}^{-\mvI}} \go(\Pi) \ \Biggr|\ \mvW\right) \\
		                               & = -\frac{1}{N_p} \sum_{\mvi} \go_{\mvi} \sum_{\Pi \in \cS_{k, \ell}^{-\mvi}({\mvi})} \go(\Pi) = - \frac{k}{N_p}W_{k, \ell} = -\gl_{k, \ell}W_{k, \ell}.
	\end{align*}
	Here, we use the fact that any subgraph $\gC\in\cS_{kl}$ will appear exactly $k$ times in the sum.
\end{proof}

Thus, to use Theorem~\ref{thm:rr-mvstein} we need to upper bound the variance of the conditional moment and the absolute third moment of $\gD \mvW$.
First, we proceed towards analyzing the variance of the conditional second moment of $\gD W_{k, \ell}$. We observe that
\begin{align*}
	\E \abs{\gD W_{k, \ell}}^2 = 2\gl_{k, \ell}\cdot \nu_{k, \ell}^2 = 2k\cdot N_{p}^{-1}\cdot \gb_{N,p}^{2k}\abs{\cS_{k, \ell}}.
\end{align*}
Define, the i.i.d.~random variables
\begin{align}
	\xi_{\mvi}:= \go_{\mvi}^2-\gb_{N,p}^2,\quad \mvi\in \sE_{N,p}.
\end{align}
Note that,
\begin{align*}
	\E(\xi_{\mvi})=\E(\xi_{\mvi}\go_{\mvi})=0.
\end{align*}
We have
\begin{align*}
	\abs{\gD W_{k, \ell}}^2
	 & = (\abs{\go_{\mvI}'}^2-\gb_{N,p}^2-2\go_{\mvI}\go_{\mvI}' + \xi_{\mvI} + 2\gb_{N,p}^2)\cdot \left(\sum_{\Pi \in \cS_{k, \ell}^{-\mvI}} \go(\Pi)\right)^2.
\end{align*}
Thus, we can write
\begin{align*}
	 & N_{p}\cdot \left(\E\left( \abs{\gD W_{k, \ell}}^2 \ \biggr| \  (\go_{\mvi})_{\mvi\in\sE_{N,p}}\right) - \E \abs{\gD W_{k, \ell}}^2\right) \\
	 & = \sum_{\mvi\in\sE_{N,p}} (\xi_{\mvi} + 2\gb_{N,p}^2)\cdot \left(\sum_{\Pi \in \cS_{k, \ell}^{-\mvi}} \go(\Pi)\right)^2 - 2k\gb_{N,p}^{2k}\abs{\cS_{k, \ell}} \\
	 & =A_{k, \ell}+2\gb_{N,p}^2\cdot B_{k, \ell}+2C_{k, \ell}+4\gb_{N,p}^2\cdot D_{k, \ell}
\end{align*}
where
\begin{equation}
	\begin{aligned}\label{abcd}
		A_{k, \ell}      & := \sum_{\mvi } \xi_{\mvi} \sum_{\Pi \in \cS_{k, \ell}^{-\mvi}}\go(\Pi)^2,
		                 &                                                                                              & B_{k, \ell}:=  \sum_{\mvi} \sum_{\Pi \in \cS_{k, \ell}^{-\mvi}} (\go(\Pi)^2 - \gb_{N,p}^{2(k-1)}), \\
		C_{k, \ell}      & := \sum_{\mvi} \xi_{\mvi}\sum_{\Pi\neq \Pi' \in \cS_{k, \ell}^{-\mvi}}\go(\Pi_1) \go(\Pi_2),
		\quad \text{and} &                                                                                              & D_{k, \ell}:=\sum_{\mvi}\sum_{\Pi\neq \Pi' \in \cS_{k, \ell}^{-\mvi}} \go(\Pi_1)\go(\Pi_2).
	\end{aligned}
\end{equation}
In particular, we need to upper-bound each of the terms
$
	\norm{A_{k, \ell}}_2,\norm{B_{k, \ell}}_2,\norm{C_{k, \ell}}_2,\norm{D_{k, \ell}}_2.
$
We will analyze the different clusters for even and odd $p$ separately as the arguments are quite subtle.

\begin{lem}[Variance of conditional second moment control]\label{lem:abcd}
	Let $p\ge 3$ be fixed. For $(k,\ell)\in \sS_{p}$, we have
	\begin{align*}
		\norm{A_{k, \ell}}_2^2+\norm{B_{k, \ell}}_2^2+\norm{C_{k, \ell}}_2^2+\norm{D_{k, \ell}}_2^2 \le cN^{-p}\cdot \nu_{k, \ell}^4
	\end{align*}
	for some constant $c_{k,\ell,p}$ depending only on $k,\ell,p,\norm{J}_4$.
\end{lem}

We also need to control variance of conditional mixed moments. We proceed in a similar way. For $(k,\ell),(k',\ell')\in \sS_{p}$ we have
\begin{align*}
	\E\left(\gD W_{k,\ell}\cdot \gD W_{k',\ell'} \right) = 0.
\end{align*}
Thus,
\begin{align*}
	 & N_{p}\cdot \left(\E\left( \gD W_{k,\ell}\cdot \gD W_{k',\ell'} \ \biggr| \  (\go_{\mvi})_{\mvi\in\sE_{N,p}}\right) - \E\left( \gD W_{k,\ell}\cdot \gD W_{k',\ell'}\right) \right) \\
	 & = \sum_{\mvi\in\sE_{N,p}} (\xi_{\mvi} + 2\gb_{N,p}^2)\cdot \sum_{\Pi \in \cS_{k, \ell}^{-\mvi},\,\Pi' \in \cS_{k', \ell'}^{-\mvi'}} \go(\Pi)\go(\Pi')
	=C_{k, \ell,k',\ell'}+2\gb_{N,p}^2\cdot D_{k, \ell,k',\ell'}
\end{align*}
where
\begin{align}\label{mixcd}
	\begin{split}
		C_{k, \ell,k',\ell'} &:= \sum_{\mvi} \xi_{\mvi} \sum_{\Pi \in \cS_{k, \ell}^{-\mvi},\,\Pi' \in \cS_{k', \ell'}^{-\mvi'}} \go(\Pi)\go(\Pi')
		\text{ and }\\
		D_{k, \ell,k',\ell'} &:=\sum_{\mvi}\sum_{\Pi \in \cS_{k, \ell}^{-\mvi},\,\Pi' \in \cS_{k', \ell'}^{-\mvi'}} \go(\Pi)\go(\Pi').
	\end{split}
\end{align}
We have the following result.
\begin{lem}[Variance of conditional mixed moment control]\label{lem:mixcd}
	Let $p\ge 3$ be fixed. For $(k,\ell),(k',\ell')\in \sS_{p}$, we have
	\begin{align*}
		\norm{C_{k, \ell,k',\ell'}}_2^2+\norm{D_{k, \ell,k',\ell'}}_2^2 \le c N^{-p}\cdot \nu_{k, \ell}^2\nu_{k', \ell'}^2
	\end{align*}
	for some constant $c$ depending only on $k,\ell,k',\ell',p,\norm{J}_4$.
\end{lem}

Now, we bound the absolute third moment. Using the bound $|xyz|\le (|x|^{3}+|y|^{3}+|z|^3)/3$, it is enough to control the absolute third moment for each cluster separately.

\begin{lem}[Upper bound for absolute third moment]\label{lem:3m}
	We have
	\begin{align*}
		N_{p}\cdot \E\abs{\gD W_{k, \ell}}^{3} =
		N^{-p/2}\cdot \gb^{3k}\norm{J}_4^{3k} \cdot \nu_{k, \ell}^{3}\cdot O(1).
	\end{align*}
\end{lem}
\begin{proof}
	Recall that,
	\begin{align*}
		\gD W_{k, \ell} = (\go_{\mvI}'-\go_{\mvI})\cdot \sum_{\Pi \in \cS_{k, \ell}^{-\mvI}} \go(\Pi).
	\end{align*}
	Moreover, the hyperedge $\mvi$ does not appear in any subgraph in $\cS_{k, \ell}^{-\mvi}$. Thus,
	\begin{align*}
		\E\abs{\gD W_{k, \ell}}^{3} \le 8N_{p}^{-1}\sum_{\mvi\in \sE_{N,p}} \left(\E|\go_{\mvi}|^4\cdot  \E \abs{\sum_{\Pi \in \cS_{k, \ell}^{-\mvi}} \go(\Pi)}^4\right)^{3/4}.
	\end{align*}
	We have
	\begin{align*}
		\E|\go_{\mvi}|^4 \le \gb^4 \norm{J}_4^4
	\end{align*}
	and
	\begin{align*}
		\E \abs{\sum_{\Pi \in \cS_{k, \ell}^{-\mvi}} \go(\Pi)}^4
		 & = \sum_{\Pi_1,\ldots,\Pi_4 \in \cS_{k, \ell}^{-\mvi}} \quad \E \prod_{i=1}^4 \go(\Pi_{i}).
	\end{align*}
	Note that, the expectation $\E \prod_{i=1}^4 \go(\Pi_{i})$ is zero if any of the edges is present only once in $\cup_{i=1}^4\Pi_i$; otherwise it is bounded by $\E\abs{\go(\Pi)}^4\le \gb^{4(k-1)}\norm{J}_4^{4(k-1)}$ as each $\Pi\in \cS_{k, \ell}^{-\mvi}$ has $(k-1)$ hyperedges and each hyperedge has $p$ many vertices. Moreover,
	\begin{align*}
		 & \abs{\{(\Pi_{1},\Pi_2,\Pi_{3},\Pi_4)\in (\cS_{k, \ell}^{-\mvi})^4\mid \text{each edge appears at least twice in }\cup_{i=1}^4\Pi_i \}} \\
		 & =O(N^{2\cdot ((kp+\ell)/2-p)})=N^{-2p}\cdot O(\nu_{k, \ell}^4)
	\end{align*}
	as each element in $\sS_{k, \ell}$ can have at most $(kp+\ell)/2-p$ many vertices and  each vertex must repeat at least twice in the union.
	Thus, we finally have
	\begin{align*}
		\E\abs{\gD W_{k, \ell}}^{3} =N^{-3p/2}\cdot \gb^{3k}\norm{J}_4^{3k} \cdot \nu_{k, \ell}^{3}\cdot O(1).
	\end{align*}
	Using $N_{p}\le N^{p}/p!$, we now complete the proof.
\end{proof}

\subsection{Control for $(3,0)$ clusters with even $p$}
Note that $\xi_{\mvi}: = \go_{\mvi}^2 - \beta_{N,p}^2$ has mean zero and variance upper bounded by $\gb^{4}\norm{J}_{4}^{4}$. For the term $A_{3,0}$, we have
\begin{align}
	\norm{A_{3,0}}_2^2 = \sum_{\mvi, \mvi'}  \sum_{\substack{\Pi \in \cS_{3,0}^{-\mvi},\, \Pi' \in \cS_{3,0}^{-\mvi'}}} \E \left(\xi_{\mvi} \xi_{\mvi'}\cdot \go(\Pi)^2 \go(\Pi')^2\right).
\end{align}
Due to the property $\E \xi_{\mvi}=0$, there must be at least two copies of the hyperedges $\mvi,\mvi'$. We compute the order of contributions in each different cases.
\begin{itemize}
	\item If $\abs{\mvi \cap \mvi'} = p$, then $\mvi,\mvi'$ are the same hyperedge. The remaining $p/2$ vertices in each of $\Pi,\Pi'$ are arbitrary. Thus, the contribution is upper bounded by $\gb_{N,p}^{12}\cdot O(N^{p}\cdot N^{2\cdot p/2}) = O(N^{-p}\cdot \nu_{3,0}^4)$ as $\nu_{3,0}^2=\Theta(N^{3p/2})$.

	\item The second possibility is $\abs{\mvi \cap \mvi'} = p/2$, where $\mvi$ must also appear in $\Pi'$, and $\mvi'$ must appear in $\Pi$. Thus, all the vertices are determined by $\mvi,\mvi'$. The contribution is upper bounded by $\gb_{N,p}^{12}\cdot O(N^{3p/2}) = O(N^{-3p/2}\cdot \nu_{3,0}^4)$.

	\item The last case is $\abs{\mvi \cap \mvi'} = 0$. But this is impossible, since we still need $\mvi \in \Pi'$, at the same time, we also need $\Pi \cup \mvi$ to be 3-loop cluster. This contradicts the assumption that $\abs{\mvi \cap \mvi'} = 0$.
\end{itemize}
Thus, we get
\begin{align*}
	\norm{A_{3,0}}_2^2  = O(N^{-p}\cdot \nu_{3,0}^{4}).
\end{align*}

Next we analyze the second term $B_{3,0}$,
\begin{align*}
	\norm{B_{3,0}}_2^2 = \sum_{\mvi, \mvi'}  \sum_{\substack{\Pi \in \cS_{3,0}^{-\mvi},\ \Pi' \in \cS_{3,0}^{-\mvi'}}} \E \left((\go(\Pi)^2-\gb_{N,p}^{4}) \cdot (\go(\Pi')^2-\gb_{N,p}^{4}) \right).
\end{align*}
The expectation inside is zero if $\Pi\cap\Pi'=\emptyset$. Moreover,
each summand can be bounded by $\gb_{N,p}^{8}\norm{J}_{4}^{8}$. Using the fact that the number of pairs of 3-loops with non-empty intersection can at most be $N^{2p}$, we get the order of
\begin{align*}
	\norm{B_{3,0}}_2^2 =O(N^{-p}\cdot \nu_{3,0}^{4}).
\end{align*}
For $C_{3,0}$, we have
\begin{align*}
	\norm{C_{3,0}}_2^2 = \sum_{\mvi, \mvi'}  \sum_{\substack{\Pi\neq \gC \in \cS_{3,0}^{-\mvi},\ \Pi'\neq\gC' \in \cS_{3,0}^{-\mvi'}}} \E \left(\xi_{\mvi} \xi_{\mvi'}\cdot \go(\Pi)\go(\gC) \go(\Pi')\go(\gC')\right).
\end{align*}
Compared to the analysis in the case of $A_{3,0}$, now there is only one case left $\abs{\mvi \cap \mvi'} = p$ which makes nonzero contributions. Let us argue why the case $\abs{\mvi \cap \mvi'} = p/2$ is impossible. It needs $\mvi \in \Pi'\cap\gC'$, as each hyperedge must appear at least twice. Combining with the fact that $\mvi' \in \Pi'\cap\gC'$ and $\Pi'\cup \{\mvi'\}, \gC'\cup \{\mvi'\}$ are all 3-loops, this implies that $\Pi' = \gC'$, a contradiction with the $\Pi' \neq \gC'$ assumption. Therefore, the only possible case is $\mvi=\mvi'$, whose contribution gives the bound
$$
	\norm{C_{3,0}}_2^2 = \gb_{N,p}^{12}\norm{J}_{4}^{4}\cdot O(N^{2p}) = O(N^{-p}\cdot \nu_{3,0}^{4}).
$$
Finally, we analyze the term $D_{3,0}$,
\begin{align*}
	\norm{D_{3,0}}_2^2 = \sum_{\mvi, \mvi'}  \sum_{\substack{\Pi\neq \gC \in \cS_{3,0}^{-\mvi},\ \Pi'\neq\gC' \in \cS_{3,0}^{-\mvi'}}} \E \left(\go(\Pi)\go(\gC) \go(\Pi')\go(\gC')\right).
\end{align*}
This case is even more subtle. Compared to the analysis for $C_{3,0}$, now for $\mvi\neq\mvi'$ we do not need the condition that $\mvi \in \Pi',\gC'$ and $\mvi' \in \Pi,\gC$. We claim that the only possibilities are the following two cases:
\begin{itemize}
	\item Case  $\abs{\mvi \cap \mvi'} = 0$: This case is delicate. Note that $\Pi,\gC$ share $p$ vertices and in total $\Pi\cup\gC$ can have at most $2p$ vertices. Same holds for $\Pi',\gC'$. Now, each edge must appear at least twice in the union of $\Pi,\gC,\Pi',\gC'$. Thus, at most $2p$ vertices, those present in $\mvi,\mvi'$, will be involved in the union. The contribution in this case is of order $\gb_{N,p}^{8}\cdot O(N^{2p}) = O(N^{-p}\cdot \nu_{3,0}^{4})$.
	\item Case $\abs{\mvi \cap \mvi'} = p$: This case is similar to $C_{3,0}$ and the contribution is of order $\gb_{N,p}^{8}\cdot O(N^{2p}) = O(N^{-p}\cdot \nu_{3,0}^{4})$.
\end{itemize}

\subsection{Control for $(2,2)$ clusters with general $p$}
We have
\begin{align}
	\norm{A_{2,2}}_2^2 =\sum_{\mvi, \mvi'}  \sum_{\substack{\Pi \in \cS_{2,2}^{-\mvi},\, \Pi' \in \cS_{2,2}^{-\mvi'}}} \E \left(\xi_{\mvi} \xi_{\mvi'}\cdot \go(\Pi)^2 \go(\Pi')^2\right).
\end{align}
Since $\E \xi_{\mvi}=0$, if $\mvi \neq \mvi'$, we must have $\mvi'\in \Pi, \mvi\in \Pi'$. But now $\Pi,\Pi'$ are single hyperedges. Thus we must have $\Pi=\{\mvi'\}, \Pi'=\{\mvi\}$. Due to the condition that $\{\mvi\}\cup \Pi, \{\mvi'\} \cup \Pi'$ all form $(2,2)$ cluster, we know there are total $N^{p+1}$ different structures. The contribution of each structure is of order $\gb_{N,p}^{8}\norm{J}_{4}^{4}$, thus in total the contribution is $\gb_{N,p}^{8}\cdot O(N^{p+1})=O(N^{-p-1}\cdot \nu_{2,2}^{4})$ as $\nu_{2,2}^{2}=\Theta(N^{p+1})$.

On the other hand, if $\mvi = \mvi'$, then we do not have to put further restrictions on $\Pi,\Pi'$. Since $\{\mvi\}\cup \Pi, \{\mvi'\} \cup \Pi'$ form $(2,2)$ cluster, it is clear that there are $N^{p+2}$ different ways, and the total contribution is of order $\gb_{N,p}^{8}\norm{J}_{4}^{4}\cdot O(N^{p+2})=O(N^{-p}\cdot \nu_{2,2}^{4})$. Combining, we get
\begin{align*}
	\norm{A_{2,2}}_2^2 =O(N^{-p}\cdot \nu_{2,2}^{4}).
\end{align*}

Next we analyze the second term $B_{2,2}$,
\begin{align*}
	\norm{B_{2,2}}_2^2 = \sum_{\mvi, \mvi'}  \sum_{\substack{\Pi \in \cS_{2,2}^{-\mvi},\, \Pi' \in \cS_{2,2}^{-\mvi'}}} \E \left((\go(\Pi)^2-\gb_{N,p}^{4}) \cdot (\go(\Pi')^2-\gb_{N,p}^{4}) \right).
\end{align*}
As we did in the case of $B_{3,0}$, the order of the summand can be bounded by a constant. Note that, we must have  $\Pi=\Pi'$; otherwise by the independence property of different edges, the contribution is zero. There are at most $N^{p+2}$ different ways and the order is
\begin{align*}
	\norm{B_{2,2}}_2^2 =O(N^{-p}\cdot \nu_{2,2}^{4}).
\end{align*}

For $C_{2,2}$, we get
\begin{align*}
	\norm{C_{2,2}}_2^2 = \sum_{\mvi, \mvi'}  \sum_{\substack{\Pi\neq \gC \in \cS_{2,2}^{-\mvi},\ \Pi'\neq\gC' \in \cS_{2,2}^{-\mvi'}}} \E \left(\xi_{\mvi} \xi_{\mvi'}\cdot \go(\Pi)\go(\gC) \go(\Pi')\go(\gC')\right).
\end{align*}
Compared to the analysis in the case of $A_{3,0}$, now there is only one case left $\mvi = \mvi'$ that can make nonzero contributions. Let us argue why the case $\mvi\neq \mvi'$ is now impossible. It needs $\mvi \in \Pi', \gC'$, but now $\Pi',\gC'$ are both different single hyperedges. The only possible case is $\mvi =\mvi'$. In this case, one  needs to pair the hyperedges $\Pi,\gC$ with $\Pi',\gC'$ such that each hyperedge appear twice for avoiding zero contribution. Thus the contribution is of order
\begin{align*}
	\norm{C_{2,2}}_2^2 =
	\gb_{N,p}^{8}\norm{J}_{4}^{4}\cdot O(N^{p+2}) =O(N^{-p}\cdot \nu_{2,2}^{4}).
\end{align*}

Finally, we analyze the term $D_{2,2}$,
\begin{align*}
	\norm{D_{2,2}}_2^2 = \sum_{\mvi, \mvi'}  \sum_{\substack{\Pi\neq \gC \in \cS_{2,2}^{-\mvi},\ \Pi'\neq\gC' \in \cS_{2,2}^{-\mvi'}}} \E \left(\go(\Pi)\go(\gC) \go(\Pi')\go(\gC')\right).
\end{align*}
Compared to the case of $C_{2,2}$, this time we do not need $\mvi,\mvi'$ to appear twice. The set of vertices in $\mvi$ is a subset of the set of vertices in $\Pi\cup\gC$. Thus there are at most $N^{p+1}$ different ways, where two different hyper-edges appear exactly twice and they build a $(2,2)$ cluster. The total contribution is of order
\begin{align*}
	\norm{D_{2,2}}_2^2 = O(N^{p+1}\gb_{N,p}^{4})=O(N^{-p-1}\cdot \nu_{2,2}^{4}).
\end{align*}

\subsection{Control for $(1,p)$ clusters with general $p$}

This is the case of single hyperedge clusters. Moreover, the terms $B_{1,p},C_{1,p},D_{1,p}$ are zero. For term $A_{1,p}$, we have
\begin{align}
	\norm{A_{1,p}}_2^2 = \E \sum_{\mvi, \mvi'} \xi_{\mvi} \xi_{\mvi'} \le  N_{p} \gb_{N,p}^{4}\norm{J}_{4}^{4} = O(N^{-p}\cdot \nu_{1,p}^{4}).
\end{align}

\subsection{Control of mixed clusters product moments for even $p$}
Note that we have
\begin{align*}
	\norm{C_{k, \ell,k',\ell'}}_2^2 & := \sum_{\mvi,\mvi'} \sum_{\Pi \in \cS_{k, \ell}^{-\mvi},\gC \in \cS_{k, \ell}^{-\mvi'}, \Pi'\in \cS_{k', \ell'}^{-\mvi},\gC' \in \cS_{k', \ell'}^{-\mvi'}} \E\left( \xi_{\mvi} \xi_{\mvi'}\cdot \go(\Pi)\go(\gC)\go(\Pi')\go(\gC')\right)
	\text{ and } \\
	\norm{D_{k, \ell,k',\ell'}}_2^2 & :=\sum_{\mvi,\mvi'} \sum_{\Pi \in \cS_{k, \ell}^{-\mvi},\gC \in \cS_{k, \ell}^{-\mvi'}, \Pi'\in \cS_{k', \ell'}^{-\mvi},\gC' \in \cS_{k', \ell'}^{-\mvi'}} \E\left( \go(\Pi)\go(\gC)\go(\Pi')\go(\gC') \right).
\end{align*}
Thus, we bound the contribution for each mixed product moment.

\subsubsection{Clusters $(3,0),(2,2)$ analysis}
First, we compute the order of $\norm{C_{3,0,2,2}}_2^2$. When $\mvi\neq \mvi'$, it is easy to check that the contribution is zero by independence property. If $\abs{\mvi \cap \mvi'} = p$, \ie~the hyperedges are the same, we need $\Pi=\gC,\Pi'=\gC'$ to get non-zero contribution in $\norm{C_{3,0,2,2}}_2^2$. The total number of different ways is $N^{3p/2+1}$. The corresponding contribution is of order
\begin{align*}
	\norm{C_{3,0,2,2}}_2^2 = O(N^{3p/2+1}\gb_{N,p}^{10}\norm{J}_{4}^{4})=O(N^{-p}\cdot \nu_{3,0}^{2}\nu_{2,2}^{2}).
\end{align*}

For $\norm{D_{3,0,2,2}}_2^2$, due to different cluster structures of $\Pi,\gC$ and $\Pi',\gC'$, it is easy to see that the nonzero contribution is still from the case $\mvi=\mvi', \Pi=\gC, \Pi'=\gC'$. Therefore, the contribution is still of the order
\begin{align*}
	\norm{D_{3,0,2,2}}_2^2 = O(N^{3p/2+1}\gb_{N,p}^6)=O(N^{-p}\cdot \nu_{3,0}^{2}\nu_{2,2}^{2}).
\end{align*}

\subsubsection{Clusters $(3,0),(1,p)$ analysis}
First, to get a non-zero contribution in $\norm{C_{3,0,1,p}}_2^2$, one must have the hyperedges in $\Pi,\gC$ appear twice. This can only happen when $\mvi = \mvi', \Pi=\gC$. Therefore, the contribution is of the order $\gb_{N,p}^{8}\norm{J}_{4}^{4}\cdot N^{3p/2}=O(N^{-p}\cdot \nu_{3,0}^{2}\nu_{1,p}^{2})$.
For $\norm{D_{3,0,1,p}}_2^2$, for the same reason as in $\norm{C_{3,0,1,p}}_2^2$, the contribution of this term is of order  $\gb_{N,p}^{4}\cdot N^{3p/2}=O(N^{-p}\cdot \nu_{3,0}^{2}\nu_{1,p}^{2})$.

\subsubsection{Clusters $(2,2),(1,p)$ analysis}
To get a non-zero contribution in $\norm{C_{2,2,1,p}}_2^2$, one must have the hyperedges in $\Pi,\gC$ appear twice. This can only happen when $\mvi = \mvi', \Pi=\gC$. Therefore, the contribution is of order $\gb_{N,p}^{6}\norm{J}_{4}^{4}\cdot N^{p+1}=O(N^{-p}\cdot \nu_{2,2}^{2}\nu_{1,p}^{2})$.
For $\norm{D_{2,2,1,p}}_2^2$, for the same reason as in  $\norm{C_{2,2,1,p}}_2^2$, the contribution is of order $\gb_{N,p}^{2}\cdot N^{p+1}=O(N^{-p}\cdot \nu_{2,2}^{2}\nu_{1,p}^{2})$.

\subsection{Control for $(4,0)$ clusters with odd $p$}

The analysis of the odd $p$, $(4,0)$ cluster case is much more delicate than the even $p$, $(3,0)$ case. The major difficulty arises in the cluster $(4,0)$, as there can be many different ways to share vertices for the given 4 different hyperedges. This is in sharp contrast to the even $p$ case, where different hyperedges can only share $p/2$ vertices between two hyperedges. This can be seen in some very simple cases, for example for $p=5$, shown in Figure~\ref{fig:min-struct-odd}. In general, we introduce the following equations, whose solutions characterize this sharing mechanism uniquely.
\begin{align}\label{eq:share}
	\sI_{p}=\{(x,y,z)\mid x + y + z = p, x\ge y \ge z \ge 0, y> 0\}.
\end{align}

Note that in the above equation~\eqref{eq:share}, only $z$ can be 0. Essentially the numbers $x,y,z$ correspond to how many vertices a hyperedge can share with the remaining three hyperedges. It is easy to check that there exists a bijective map between the solutions in $\sI_{p}$ and the vertex-sharing structure.

Due to the different sharing structures illustrated, we further decompose the contribution from cluster $(4,0)$ into the following random variables. We define
\begin{align*}
	\cS_{4,\gd}:= \{ \gC\subseteq \sE_{N,p} 
    & \mid \abs{\gC}=4, \abs{\partial\gC}=0, \text{ for any hyperedge } \mvi\in\gC,\text{ number of common } \\
	& \qquad  \text{vertices between  $\mvi$ and the other three hyperedges in $\gC$ is } \\
	& \qquad  \text{given by the vector $(x,y,z)$ in decreasing order} \}
\end{align*}
and
\begin{align}
	W_{4,\gd}:=\sum_{\gC\in \cS_{4,\gd}}\go(\gC)
\end{align}
for $\gd \in \sI_{p}$. Here, we abuse the notation a bit to replace $(4,0,\gd)$ by $(4,\gd)$, for simplicity. Note that, we have
\begin{align*}
	W_{4,0}=\sum_{\gd\in \sI_p} W_{4,\gd}
\end{align*}
and for each $\gd\in\sI_{p}$, $\var(W_{4,\gd})=\Theta(\var(W_{4,0}))=\Theta(N^{2p})$. Moreover, for different $\gd,\gd'\in\sS_{p}$, the random variables $W_{4,\gd},W_{4,\gd'}$ are uncorrelated. See Figure~\ref{fig:min-struct-odd} for different intersection profiles $(2,2,1),(3,2,0),(4,1,0)$ for $p=5$.

One can follow a similar calculation as in Lemma~\ref{lem:lin} check linearity condition for each $W_{4,\gd}$, \ie
\begin{align*}
	\E\left( \gD W_{4,\gd}\mid (\go_{\mvi})_{\mvi\in \sE_{N,p}} \right) = -\gl_{4,0}W_{4,\gd}
\end{align*}
for all $\gd\in\sI_{p}$ and prove joint CLT for the vector $(W_{4,\gd})_{\gd\in\sS_p}$, which is not needed for our purpose. To bound the variance of the conditional second moment for $W_{4,0}$, it is enough to bound the same for each $W_{4,\gd}$ and the mixed product moment for different $\gd,\gd'\in \sI_{p}$. We define
\begin{align*}
	\cS_{4,\gd}^{-\mvi}:=\{\gC\setminus\{\mvi\}\mid \mvi\in \gC\in \cS_{4,\gd}\}
\end{align*}
and $A_{4,\gd},B_{4,\gd},C_{4,\gd},D_{4,\gd},A_{4,\gd,\gd'},B_{4,\gd,\gd'},C_{4,\gd,\gd'},D_{4,\gd,\gd'}$ in a similar way to equation~\eqref{abcd} and~\eqref{mixcd} for $\gd,\gd'\in\sI_{p}$, by replacing $\sS_{k,\ell}^{-\mvi}$ by $\sS_{4,\gd}^{-\mvi}$, appropriately.

\subsubsection{Control for $(4,0)$ clusters with intersection profile $\gd$}

We start analyzing $A_{4,\gd}$ for $\gd=(x,y,z), x\ge y\ge z\ge 0$. Even if there are many different solutions to~\eqref{eq:share} for a given odd $p$, it turns out that the analysis can be split into two disjoint subcases: $z=0$ and $z>0$.

Recall that,
\begin{align}
	\norm{A_{4,\gd}}_2^2 = \sum_{\mvi, \mvi'}  \sum_{\substack{\Pi \in \cS_{4,\gd}^{-\mvi},\ \Pi' \in \cS_{4,\gd}^{-\mvi'}}} \E \left(\xi_{\mvi} \xi_{\mvi'}\cdot \go(\Pi)^2 \go(\Pi')^2\right).
\end{align}
Note that $\Pi$ is another copy of $\Pi'$, thus they have the same structures (\ie~same $x,y,z$ values), though the vertices in $\Pi$ and $\Pi'$ might be different.
\begin{itemize}
	\item 	If $\mvi = \mvi'$, the only restriction is $\{\mvi\} \cup \Pi$ and $\{\mvi\} \cup \Pi'$ both form a  $(4,0)$ cluster with the fixed $(x,y,z)$ sharing profile. We give an example structure with $p=5$ in Figure~\ref{fig:ex-odd1a} for $(4,1,0)$. Each size $4$ hyper-loop has $2p$ many vertices and with $\mvi=\mvi'$, the other $4$ hyper-loop has at most $p$ extra vertices. Therefore, for each $\gd=(x,y,0)\in \sI_{p}$, the total number of different ways to form such structures is of the order $N^{3p}$, and the total contribution is $O(\gb_{N,p}^{16}\norm{J}_{4}^{4}\cdot N^{3p}) = N^{-p}\cdot O(\nu_{4,0}^4)$. The other nonzero contribution case is $\Pi=\Pi'$, but the total contribution is in subleading order.
	      \begin{figure}[htbp]
		      \centering
		      \includegraphics[width=5cm]{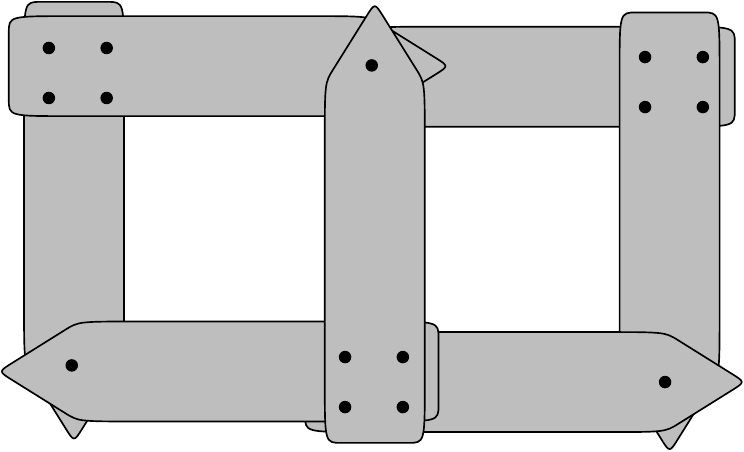}
		      \caption{Example structure for $p=5$ with $x=4,y=1,z=0$.}
		      \label{fig:ex-odd1a}
	      \end{figure}

	\item If $\mvi \neq \mvi'$, due to the fact that $\E \xi_{\mvi}=0$, we have $\mvi\in \Pi'$ and $\mvi'\in\Pi$.  Thus, $s=\abs{\mvi \cap \mvi'} \in \{x,y,z\}$ from the intersection profile restriction. In particular, $\{\mvi\}\cup\Pi, \{\mvi'\}\cup\Pi'$ share at least $2p-s$ many vertices. Total number of vertices present in $\{\mvi\}\cup\Pi\cup\{\mvi'\}\cup\Pi'$ is at most $(2p-s)+2s=2p+s$. Thus the contribution is bounded by $O(N^{2p+s}\gb_{N,p}^{16}\norm{J}_{4}^{8})=O(N^{-p}\cdot \nu_{4,0}^{4})$.
\end{itemize}

Next we analyze the second term $B_{4,\gd}$,
\begin{align*}
	\norm{B_{4,\gd}}_2^2 = \sum_{\mvi, \mvi'}  \sum_{\substack{\Pi \in \cS_{4,\gd}^{-\mvi},\, \Pi' \in \cS_{4,\gd}^{-\mvi'}}} \E \left((\go(\Pi)^2-\gb_{N,p}^{6}) \cdot (\go(\Pi')^2-\gb_{N,p}^{6}) \right).
\end{align*}
Each term in the summand is bounded.
Note that, if $\Pi,\Pi'$ share no edge, then the contribution is 0 due to the independence property. Thus the leading term corresponds to the structures $\Pi,\Pi'$ sharing one $p$-hyperedge. Note that $\{\mvi\}\cup\Pi,\{\mvi'\}\cup\Pi'$ are both $4$ hyperloops with having $2p$ vertices. Thus, sharing one hyperedge corresponds to the structure $\{\mvi,\mvi'\}\cup \Pi\cup\Pi'$ having at most $3p$ many vertices. Therefore, the leading order of $\norm{B_{4,\gd}}_2^2 = O(N^{3p}) = O(N^{-p}\cdot \nu_{4,0}^{4})$.

For $C_{4,\gd}$, we have
\begin{align*}
	\norm{C_{4,\gd}}_2^2 = \sum_{\mvi, \mvi'}  \sum_{\substack{\Pi\neq \gC \in \cS_{4,\gd}^{-\mvi},\ \Pi'\neq\gC' \in \cS_{4,\gd}^{-\mvi'}}} \E \left(\xi_{\mvi} \xi_{\mvi'}\cdot \go(\Pi)\go(\gC) \go(\Pi')\go(\gC')\right).
\end{align*}
The analysis of this part is similar to $A_{4,\gd}$. If $\mvi=\mvi'$, all hyperedges in $\Pi\cup\Pi'\cup\gC\cup\gC'$ must appear at least twice. A vertex counting argument along with the $4$ hyperloop restriction, shows that in total there can be $3p$ many vertices involved giving a contribution of $N^{3p}=O(N^{-p}\cdot \nu_{4,0}^{4})$. For $\mvi\neq \mvi'$, the same argument holds as to get non-zero contribution we must have $\mvi\in \Pi',\gC'$ and $\mvi'\in \Pi,\gC$.

Finally, we analyze the term $D_{4,\gd}$,
\begin{align*}
	\norm{D_{4,\gd}}_2^2 = \sum_{\mvi, \mvi'}  \sum_{\substack{\Pi\neq \gC \in \cS_{4,\gd}^{-\mvi},\ \Pi'\neq\gC' \in \cS_{4,\gd}^{-\mvi'}}} \E \left(\go(\Pi)\go(\gC) \go(\Pi')\go(\gC')\right).
\end{align*}

Compared to the analysis for $C_{4,\gd}$, now the restrictions $\mvi \in \Pi',\gC'$ and $\mvi' \in \Pi,\gC$ are no longer needed. If $\mvi = \mvi'$, even number of edge appearances and vertex counting argument along with $4$ hypercycle structure shows that the contribution is of order $N^{3p}=O(N^{-p}\cdot \nu_{4,0}^{4})$. In general, with $s=\abs{\mvi\cap\mvi'}$, we know that each vertex in $\mvi\cap\mvi'$ appears in $\Pi\cup\gC\cup\Pi'\cup\gC'$ at least four times (because of the hypercycle condition), there are $s$ many such vertices. Each vertex in $\mvi\triangle\mvi'$ appears in $\Pi\cup\gC\cup\Pi'\cup\gC'$ at least two times, and there are $2(p-s)$ many such vertices. All other vertices must appear at least four times (because of an even number of hyperedge appearance and hypercycle conditions), and there are, say, $m$ such vertices. Then $4s+4(p-s) + 4m\le 4\cdot 3p$ or $m\le 2p$. Thus, the total number of vertices is $s+2(p-s)+m=2p+m-s\le 4p-s$. However, for $s<p$ this argument is not optimal as we did not use the every hyperedge appearing at least twice condition.
We separate the analysis according as $z=0$ or $z \neq 0$.

For $z=0$,
\begin{itemize}
	\item If $\mvi = \mvi'$, the contribution is of order $N^{3p}$.
	\item If $\abs{\mvi \cap \mvi'}=0$, the contribution is of order $N^{2p}$.
	\item If $\abs{\mvi \cap \mvi'}=y$, the contribution is of order $N^{2p+y}$.
	\item If $\abs{\mvi \cap \mvi'}=x$, the contribution is of order $N^{2p+x}$.
\end{itemize}

For $z \neq 0$,
\begin{itemize}
	\item If $\mvi = \mvi'$, the contribution is of order $N^{3p}$.
	\item If $\abs{\mvi \cap \mvi'}=0$, the contribution is zero. Since a given hyperedge will intersect all the other 3 hyperedges by $z>0$, then $\mvi' \notin \Pi,\gC$ and $\mvi \notin \Pi',\gC'$. This further implies that $\mvi \cup \Pi$ and $\mvi' \cup \Pi'$ will form different $(4,0)$ structure. Due to the fact that $\Pi\neq \gC$ and $\Pi' \neq \gC'$, it can be seen that the total contribution is zero.
	\item If $\abs{\mvi \cap \mvi'}=z$, the contribution is of order $N^{2p+z}$. The analysis and leading structure are similar to the case $C_{4,\gd}$.
	\item If $\abs{\mvi \cap \mvi'}=y$, the contribution is of order $N^{2p+y}$.
	\item If $\abs{\mvi \cap \mvi'}=x$, the contribution is of order $N^{2p+x}$.
\end{itemize}

\subsection{Control for $(3,1)$ clusters with odd $p$}
Now we analyze the final cluster $(3,1)$. This cluster is simpler than the $(4,0)$ cluster since now the structures are topologically unique. For the term $A_{3,1}$, we have
\begin{align}
	\norm{A_{3,1}}_2^2 = \sum_{\mvi, \mvi'}  \sum_{\substack{\Pi \in \cS_{3,1}^{-\mvi},\ \Pi' \in \cS_{3,1}^{-\mvi'}}} \E \left(\xi_{\mvi} \xi_{\mvi'}\cdot \go(\Pi)^2 \go(\Pi')^2\right).
\end{align}

We need to discuss $\mvi = \mvi'$ and $\mvi'\neq \mvi$ separately. If $\mvi = \mvi'$, we need to make sure that $\mvi \cup \Pi$ and $\mvi' \cup \Pi'$ form a structure in $(3,1)$ cluster and no further restriction is needed. It is clear that the leading structure should be as in Figure~\ref{fig:ex-odd2a}. There are total $N^{p+2\cdot(p+1)/2}=N^{2p+1}$ different ways to form such structure, and the total contribution is of order $O(N^{2p+1})=O(N^{-p}\cdot \nu_{3,1}^{4})$ as $\nu_{3,1}^{2}=O(N^{(3p+1)/2})$.

\begin{figure}
	\centering
	\includegraphics[width=4cm]{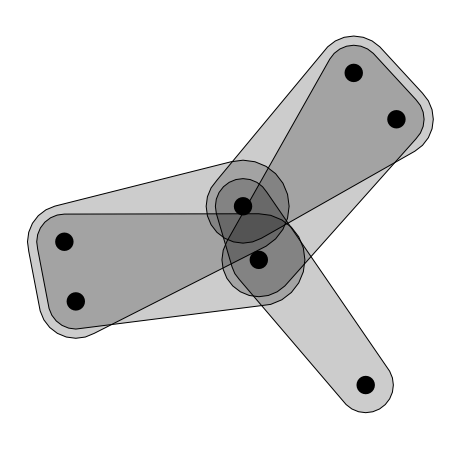}
	\caption{Example structure for (3,1) cluster $A_{3,1}$.}
	\label{fig:ex-odd2a}
\end{figure}

If $\mvi \neq \mvi'$, there are 3 possible ways to share the vertices. Due to the special structure of  $(3,1)$ cluster, $\abs{\mvi \cap \mvi'} \in \{0, (p-1)/2,(p+1)/2 \}$. Notice that by the mean zero property of $\xi_{\mvi}$, we need $\mvi \in \Pi', \mvi'\in \Pi$ restrictions.
\begin{itemize}
	\item If $\abs{\mvi \cap \mvi'} = 0$, it is impossible to form such structures. Since $\mvi' \in \Pi$ and $\mvi \cup \Pi$ forming $(3,1)$ cluster contradicts to $\abs{\mvi \cap \mvi'}=0$.
	\item If $\abs{\mvi \cap \mvi'} = (p-1)/2$, there are total $N^{(3p+1)/2}$ different ways to form the structure. The reason is that in $(3,1)$ cluster, there are in total 3 edges, $\mvi$ and $\mvi'$ already determine the whole structure. Therefore, the total contribution is of order $N^{(3p+1)/2}=O(N^{-3p/2}\cdot\nu_{3,1}^{4})$.
	\item The analysis for $\abs{\mvi \cap \mvi'} = (p+1)/2$ is similar to the previous case, with possibly one extra vertex. The total contribution is of order $N^{1+(3p+1)/2}=O(N^{-p}\cdot\nu_{3,1}^{4})$.
\end{itemize}

Next we analyze the second term $B_{3,1}$,
\begin{align*}
	\norm{B_{3,1}}_2^2 = \sum_{\mvi, \mvi'}  \sum_{\substack{\Pi \in \cS_{3,1}^{-\mvi},\ \Pi' \in \cS_{3,1}^{-\mvi'}}} \E \left((\go(\Pi)^2-\gb_{N,p}^{6}) \cdot (\go(\Pi')^2-\gb_{N,p}^{6}) \right).
\end{align*}
Each term in the summand is bounded by a constant. Note that $\Pi,\Pi'$ must share at least one hyperedge. Now that we do not need to place further restriction on $\mvi,\mvi'$, the leading structure will involve at most $(2p+1)$ vertices. Therefore, the total contribution is of order $O(N^{2p+1})=O(N^{-p}\cdot\nu_{3,1}^{4})$.

For $C_{3,1}$, we have
\begin{align*}
	\norm{C_{3,1}}_2^2 = \sum_{\mvi, \mvi'}  \sum_{\substack{\Pi\neq \gC \in \cS_{3,1}^{-\mvi},\ \Pi'\neq\gC' \in \cS_{3,1}^{-\mvi'}}} \E \left(\xi_{\mvi} \xi_{\mvi'}\cdot \go(\Pi)\go(\gC) \go(\Pi')\go(\gC')\right).
\end{align*}
Compared to term $A_{3,1}$, now one has to further pair the edges in $\mvi \cup \Pi$ and $\mvi \cup \gC$ with $\mvi' \cup \Pi'$ and $\mvi' \cup \gC'$ to obtain non-zero contributions. We start the simpler case, $\mvi = \mvi'$, then there is no need to worry about the zero mean issues of $\xi_{\mvi}$. It is easy to see that the leading structure is shown in Figure~\ref{fig:ex-odd2a}. One  need to pair $\Pi$ with $\Pi'$ and $\gC$ with $\gC'$ (or the other way). The total number of different ways is $N^{2p+1}$. The contribution is of order $O(N^{-p}\cdot\nu_{3,1}^{4})$.

Next, we claim that if $\mvi\neq \mvi'$, the contributions are all 0, \ie~it is impossible to construct structures with nonzero contributions. First, if $\abs{\mvi \cap \mvi'}=0$, then due to the constraints $\mvi\cup \Pi$ forms $(3,1)$ cluster  and $\mvi' \in \Pi$ , we know  that $\abs{\mvi \cap \mvi'} \neq 0$. Second, if $\abs{\mvi \cap \mvi'} = (p-1)/2$, then $\mvi \cup \mvi'$ already determine the cluster $(3,1)$. It means $\Pi,\Pi'$ have been specified. To avoid zero contribution, we have to pair $\Pi,\gC$ with $\Pi',\gC'$. However, due to the condition $\Pi\neq \gC, \Pi'\neq \gC'$, this is impossible. For a similar reason, we can argue that the case $\abs{\mvi \cap \mvi'}=(p-1)/2$ is also impossible. Therefore, the overall contribution of $\norm{C_{3,1}}_2^2$ is of order $O(N^{2p+1}) = O(N^{-p}\cdot \nu_{3,1}^{4})$.

Finally, we analyze the term $D_{3,1}$,
\begin{align*}
	\norm{D_{3,1}}_2^2 = \sum_{\mvi, \mvi'}  \sum_{\substack{\Pi\neq \gC \in \cS_{3,1}^{-\mvi},\ \Pi'\neq\gC' \in \cS_{3,1}^{-\mvi'}}} \E \left(\go(\Pi)\go(\gC) \go(\Pi')\go(\gC')\right).
\end{align*}
Compared to the analysis for $C_{3,1}$, now we do not need the condition that $\mvi \in \Pi',\gC'$ and $\mvi' \in \Pi,\gC$.
\begin{itemize}
	\item If $\mvi = \mvi'$, the total contribution is of order $N^{2p+1}$.
	\item If $\abs{\mvi \cap \mvi'} = 0$, the total contribution is of order $N^{2p+2}$.
	\item If $\abs{\mvi \cap \mvi'} = (p-1)$, the leading structure is same as the case $\abs{\mvi \cap \mvi'} = 0$ by relabeling the edges $\mvi,\mvi'$. Thus the total contribution is of order $N^{2p+2}$.
	\item If $\abs{\mvi \cap \mvi'} = (p+1)/2, (p-1)/2$, the contribution is zero, since there is a total of 3 edges in the cluster.
\end{itemize}

\subsection{Control of mixed clusters product moments for odd $p$}
In this subsection, we will bound the mixed products among different clusters. Since the first cluster $(4,0)$ contains many different substructures, the overall mixed products are more involved. We control the mixed products of substructures inside $(4,0)$ and the rest mixed products separately. The following lemma first gives the results for the substructures in $(4,0)$.

\begin{lem}
	For all $\gd=(x,y,z),\gd'= (x',y',z'), \gd\neq \gd'$, we have
	\begin{align}
	\norm{C_{4,\gd,\gd'}}_2^2 + \norm{D_{4,\gd,\gd'}}_2^2 \lesssim N^{-p}\cdot \nu_{4,0}^4.
	\end{align}
\end{lem}
\begin{proof}
We carry out the analysis in the following separate cases.
	\begin{itemize}
		\item $z=0,z'=0$,
		\item $z=0, z'\neq 0$, note that the counterpart $z\neq 0, z'=0$ is of same nature, we omit the details.
		\item $z \neq 0, z' \neq 0 $.
	\end{itemize}

Note that we have
\begin{align*}
	\norm{C_{4,\gd,\gd'}}_2^2 
        &=
         \sum_{\mvi, \mvi'}  \sum_{\substack{\Pi \in \cS_{4,\gd}^{-\mvi},\gC \in \cS_{4,\gd}^{-\mvi'},\Pi'\in \cS_{4,\gd'}^{-\mvi},\gC' \in \cS_{4,\gd'}^{-\mvi'}}} \E \left(\xi_{\mvi} \xi_{\mvi'}\cdot \go(\Pi)\go(\gC) \go(\Pi')\go(\gC')\right)
	\text{ and } \\
	\norm{D_{4,\gd,\gd'}}_2^2 & :=
     \sum_{\mvi, \mvi'}  \sum_{\substack{\Pi \in \cS_{4,\gd}^{-\mvi},\gC \in \cS_{4,\gd}^{-\mvi'},\Pi'\in \cS_{4,\gd'}^{-\mvi},\gC' \in \cS_{4,\gd'}^{-\mvi'}}} 
     \E\left( \go(\Pi)\go(\gC)\go(\Pi')\go(\gC') \right).
\end{align*}

	For all the cases of $(z,z')$, if $\mvi = \mvi'$, the leading structure in $C_{4,\gd,\gd'}$ is where one needs to pair the edges that appeared in $\Pi,\gC,\Pi',\gC'$. The total contribution is of order $N^{3p}$. If $\mvi \neq \mvi'$, a further restriction $\mvi \in \Pi',\gC', \mvi' \in \Pi,\gC$ is needed due to the mean zero property of $\xi_{\mvi},\xi_{\mvi'}$. For convenience, we argue that the leading contribution is of the order $O(N^{3p})$. For all possible cases of $(z,z')$, we need to argue that the structures containing three non-intersecting hyperedges can not have non-zero contributions. This is clearly true for the case $z\neq 0, z'\neq0$, where a given hyperedge will intersect with all the other three edges in the $(4,0)$ cluster. The analysis for $D_{4,\gd,\gd'}$ is similar.
\end{proof}

The next result is about all the rest mixed products.	Note that the mixed product analysis between cluster $(2,2)$ and $(1,p)$ works no matter what $p$ is even or odd. The details are presented in the last section. The analysis of the remaining mixed products is given below. Recall that
\begin{align*}
	\norm{C_{k, \ell,k',\ell'}}_2^2 & := \sum_{\mvi,\mvi'} \sum_{\Pi \in \cS_{k, \ell}^{-\mvi},\gC \in \cS_{k, \ell}^{-\mvi'}, \Pi'\in \cS_{k', \ell'}^{-\mvi},\gC' \in \cS_{k', \ell'}^{-\mvi'}} \E\left( \xi_{\mvi} \xi_{\mvi'}\cdot \go(\Pi)\go(\gC)\go(\Pi')\go(\gC')\right)
	\text{ and } \\
	\norm{D_{k, \ell,k',\ell'}}_2^2 & :=\sum_{\mvi,\mvi'} \sum_{\Pi \in \cS_{k, \ell}^{-\mvi},\gC \in \cS_{k, \ell}^{-\mvi'}, \Pi'\in \cS_{k', \ell'}^{-\mvi},\gC' \in \cS_{k', \ell'}^{-\mvi'}} \E\left( \go(\Pi)\go(\gC)\go(\Pi')\go(\gC') \right).
\end{align*}

\begin{lem}
	We have
	\begin{align*}
		\norm{C_{3,1,1,p}}_2^2+\norm{D_{3,1,1,p}}_2^2 \lesssim N^{-p}\cdot \nu_{3,1}^2\nu_{1,p}^2.
	\end{align*}
\end{lem}
\begin{proof}
	Notice that to avoid zero contribution in $\norm{C_{3,1,1,p}}_2^2$, we need $\Pi=\gC, \mvi = \mvi'$, where the total contribution is $N^{(3p+1)/2}=N^{-p}\cdot \nu_{3,1}^2\nu_{1,p}^2 $.
	In a similar fashion one can see that contribution for $\norm{D_{3,1,1,p}}_2^2$ is also of order $N^{(3p+1)/2}$. The restriction $\Pi=\gC$ is strong enough to make $\mvi = \mvi'$.
\end{proof}

\begin{lem}
	We have
	\begin{align*}
		\norm{C_{3,1,2,2}}_2^2+\norm{D_{3,1,2,2}}_2^2 \lesssim N^{-p}\cdot \nu_{3,1}^2\nu_{2,2}^2.
	\end{align*}
\end{lem}
\begin{proof}
	For $C_{3,1,2,2}$,
	if $\mvi = \mvi'$, we need $\Pi  = \gC, \Pi' = \gC'$ to avoid zero contribution; The total number of different ways is $N^{(3p+1)/2+1}=N^{-p}\cdot \nu_{3,1}^2\nu_{2,2}^2$. 
	When $\mvi \neq \mvi'$, then it needs $\mvi \in \Pi',\gC'$ and $\mvi' \in \Pi,\gC$ due to $\E \xi_{\mvi}=0$. This implies that $\abs{\mvi \cap \mvi'} = p-1$ by the $(2,2)$ cluster structure. On the other hand, given two hyperedges in the cluster $(3,1)$ can uniquely identify the structure, \ie~no new vertices needed.
 
	For $D_{3,1,2,2}$, it is clear that in the case $\mvi=\mvi'$, the leading contribution is of order $\cO(N^{(3p+3)/2})$.
	In the case of $\mvi \neq \mvi'$, now the restriction $\mvi \in \Pi',\gC'$ is removed. Basically, $\abs{\mvi \cap \mvi'} \in \{(p-1)/2,(p+1)/2, p-1 \}$, the leading contribution among those is $\abs{\mvi \cap \mvi'} = p-1$, and the associated structure is same as $C_{3,1,2,2}$. Now we argue why the other two cases are impossible. Suppose that $\abs{\mvi \cap \mvi'} = (p-1)/2$; to avoid zero contribution, it needs to match the edges in $\mvi \cup \Pi$ and $\mvi' \cup \Pi'$. Thus we also have to match $\mvi \cup \gC$ and $\mvi' \cup \gC'$. Equivalently, we need to find one edge $\mvj$ such that $\abs{\mvj \cap \mvi }=\abs{\mvj \cap \mvi' } = p-1 $ due to the $(2,2)$ cluster structure. It further needs $(p-1 - (p-1)/2)/2$ vertices for each $\mvi,\mvi'$. But in this case, the contribution is of order $N^{(3p+1)/2}$.
\end{proof}

\begin{lem}
	We have
	\begin{align*}
		\norm{C_{3,1,4,0}}_2^2+\norm{D_{3,1,4,0}}_2^2 \lesssim N^{-p}\cdot \nu_{3,1}^2\nu_{4,0}^2.
	\end{align*}
\end{lem}
\begin{proof}
	For $C_{3,1,4,0}$, if $\mvi = \mvi'$, then one needs to match the edges in $\Pi$ and $\Pi'$; and $\gC$, and $\gC'$. The contribution is of order $N^{3p+1+p}$. If $\mvi \neq \mvi'$, then it needs $\mvi \in \Pi',\gC'$ and $\mvi' \in \Pi,\gC$. Since the special structure of cluster $(3,1)$, any given two edges can uniquely identify the structure. This implies that $\abs{\mvi \cap \mvi'} \in\{(p-1)/2, (p+1)/2 \}$. Since $\mvi,\mvi'$ also appears in the $(4,0)$ cluster, thus only $p-\abs{\mvi \cap \mvi'}<p$ new vertices needed to form the cluster $(4,0)$. Therefore the contribution, in this case, is smaller than $N^{3p+1+p}$.
	Next for $\norm{D_{3,1,4,0}}_2^2$,
	if $\mvi = \mvi'$, the contribution is same as in case $C_{3,1,4,0}$. 
    If $\mvi \neq \mvi'$, since we need to pair the edges in $\mvi \cup \gC$ and $\mvi' \cup \gC'$, it requires that $\abs{\mvi \cap \mvi'} = (p-1)/2, (p+1)/2$. 
    This is the same as the analysis in $C_{3,1,4,0}$; therefore the contribution is smaller than the leading order.
\end{proof}

\begin{lem}
	We have for $\gd=(x,y,z)$
	\begin{align*}
		\norm{C_{4,\gd,2,2}}_2^2+\norm{D_{4,\gd,2,2}}_2^2 \lesssim N^{-p}\cdot \nu_{4,0}^2\nu_{2,2}^2.
	\end{align*}
\end{lem}
\begin{proof}
	As usual, for $C_{4,\gd,2,2}$ if $\mvi = \mvi'$, one needs to match the hyperedges, which leads to the contribution of order $N^{2p+1}$. If $\mvi \neq \mvi'$, the special structure of $(2,2)$ cluster requires $\Pi' = \mvi', \gC' = \mvi$. To form the $(4,0)$ cluster, one needs $p-1$ new vertices, thus the total contribution is of order $N^{2p}$.
	For $D_{4,\gd,2,2}$, if $\mvi = \mvi'$, the contribution is $N^{2p+1}$. The difference lies in $\mvi \neq \mvi'$. If $\abs{\mvi \cap \mvi'}=0$, this can happen in the following 2 subcases. When $\mvi,\mvi'$ appears in two different $(4,0)$ clusters, it is clear that the contribution is zero in this case by the independence property. The other case is $\mvi,\mvi'$ appears in two identical (4,0) clusters, recall the special vertices sharing mechanism, this can only happen when $z=0$. In this case, to avoid zero contribution, we need $x=p-1,y=1$. Therefore, the contribution is of order $N^{2p}$.
\end{proof}

\begin{lem}
	We have for $\gd=(x,y,z)$
	\begin{align*}
		\norm{C_{4,\gd,1,p}}_2^2+\norm{D_{4,\gd,1,p}}_2^2 \lesssim N^{-p}\cdot \nu_{4,0}^2\nu_{1,p}^2.
	\end{align*}
\end{lem}
\begin{proof}
	The proof is nearly identical to the case for $(3,1)$ and $(1,p)$.
	For $C_{4,\gd,1,p}$
	since we need to pair two $(4,0)$ clusters, $\mvi \cup \Pi=\mvi' \cup \gC'$; the total contribution is clearly $N^{2p}$. The other bound is similar.
\end{proof}

\begin{proof}[Proof of Theorem~\ref{thm:stein-main}]
Now we collect everything above to prove Theorem~\ref{thm:stein-main}.
To apply Theorem~\ref{thm:rr-mvstein}, the linearity condition was checked in Lemma~\ref{lem:lin}. Moreover, the variance bound for the conditional second moment and the absolute third-moment control are established in Lemmas~\ref{lem:abcd} and~\ref{lem:mixcd}; and Lemma~\ref{lem:3m}, respectively. Combining them, we complete the proof of Theorem~\ref{thm:stein-main}.
\end{proof}

\section{Stein's Method for Mixed \texorpdfstring{$p$}{p}-Spin Models}\label{sec:stein-mix}
In this section, we discuss how to carry out the multivariate Stein's method computations in the mixed $p$-spin setting. The overall strategy is essentially the same as the pure $p$-spin case in Section~\ref{sec:stein}. The difference lies in the more involved leading cluster structures, which are identified in Section~\ref{sec:dominant} and summarized in Table~\ref{tab:transm}. 

First, we only need to consider the contributions from the identified dominating hypergraph clusters with structures in $\sS_{\text{mix}}$, which is defined as 
\begin{align}
\sS_{\text{mix}}:= \{(a_e,a_o,b): (a_e,a_o,b) \ \text{is from the last column of Table~\ref{tab:transm}}\}
\end{align}

Recall the definition of hypergraph clusters with structure $c=((a_e,p_e),(a_o,p_o),b)$, with $a_e,a_o,b \in \bN$,  

\begin{align*}
\cS_{a_e,a_o,b}=\{ \gC \subseteq \cup_{p\ge3} \sE_{N,p}| \abs{\gC}_{p_e}=a_e, \abs{\gC}_{p_o}=a_o, \abs{\partial \gC} = b\} 
\end{align*}
where $\abs{\gC}_{p_e}$ represent the number of $p_e$-hyperedges and $\abs{\partial \gC}$ are the number of odd-degree vertices in $\gC$. Clearly $\abs{\cS_{a_e,a_o,b}} = \Theta(N^{t_{a_e,a_o,b}})$ with $t_{a_e,a_o,b} := (a_ep_e+a_op_o+b)/2 \in \bN$. Similarly, we consider the cluster weights as random variables, 
\begin{align*}
W_{a_e,a_o,b}:= \sum_{\gC \in \cS_{a_e,a_o,b}} \go(\gC).
\end{align*}
Similarly notice that $\E W_{a_e,a_o,b}=0$ and the variance is $$\gn_{a_e,a_o,b}^2 :=\var(W_{a_e,a_o,b}) =(1+o(1))\cdot  \abs{\cS_{a_e,a_o,b}}\gb^{2(a_e+a_o)}\theta_{p_e}^{2a_e} \theta_{p_o}^{2a_o}. $$ 
The normalized version is defined as 
\begin{align*}
\widehat{W}_{a_e,a_o,b} := W_{a_e,a_o,b}/\gn_{a_e,a_o,b}.
\end{align*}
In general, we have the random vector and its rescaled version, 
\begin{align*}
\mvW:= (W_{a_e,a_o,b})_{(a_e,a_o,b) \in \sS_{\text{mix}}} \ \text{and} \ \widehat{\mvW}:= (\widehat{W}_{a_e,a_o,b})_{(a_e,a_o,b) \in \sS_{\text{mix}}}.
\end{align*}
We now describe how to construct the exchangeable pairs. Notice that now the dominant hypergraph cluster contains non-homogeneous $p$-hyperedges, but it only depends on the minimum effective even and odd $p$-spins, \ie~$p_e,p_o$. First, we need to select a hyperedge $\mvI$ uniformly at random from $\sE_{N,p_o}\cup \sE_{N,p_e}$  and replace the edge weight $\go_{\mvI}$ by an i.i.d.~copy $\go'_{\mvI}$. It is easy to check that the resulting random vector $\mvW'$ and the original random vector $\mvW$ create an exchangeable pair. Define $\gD \mvW = \mvW'-\mvW$. Similarly one can check the linearity conditional expectation as in Lemma~\ref{lem:lin}. 
\begin{align*}
\E (\gD \mvW\mid \mvW) = -\gL_{\text{mix}} \mvW, 
\end{align*}
where $\gL_{\text{mix}}:= \text{diag}(\gl_{a_e,a_o,b}, (a_e,a_o,b) \in \sS_{\text{mix}})$, and $\gl_{a_e,a_o,b} := \frac{a_e+a_o}{N_{p_e}+N_{p_o}}$. Recall the notation $N_{p}:=\binom{N}{p}$. 
We first state the analog of Lemma~\ref{lem:varlim} for the limiting variance of different cluster structures. 

\begin{lem}[Limiting variance for mixed $p$-spin] \label{lem:varlim-mix}
	For $a_e,a_o,\ge 0, b\ge 0$, we have
	\begin{align*}
		& u_{a_e,a_o,b}(p_e,p_o)^{2} \\ & :=\lim_{N\to\infty}\abs{\cS_{a_e,a_o,b}} \cdot N^{-(a_ep_e +a_op_o+b)/2} \\
  & = \frac{1}{a_e! \cdot a_o! \cdot b! \cdot (p_e!)^{a_e/2} \cdot (p_o!)^{a_o/2}}  \E ( H_{a_e}(H_{p_e}(\eta)/\sqrt{p_e!}) \cdot H_{a_o}(H_{p_o}(\eta)/\sqrt{p_o!})\cdot H_b(\eta) ),
	\end{align*}
	where $\eta\sim\N(0,1)$. 
\end{lem}

\begin{rem}
The proof of Lemma~\ref{lem:varlim-mix} is essentially the same as Lemma~\ref{lem:varlim}. The key idea is to express the combinatorial number as an expectation of the product of spin variables, then connect this with the complete exponential Bell polynomial. In the mixed $p$-spin case, one has to deal with several different $p$-hyperedges depending on the cluster structure. 
\end{rem}

Now we state the analog of Theorem~\ref{thm:stein-main} in mixed $p$-spin setting. 

\begin{thm}\label{thm:stein-main-mix}
For the $\abs{\sS_{\text{mix}}}$-dimensional random vector $\mvW= (W_{a_e, a_o,b})_{(a_e,a_o,b)\in \sS_{\text{mix}}}$ with mean and variance structure $\gS$ specified above in the Lemma~\ref{lem:varlim-mix}. For any thrice differentiable function $f$, we have 
\begin{align*}
\abs{\E f(\mvW) - \E f(\gS^{1/2}\mvZ)} \le C \left( \abs{f}_2 N^{-p_m} + \abs{f}_3N^{-p_m/2}\right),
\end{align*}
where $C$ is some constant depending only on $\gb$ and the distribution of disorder $J_{12}$.
\end{thm}

The proof of this Theorem needs the variance control of the conditional second moment and the absolute third-moment control. The computation will be morally the same as in the pure $p$-spin setting in the last section. We will not repeat the computational details especially the coefficients in $\gL_{\text{mix}}$ have already been identified, which essentially controls the error rate in Theorem~\ref{thm:stein-main-mix}.

\section{Proof of Main Results}\label{sec:main-pf}
In this section, we put everything together to present the proof of main theorems for the pure $p$-spin and mixed $p$-spin models.

\begin{proof}[Proof of Theorem~\ref{thm:h0} and~\ref{thm:h1}]
By Proposition~\ref{prop:first-red} and~\ref{prop:second-red}, it suffices to study the limiting distributional behavior for $\hat{Z}_{N,4}$, which by the combinatorial interpretation corresponds to the hypergraph clusters' contribution with cluster size less than or equal to 4. By Proposition~\ref{prop:smallest-h0} and~\ref{prop:critical}, the leading clusters are identified as given in the last column of Table~\ref{tab:trans} in various regimes. For those identified leading clusters, by Theorem~\ref{thm:stein-main}, we have the distributional convergence to Gaussian with specified variance structure in Lemma~\ref{lem:varlim}. This completes the proofs for the central limit theorem of $\log \hat{Z}_N(\gb,h)$ for the weak external field $h$. 
\end{proof}

\begin{proof}[Proof of Corollary~\ref{cor:logz}]
By the expansion of $\log \bar{Z}_N(\gb,h)$ in~\eqref{eq:zbar-approx}, it is clear that the variance order is $\Theta(N^{-(p-2)})$. On the other hand, the largest variance in $\log \hat{Z}_N(\gb,h)$ is from the single hyperedge cluster. The variance order is $\Theta(N^{1-2\ga p})$, thus it's easy to deduce that for $\ga > \frac{p-1}{2p}$, the contribution is mainly from $\log \hat{Z}_N(\gb,h)$. While for $\frac14\le \ga <\frac{p-1}{2p}$, the contribution is mainly from the single hyperedge cluster, \ie~in $\log \hat{Z}_N(\gb,h)$. The critical regime for is $\ga = \frac{p-1}{2p}$, which contains contributions from both $\log \bar{Z}_N(\gb,h)$ and $\log \hat{Z}_N(\gb,h)$.
\end{proof}

\begin{proof}[Proof of Theorem~\ref{thm:mixed0} and~\ref{thm:mixedh}]
Similarly in the mixed $p$-spin setting, by Proposition~\ref{prop:first-red} and~\ref{prop:second-red}, it suffices to study the limiting distributional behavior for $\log \hat{Z}_{N,4}$, which by the combinatorial interpretation corresponds to the hypergraph clusters' contribution with cluster size less or equal to 4. However, a further roadblock in mixed $p$-spin setting is the large $p$-hyperedge, thanks to the Proposition~\ref{prop:p-reduction}, one can get over this issue and using Proposition~\ref{prop:smallest-h0} and~\ref{prop:critical}, the leading clusters are identified as given in the last column of Table~\ref{tab:transm} in various regimes. For those identified leading clusters, by Theorem~\ref{thm:stein-main-mix}, we have the distributional convergence to Gaussian with specified variance structure in Lemma~\ref{lem:varlim-mix}. This completes the proofs for the central limit theorem of $\log \hat{Z}_N(\gb,h)$ under the weak external fields $h$. 
\end{proof}

\begin{proof}[Proof of Corollary~\ref{thm:mixed-log}]
By the expansion of $\log \bar{Z}_N(\gb,h)$ in~\eqref{eq:zbar-approx-mix}, it is clear that the variance order is $\Theta(N^{-(p_m-2)})$, where $p_m = \min\{p\ge 3:\theta_p>0\}$ is the smallest effective $p$. On the other hand, the largest variance in $\log \hat{Z}_N(\gb,h)$ is from the single hyperedge cluster. The variance order is $\Theta(N^{1-2\ga p_m})$, thus it's easy to deduce that for $\ga > \frac{p_m-1}{2p_m}$, the contribution is mainly from $\log \hat{Z}_N(\gb,h)$. While for $\frac14\le \ga <\frac{p_m-1}{2p_m}$, the contribution is mainly from the single hyperedge cluster, \ie~in $\log \hat{Z}_N(\gb,h)$. The critical regime for is $\ga = \frac{p_m-1}{2p_m}$, which contains contributions from both $\log \bar{Z}_N(\gb,h)$ and $\log \hat{Z}_N(\gb,h)$.
\end{proof}

\section{Open Questions}\label{sec:open}

\subsection{Cluster expansion beyond the large hypergraph decaying regime}
Currently, this cluster expansion framework is restricted to the high temperature and weak external field regimes. The main reason is that the large hypergraph decaying property is only true in this regime. Otherwise, one has to deal with very large cluster structures. Counting those clusters and applying Stein's method to prove CLT does not work anymore. Therefore, extending this framework beyond the large graph decaying regime is of great interest. This relates to many other important open questions in mean field spin glasses, such as the Almeda-Thouless conjecture and the Gumble fluctuation of free energy at low temperatures. We further remark that analytical methods developed in~\cite{DW21} should give the fluctuation results in high temperature and strong external field ($\ga<\frac14$), however, the disorder has to be Gaussian and the results can possibly be proved up to some higher temperature than $\gb_f$.

\subsection{Classification of spin glass models}
The classification of spin glass models was addressed in~\cite{AB13}, where the authors asked whether 1-step replica symmetry breaking (RSB) mixed $p$-spin spherical models are pure-like. They studied the problem using a complexity-based approach. In~\cite{AZ19}, Auffinger and Zeng constructed some examples which are both 2-step RSB but one is pure-like, and the other one is a full mixture. The definitions of ``pure-like" and ``full-mixture'' are based on the complexity of the energy function. In this article, our results on the mixed Ising $p$-spin case can also be treated as a classification of the mixed $p$-spin models in terms of cluster structures. For simplicity, consider the $h=0$ case; we proved that if $p_e<p_o$, the mixed model behaves essentially like the pure $p_e$-spin model, while in the case $p_e \ge 2p_o$, the mixed model is more like pure $p_o$-spin model. Only in the case $p_o<p_e<2p_o$ the mixed $p$-spin model behaves really like a mixture. The question is whether this is a solid criterion to classify general mixed $p$-spin models. Indeed, this question is asked in the low-temperature regime, and our results answer the high-temperature case. Besides that, it is also observed that if the external field is very strong, then it seems that the system does not have a mixture regime. Specifically, if $p_e<p_o$, the model is like a pure $p_e$-spin case. If $p_e>p_o$, it is more like a pure $p_o$-spin model. It is also interesting to investigate classification under the effects of an external field.

\subsection{Cluster expansion in vector spin case}
Although extending the technique beyond the large graph decaying regime is fundamentally challenging, it is still hopeful to use this technique to obtain new results in other more complicated but not well-understood spin glass models. For example, in the vector spin glass model, the spins at each site take vector values and interact via inner products. The phase transition and fluctuation results are both not rigorously well-understood. There are a few results~\cites{Pan18pott,Pan18vec} concerning the limiting free energy using the Parisi formula. For example, Potts spin glass model, as a particular case of the vector spin model, has been studied a lot in physics community~\cites{Erzan83,DPR95,DD83b,DD83,CFR12,Go88,GKS85,MMP99}, where some more complicated phase transition is expected even in the high-temperature regime. However, mathematically rigorous proof for this has long been elusive. As indicated in this article, cluster expansion can be used to establish fluctuation results and also detect the sharp phase transition threshold. It would also be interesting to extend this technique to those models; this is currently under investigation by the authors.

\subsection{Beyond the distributional limit of free energy}
For simplicity, consider the zero external field case; the results in this article give a way to decompose the partition function as several simpler components. Recall the definition of the Gibbs measure
\begin{align*}
	dG_{N,\gb}(\mvgs)  = \exp(\gb H_N(\mvgs))\cdot  Z_N^{-1} d\mvgs.
\end{align*}
One can use the identity $\exp(\gs x) = \cosh(x) + \gs \sinh(x)$ to rewrite the Gibbs average as a mean with respect to the measure defined on the space of ``loops". The question is whether this can be used to study other aspects of the spin glass system. For example, the overlap as an order parameter carries much information about the model. It is essential to study the limiting behavior of the overlap. More generally, the question is whether cluster expansion can be used to study questions beyond the distributional behavior of the free energy. We further remark that the above questions in the low-temperature regime would be more interesting due to the replica symmetry-breaking phase.

\subsection{Understanding the threshold $\gb_f$}
The form of $\gb_f$ comes from an optimal second-moment control of the partition function in Lemma~\ref{lem:zhatvar}. As the numerical simulation in Figure~\ref{fig:betac} suggested, $\gb_f$ is not the usual static phase transition point. It is not clear whether this is due to technical difficulties or there is some unknown physical phenomenon that breaks down at $\gb_f$. Since it possesses such a simple and explicit form, it will be very interesting to investigate the physical nature of $\gb_f$. On the other hand, as shown in Figure~\ref{fig:betac}, $\gb_f \to \sqrt{2\ln 2}$ as $p \to \infty$, which is located between the two thresholds $\sqrt{\ln 2}$, $2\sqrt{\ln 2}$ discovered in~\cite{BKL02} for REM. A first step towards understanding $\gb_f$ would be investigating the differences for REM in regimes $\gb \in (\sqrt{\ln 2}, \sqrt{2 \ln 2})$ and $\gb \in (\sqrt{2\ln 2}, 2\sqrt{ \ln 2})$.

\bigskip\noindent
\textbf{Acknowledgements.} We are grateful to Wei-Kuo Chen, Mark Sellke, and Arnab Sen for valuable feedback and comments. We also warmly thank David Belius for suggesting~\cite{MR03} on the numerical simulation of static critical inverse temperature $\gb_c$ and discussions on the critical threshold $\gb_f$. Q.W. was supported by the Campus Research Board Grant RB23016 at the University of Illinois at Urbana-Champaign.

\bibliography{cluster.bib}
\end{document}